\documentclass[11pt]{article}%{amsart}
\usepackage{amsmath,amsfonts,amssymb,graphicx,amsthm, times}
\usepackage{bbm,latexsym,delarray,verbatim}
\usepackage{epsfig,color,ifthen}
\usepackage[hmargin=0.8in,height=8.6in]{geometry}

\definecolor{webgreen}{rgb}{0,.5,0}
\definecolor{webbrown}{rgb}{.8,0,0}
\definecolor{emphcolor}{rgb}{0.95,0.95,0.95}
\usepackage{hyperref}
\hypersetup{%
	colorlinks=true,
	linkcolor=webbrown,
	filecolor=webbrown,
	citecolor=webgreen,
	breaklinks=true}
\ifpdf \hypersetup{pdftex,
	pdfstartview=FitH, %%Fit, FitB, FitH
	bookmarksopen=true,
	bookmarksnumbered=true
} \else \hypersetup{dvips} \fi

\newcommand {\ud}{{\rm d}}

\newcommand {\I}[1]{\mathbbm{1}_{\{#1\}}}

\newcommand{\lev}{L\'{e}vy }

\newtheorem{theorem}{Theorem}[section]

\newtheorem{corollary}{Corollary}[section]
\newtheorem{remark}{Remark}[section]
\newtheorem{lemma}{Lemma}[section]

\newcommand {\R}{\mathbb{R}}

\newcommand {\p}{\mathbb{P}}

\newcommand {\E}{\mathbb{E}}

\newcommand{\diff}{{\rm d}}

\newcommand{\e}{\mathbb{E}}

\def\I{\infty}  \def\s{\sigma} 
\def\beq{\begin{eqnarray}} \def\eeq{\end{eqnarray}}
\def\bea{\begin{eqnarray*}}
\def\eea{\end{eqnarray*}}
\def\le{\left} \def\ri{\right} 
\def\k{\kappa}   \def\t{\tau} 
\def\fr{\frac} \def\th{\theta} \def\p{r}  \def\q{q} \def\la{\label}
\long\def\symbolfootnote[#1]#2{
\begingroup
\def\thefootnote{\fnsymbol{footnote}}\footnote[#1]{#2}
\endgroup}
 
\def\bc{\begin{cases}
  }
\def\ec{\end{cases}} 
\def\beT{\begin{theorem}
  }
\def\eeT{\end{theorem}}
\def\beR{\begin{remark}}
\def\eeR{\end{remark}}

\def\BEN{\begin{enumerate}}   
\def\EEN{\end{enumerate}} 
\begin{document}
\title{Spectrally negative L\'evy
processes  with Parisian reflection  below  and
classical reflection above}

\author{Florin Avram, Jos\'e-Luis P\'erez, and Kazutoshi Yamazaki }
\maketitle

\begin{abstract}

We consider a company that receives capital injections so as to avoid ruin.  Differently from the classical bail-out settings, where the underlying process is restricted to stay at or above zero, we study the case bail-out can only be made at independent Poisson observation times.
Namely, we study a version of the reflected  process that is pushed up to zero only on Poisson arrival times at which the process is below zero.
We also study the case with additional classical reflection above so as to model a company that pays dividends according to a barrier strategy.  Focusing on the spectrally negative \lev case, we compute, using the scale function, various fluctuation identities, including capital injections and dividends.
%
%
%Toward this end, this paper studies a version of the reflected spectrally negative \lev process.  Given a sequence of independent Poisson observation times, the process is pushed up to zero on these times if the process is below zero.   The case with additional reflection above is also studied.
%
%\red{The case reflected/refracted above is also studied. [do we include them?] add about capital injections} \blue{Kazu: I was thinking about it, and perhaps its pertinent to include it here for the control problem, if not we have to wait for Florin's other paper to be finished?} \red{JL: OK. I guess we can ask him.  But if we do not include, we do not see much point of considering the killed version upon upcrossing $b$?}\blue{Agreed, the I believe we should include it?}

\noindent \small{\textbf{Key words:}
	capital injections, dividends, scale functions, \lev processes, excursion theory.  \\
	\noindent  AMS 2010 Subject Classifications: 60G51, 91B30}\\
\end{abstract}

%\tableofcontents
%\input{int}

%\red{Most results here are for the process NOT reflected above... So we can change the title?}

\section{Introduction} \label{section_introduction}

In this paper, we revisit the study of risk processes where a  company is bailed out by capital injections.  In the classical setting, capital injections can be made at all times and instantaneously; the resulting process becomes the classical reflected process that stays at or above zero uniformly in time. In reality, however, this may not be  observable continuously, and it is necessary to consider the time taken to execute this process.  Motivated by recent research on Poisson observations and Parisian ruin found in (among others) \cite{albrecher2015strikingly,AIZ,BPPR,LRZ}, we consider the scenario where bail-outs can only be made at independent Poisson times.

We consider  a general spectrally negative \lev process as the underlying process.  Additionally, we are
 given \emph{Poisson observation times}, or an increasing sequence of jump times of an independent Poisson process.   At each Poisson observation time when the process is below zero, the process is pushed up to zero: we call this \emph{Parisian reflection}.
Related processes with underlying compound Poisson processes have been studied in \cite{albrecher2011randomized}  and
\cite{avanzi2013periodic}. In the former study, a number of  identities were obtained when solvency is only observed periodically, whereas the latter study analyzes a  case where observation intervals are Erlang-distributed.
%  It is remarked that Avanzi et al.\ ``On periodic dividend strategies in the dual model with diffusion" considered the case jumps are hyperexponential and solved the optimal dividend problem in the dual model when the dividends can be paid only at Poisson observation times.  We consider a generalization of their process.

In this paper, we are also interested in the case of a dividend-paying company that pays dividends according to a barrier strategy.
With regard to this, we consider a version of the (doubly) reflected process, where, given a spectrally negative \lev process reflected from above, it is pushed up to zero at Poisson observation times at which the process is below zero.

Our objective is to obtain, using fluctuation/excursion theories, concise expressions of several identities of  the following:%\red{[itemized]}
\begin{enumerate}
\item[(1)] the spectrally negative \lev process with Parisian reflection below, and
\item[(2)] its variant with additional classical reflection above.
\end{enumerate}
In particular, we are interested in the following:
\begin{description}
\item[(absolute ruin)]  We define \emph{absolute ruin} to be the event that the process goes below a specified level $a < 0$.  The absolute ruin probability and its time can be used to evaluate the risk of the company just like the classical ruin, which is the event that the process goes below $0$.
\item[(capital injections)]  Capital injections correspond to Parisian reflection below. We compute their total discounted values for both (1) and (2) for the infinite horizon case as well as for the cases that are killed upon exiting $[a,b]$, $[a, \infty)$ and $(-\infty, b]$ for $a < 0 < b$.
\item[(dividends)]  If dividends are assumed to be paid continuously, then they are modeled by the
 classical reflection above in process (2).  We compute their total expected discounted values  for the infinite horizon case and for the case that is killed upon exiting $[a, \infty)$ for $a < 0$.
 \end{description}

%In order to compute these fluctuation identities, we will use the scale function. 
We use the scale function to compute these fluctuation identities.
 It is well known, as in \cite{K}, that the scale function existing for every spectrally one-sided \lev process can be applied to obtain various fluctuation identities of the process and its reflected/refracted processes.

The main difficulty with the spectrally negative \lev process is handling the possible overshoot at its down-crossing time: it is  typically necessary to express the identities in terms of the convolution of the \lev measure and the resolvent measure via the scale function.  Recent results show, however, that these can be concisely written under some conditions.  In this paper, for process (1), we use the simplifying formula obtained in \cite{LRZ}. Together with this, the desired identities for the \emph{bounded variation case} can be obtained using a well-known technique via the strong Markov property; see, e.g., \cite{AIZ}.

For the \emph{unbounded variation case}, we shall use excursion theory instead of using the commonly used approximation methods as in \cite{AIZ}.  In doing so, we first obtain an excursion-measure version of the simplifying formula in \cite{LRZ} (see Theorem \ref{lemma_key_excursion}). Using this and excursion theory, we can obtain identities directly without relying on the approximation scheme.  Our approach follows from the recent characterization of the excursion measure away from $0$, as obtained in \cite{PPR2}.

For  process (2), we derive an analogue of the simplifying formula in \cite{LRZ} for the spectrally negative \lev process reflected from above (see Theorem \ref{lemma_simplifying_reflected}). With this and the results for (1), similar fluctuation identities can be obtained when process (1) is replaced with (2).

%\red{[say more references about excursion away from $0$ like your paper with Victor and JC?]}

The rest of the paper is organized as follows.  Section \ref{section_model} introduces (1) the spectrally negative \lev process with Parisian reflection below and (2) its version with additional classical reflection above. 
The scale functions and their applications are also reviewed.
%We also review scale functions and their applications. 
Section \ref{section_main} presents the main results for both processes (1) and (2) and their corollaries.  Sections \ref{section_bounded} and \ref{section_unbounded}  give the proofs of the main results related to process (1) for the bounded and unbounded variation cases, respectively.   Finally, Section \ref{section_proof_reflected} gives those of the theorems related to (2).  Some proofs of the  corollaries are provided in the appendix.

\section{\lev processes with Parisian reflection below} \label{section_model}

Let $X = (X(t); t \geq 0)$ denote a spectrally negative \lev process defined on a probability space $(\Omega, \mathcal{F}, \mathbb{P})$.
For $x\in \R$, we denote by $\mathbb{P}_x$ the law of $X$ when it starts at $x$ and write for convenience  $\mathbb{P}$ in place of $\mathbb{P}_0$. Accordingly, we shall write $\e_x$ and $\e$ for the associated expectation operators.  We assume that its
Laplace exponent $\k: [0, \infty) \to \R$ is given by
\begin{align}
 \E (e^{\th X(t)})=e^{ t \k(\th)}, \quad t \geq 0, \; \theta \geq 0, \label{def_Laplace_exp}
 \end{align}
with its L\'evy-Khintchine decomposition
 $$ \k(\th) = \frac{\s^2}{2}\th^2 + \gamma \th + \int_{(-\infty, 0)}[e^{\th y} - 1 - \th 1_{\{y > -1\}} ]\Pi(\diff y), \quad \theta \geq 0.$$
Here, $\s \geq 0,$ $\gamma \in \R$, and the \lev measure satisfies $ \int_{(-\infty, 0)} (1\wedge y^2) \Pi(\diff y) < \I$. Let $\mathbb{F} = (\mathcal{F}(t); t \geq 0)$ be the filtration generated by $X$. %\green{[put it back because we actually use it later.]}

%\red{[added]}
It is well known that $X$ has paths of bounded variation if and only if $\sigma=0$ and $\int_{(-1, 0)} |y|\Pi(\mathrm{d}y) < \infty$. In this case $X$ can be written as
\begin{equation}
X(t)=ct-S(t), \,\,\qquad t\geq 0,\notag
%\label{BVSNLP}
\end{equation}
where
\begin{align*}
c:=\gamma-\int_{(-1,0)} y\Pi(\mathrm{d}y) %\label{def_drift_finite_var}
\end{align*}
 and $(S(t); t\geq0)$ is a driftless subordinator. We exclude the case $X$ is the negative of a subordinator, and hence necessarily $c>0$. Its Laplace exponent is given by
\begin{equation*}
\kappa(\theta) = c \theta+\int_{(-\infty,0)}\big( {\rm e}^{\theta y}-1\big)\Pi(\ud y), \quad \theta \geq 0.
\end{equation*}

\subsection{\lev processes with Parisian reflection below} \label{subsection_process_defined}

%\red{[made this subsection. Feel free to change]}

%\red{[changed from $T_i$ to $T(i)$]}

Let $\mathcal{T}_\p=\{T(i); i \in \mathbb{N}\}$ be an increasing sequence of epochs of a Poisson process with rate $\p >0$, independent of $X$.   We construct the \emph{\lev process with Parisian reflection below} $X_r = (X_r(t); t \geq 0)$ as follows: the process is only observed at times $\mathcal{T}_\p$ and is pushed up to $0$ if and only if it is below $0$.

More precisely, we have
\begin{align} \label{X_X_r_the_same}
X_r(t) = X(t), \quad 0 \leq t < T_0^-(1)
\end{align}
 where
\begin{align} T_{0}^-(1) := \inf\{  S \in \mathcal{T}_r :  X(S-) < 0 \}; \label{def_T_0_1}
 \end{align}
 here and throughout, let $\inf \varnothing = \infty$.
 The process is then pushed upward by $|X(T_0^-(1))|$ so that $X_r(T_0^-(1)) = 0$. For $T_0^-(1) \leq t < T_0^-(2)  := \inf\{ S \in \mathcal{T}_r : S > T_0^-(1), X_r(S-) < 0\}$, we have $X_r(t) = X(t) + |X(T_0^-(1))|$.  The process can be constructed by repeating this procedure.

Suppose $R_r(t)$ is the cumulative amount of (Parisian) reflection until time $t \geq 0$. Then we have
\begin{align*}
X_r(t) = X(t) + R_r(t), \quad t \geq 0,
\end{align*}
 with
\begin{align}
	R_r(t) := \sum_{i=1}^\infty 1_{\{ T_0^-(i) \leq t\}} |X_r(T_0^-(i) -)|, \quad t \geq 0, \label{def_L_r}
\end{align}
%\red{[should we delete the following?]} 
where $(T_{0}^-(n); n \geq 1)$ can be constructed inductively by \eqref{def_T_0_1} and
\begin{eqnarray*} T_{0}^-(n+1) := \inf\{ S \in \mathcal{T}_r : S > T_0^-(n), X_r(S-) < 0 \}, \quad n \geq 1.
 \end{eqnarray*}
%\blue{[Kazu: you mena removing the previous two sentences?]} \red{[Yes, but is $T_0^-(n)$ used somewhere in the rest of the paper?]}\blue{[Kazu: It is used in the construction algorithm just below.]} \red{OK. How about writing above
%\begin{align}
%R_r(t) := \sum_{i=1}^\infty 1_{\{ T_0^-(i) \leq t\}} |X_r(T_0^-(i) -)|, \quad t \geq 0, \label{def_L_r}
%\end{align}
%??
%}
%\red{[added this]}
The process $Y_r^b$ with additional (classical) reflection above can be defined analogously.  Fix $b > 0$.  Let
 \begin{align*}
 Y^b(t) := X(t) - L^b(t) \quad \textrm{where } L^b(t) := \sup_{0 \leq s \leq t} (X(s) - b) \vee 0, \quad t \geq 0,
 \end{align*}
 be the process reflected from above at $b$.  We have
 \begin{align}
 Y^b_r(t) = Y^b(t), \quad 0 \leq t < \widehat{T}_0^{-} (1) \label{Y_matches}
 \end{align}
  where $\widehat{T}_{0}^-(1) := \inf\{ S \in \mathcal{T}_r : Y^b(S-) < 0 \}$.
The process then jumps upward by $|Y^b(\widehat{T}_0^-(1))|$ so that $Y^b_r(\widehat{T}_0^-(1)) = 0$. For $\widehat{T}_0^-(1) \leq t < \widehat{T}_0^-(2)  := \inf\{S \in \mathcal{T}_r : S > \widehat{T}_0^-(1), Y_r^b(S-) < 0 \}$, $Y_r^b(t)$ is the reflected process of $X(t) - X(\widehat{T}_0^-(1))$. %$X(t) + |Y^b(\widehat{T}_0^-(1))|$ \green{[this is probably not correct -- should be $X(t) - X(\widehat{T}_0^-(1))$?]}.
The process can be constructed by repeating this procedure.
It is clear that it admits a decomposition
\begin{align*}
Y^b_r(t) = X(t) + R_r^b (t) - L_r^b(t), \quad t \geq 0,
\end{align*}
where $R_r^b(t)$ and $L_r^b(t)$ are, respectively, the cumulative amounts of Parisian and classical reflection until time $t$.

For the sake of completeness, we provide below a formal construction of the processes $Y^b_r$, $R_r^b$, and $L_r^b$.
\begin{center}
\line(1,0){300}
\end{center}
 \textbf{Construction of the process $Y^b_r$, $R_r^b$, and $L_r^b$ under $\mathbb{P}_x$}
\begin{description}
\item[Step 0:] Set $n = 0$, $\widehat{T}_0^- (0) = 0$, $R_r^b (0-) = L_r^b(0-) = 0$,  and $\tilde{x} = x$ and go to \textbf{Step 1}. 
\item[Step 1:] Let $\{ \widetilde{Y}^b(t); t \geq \widehat{T}_0^- (n) \}$ be the reflected \lev process with the barrier level $b$ that starts at the time $\widehat{T}_0^- (n)$ at the level $\tilde{x}$, given by
\begin{align*}
\widetilde{Y}^b(t) = \widetilde{X}(t) - \widetilde{L}^b(t)
\end{align*}
where $\widetilde{X}(t) := X(t) - X(\widehat{T}_0^- (n))$ and $\widetilde{L}^b(t) := \sup_{\widehat{T}_0^- (n) \leq s \leq t} (\widetilde{X}(s) - b) \vee 0$.
\begin{description}
\item[Step 1-0:]   Set $\widehat{T}_0^- (n+1) := \inf\{S \in \mathcal{T}_r : S > \widehat{T}_0^-(n), \widetilde{Y}^b(t)  < 0 \} $.
\item[Step 1-1:]  For $t \in (\widehat{T}_0^- (n), \widehat{T}_0^- (n+1))$,  set $R_r^b (t) = R_r^b (\widehat{T}_0^- (n))$.  Also, set $R_r^b (\widehat{T}_0^- (n+1)) = R_r^b (\widehat{T}_0^- (n)) + |\widetilde{Y}^b(\widehat{T}_0^- (n+1))|$.
\item[Step 1-2:]  For $t \in (\widehat{T}_0^- (n), \widehat{T}_0^- (n+1)]$,  set $L_r^b (t) = L_r^b (\widehat{T}_0^- (n)) + \widetilde{L}^b (t) $. 
\item[Step 1-3:]  For $t \in [\widehat{T}_0^- (n), \widehat{T}_0^- (n+1))$,  set $Y^b_r(t) = X(t) + R_r^b (t) - L_r^b(t)$. 
\end{description}
Increment the value of $n$, set $\tilde{x} = 0$ and go back to the beginning of \textbf{Step 1}.
\end{description}
\begin{center}
\line(1,0){300}
\end{center}

\subsection{Scale functions}
Fix $q \geq 0$. Let $W_q$ be the scale function of $X$.  Namely,  this is a mapping from $\R$ to $[0, \infty)$ that takes the value zero on the negative half-line, while on the positive half-line it is a strictly increasing function that is defined by its Laplace transform:
\begin{align} \label{scale_function_laplace}
\begin{split}
\int_0^\infty  \mathrm{e}^{-\theta x} W_q(x) \diff x &= \frac 1 {\kappa(\theta)-q}, \quad \theta > \Phi_q, \\
%\int_0^\infty  \mathrm{e}^{-\theta x} \mathbb{W}^{(q)}(x) \diff x &= \frac 1 {\psi_Y(\theta) -q}, \quad \theta > \varphi(q),
\end{split}
\end{align}
where
\begin{align}
\begin{split}
\Phi_q := \sup \{ \lambda \geq 0: \kappa(\lambda) = q\}. \notag \end{split}
%\label{def_varphi}
\end{align}
In particular, when $q=0$, we shall drop the superscript.
%\red{[I guess $q=0$ case does not exist in this paper?]}
%\blue{I think it can, because some expressions have $\Phi_q$ and $q$ can be $0$?}
%By the strict  convexity of $\psi$, we derive the inequality $\varphi(q) > \Phi(q) > 0$ for $q > 0$ and  $\varphi(q) \geq \Phi(q) \geq 0$ for $q = 0$.
We also define, for $x \in \R$,
\begin{align*}
\overline{W}_q(x) &:=  \int_0^x W_q(y) \diff y, \\
Z_q(x) &:= 1 + q \overline{W}_q(x),  \\
\overline{Z}_q(x) &:= \int_0^x Z_q (z) \diff z = x + q \int_0^x \int_0^z W_q (w) \diff w \diff z.
\end{align*}
Noting that $W_q(x) = 0$ for $-\infty < x < 0$, we have
\begin{align*}
\overline{W}_q(x) = 0, \quad Z_q(x) = 1,  \quad \textrm{and} \quad \overline{Z}_q(x) = x, \quad x \leq 0.  %\label{z_below_zero}
\end{align*}

%\blue{Kazu: Instead of your red comment how about this: where the right hand side above is understood in the limiting sense $\lim_{q\downarrow0} q/\Phi(q)=
%0\vee(1/\psi'(0))$, and  $\lim_{q\downarrow0} q/\varphi(q)=
%0\vee(1/\psi'_Y(0))$ when $q = 0$.}

Let \begin{align*}
%\label{first_passage_time}
\tau_a^- := \inf \left\{ t > 0: X(t) < a \right\} \quad \textrm{and} \quad \tau_a^+ := \inf \left\{ t > 0: X(t) >  a \right\}, \quad a \in \R.
\end{align*}
Then, for any $b > a$ and $x \leq b$, %\red{[changed the indicator to a comma]}
\begin{align}
\begin{split}
\E_x \left( e^{-q \tau_b^+}; \tau_b^+ < \tau_a^- \right) &= \frac {W_q(x-a)}  {W_q(b-a)}, \\
 \E_x \left( e^{-q \tau_a^-}; \tau_b^+ > \tau_a^- \right) &= Z_q(x-a) -  Z_q(b-a) \frac {W_q(x-a)}  {W_q(b-a)}.
% \E_x \left[ e^{-q \tau_0^-} \right] &= Z^{(q)}(x) -  \frac q {\Phi(q)} W^{(q)}(x).
\end{split}
 \label{laplace_in_terms_of_z}
\end{align}
In addition, as in Theorem 8.7 of \cite{K}, the \emph{$q$-resolvent measure} is known to have a density written as %\red{[swapped $a$ and $b$]}
\begin{align} \label{resolvent_density}
\E_x \Big( \int_0^{\tau_{a}^- \wedge \tau^+_b} e^{-qt} 1_{\left\{ X(t) \in \diff y \right\}} \diff t\Big) &= \Big[ \frac {W_q(x-a) W_q (b-y)} {W_q(b-a)} -W_q (x-y) \Big] \diff y,  \quad x \leq b.
\end{align}

It is known that a spectrally negative \lev process creeps downwards (i.e.\ $\mathbb{P}_x ( X(\tau_a^-)= a, \tau_a^- < \infty ) > 0$ for $x > a$) if and only if $\sigma > 0$ (see Exercise 7.6 of \cite{K}). %Hence, \emph{for the case of bounded variation}, the above identity \eqref{resolvent_density} together with the compensation formula (see Theorem 4.4 of \cite{K})
On the other hand, it is known that  for any $a \leq x\leq  b$ and a positive measurable function $l$ (with a slight abuse of notation noting that $W_q$ is differentiable on $(0, \infty)$ when $\sigma > 0$ as in Remark \ref{remark_smoothness_zero} below),
\begin{align} \label{undershoot_expectation}
\E_x\Big(e^{-q\tau_a^-}l(X(\tau_a^-))&; \tau_a^-<\tau_b^+ \Big)=l(a)\frac{\sigma^2}{2}\left[W_q'(x-a)-W_q(x-a)\frac{W_q'(b-a)}{W_q(b-a)}\right]\\
&+\int_0^{b-a}\int_{(-\infty,-y)} l(y+u+a)\left\{\frac{W_q(x-a) W_q(b-a-y)}{ W_q(b-a)}-W_q(x-a-y)\right\}\Pi(\ud u)\ud y;\notag
\end{align}
see for instance the identity (4) in \cite{LRZ}.

%\red{[moved here]
By taking $a \downarrow -\infty$ and $b \uparrow \infty$ in \eqref{resolvent_density}, we obtain that $U_{q}(\diff y):=\int^{\infty}_{0}e^{-qt}\mathbb{P}(X(t)\in \diff y) \diff t $, $y\in \mathbb{R}$,
 is absolutely continuous with respect to the Lebesgue measure, and its density is given by
 \begin{equation}
 u_{\q}(y)=\kappa'(\Phi_q)^{-1}e^{-\Phi_{\q}y}-W_{\q}(-y),\qquad y\in \mathbb{R}; \label{u_q_resolvent}
 \end{equation}
 %  \green{[If we stick to $\Phi_q$, we need to modify $\Phi'(q)$ somehow]}\blue{Kazu: I left $\Phi_q'$ because I believe is the notation Florin is using.}
 see for instance Corollary 8.9 in \cite{K} and Exercise 2 in Chapter VII in \cite{B}.

 %\red{[JL: Moved this remark earlier because \eqref{W_q_limit} is used.]}
 \begin{remark} \label{remark_smoothness_zero}
\begin{enumerate}
\item If $X$ is of unbounded variation or the \lev measure is atomless, it is known that $W_q$ is $C^1(\R \backslash \{0\})$; see, e.g.,\ \cite[Theorem 3]{Chan2011}.
%Hence,
%\begin{enumerate}
%\item $Z^{(q)}$ is $C^1 (\R \backslash \{0\})$ and $C^0 (\R)$ for the bounded variation case, while it is $C^2(\R \backslash \{0\})$ and $C^1 (\R)$ for the unbounded variation case, and
%\item $\overline{Z}^{(q)}$ is $C^2(\R \backslash \{0\})$ and $C^1 (\R)$ for the bounded variation case, while it is $C^3(\R \backslash \{0\})$ and $C^2 (\R)$ for the unbounded variation case.
%\end{enumerate}f
\item Regarding the asymptotic behavior near zero, as in Lemmas 3.1 and 3.2 of \cite{KKR},
\begin{align}\label{eq:Wqp0}
\begin{split}
W_q (0) &= \left\{ \begin{array}{ll} 0 & \textrm{if $X$ is of unbounded
variation,} \\ \frac 1 {c} & \textrm{if $X$ is of bounded variation,}
\end{array} \right. \\
W_q^{\prime} (0+) &:= \lim_{x \downarrow 0}W_q^{\prime} (x) =
\left\{ \begin{array}{ll}  \frac 2 {\sigma^2} & \textrm{if }\sigma > 0, \\
\infty & \textrm{if }\sigma = 0 \; \textrm{and} \; \Pi(-\infty,0) = \infty, \\
\frac {q + \Pi(-\infty, 0)} {c^2} &  \textrm{if }\sigma = 0 \; \textrm{and} \; \Pi(-\infty, 0) < \infty.
\end{array} \right.
\end{split}
\end{align}
%Regarding their asymptotic values as $x \downarrow 0$, as in Lemma 3.1 of \cite{KKR},
%\begin{align*}%\label{eq:Wqp0}
%\begin{split}
%W_q (0) &= \left\{ \begin{array}{ll} 0 & \textrm{if $X$ is of unbounded
%variation,} \\ c^{-1} & \textrm{if $X$ is of bounded variation,}
%\end{array} \right.
%% \mathbb{W}^{(q)} (0) &= \left\{ \begin{array}{ll} 0 & \textrm{if $Y$ is of unbounded
%%variation,} \\ (c-\delta)^{-1} & \textrm{if $Y$ is of bounded variation,}
%%\end{array} \right.
%%\\
%%W^{(q)'} (0+) &:= \lim_{x \downarrow 0}W^{(q)'} (x) =
%%\left\{ \begin{array}{ll}  \frac 2 {\sigma^2} & \textrm{if }\sigma > 0, \\
%%\infty & \textrm{if }\sigma = 0 \; \textrm{and} \; \nu(-\infty,0) = \infty, \\
%%\frac {q + \nu(-\infty, 0)} {c^2} &  \textrm{if }\sigma = 0 \; \textrm{and} \; \nu(-\infty, 0) < \infty.
%%\end{array} \right.
%\end{split}
%\end{align*}
%\red{moved to here:}
On the other hand, as in Lemma 3.3 of \cite{KKR}, %\red{given $\psi'(\Phi(q)) \neq 0$ and $\psi_Y'(\varphi(q)) \neq 0$, respectively,}
\begin{align}
\begin{split}
e^{-\Phi_q x}W_q (x) \nearrow \kappa'(\Phi_q)^{-1}, \quad \textrm{as } x \uparrow \infty,
\end{split}
\label{W_q_limit}
\end{align}
where in the case $\kappa'(0+) = 0$, the right hand side, when $q=0$,  is understood to be infinity.
%\green{[just to be consistent throughout the paper, can we use $\uparrow$ and $\downarrow$ instead of $\rightarrow$?]}
\item As in (8.22) and Lemma 8.2 of \cite{K}, $ {W_q^{\prime}(y+)} / {W_q(y)} \leq  {W_q^{\prime}(x+)} / {W_q(x)}$ for  $y > x > 0$.
In all cases, $W_q^{\prime}(x-) \geq W_q^{\prime}(x+)$ for all $x >0$.
\end{enumerate}
\end{remark}

%\red{[need to define $Z_{\q}(x,\Phi_{\q+\p}).$]}
%\red{define also $W_q(x, \theta)$ and $\overline{Z}_q(x, \theta)$?}\blue{Kazu: Are they used?}
%\red{[changed the order a bit]}
%\red{[changed below because the firsst one does not depend on $\beta$.]}
Along this work we also define, for $\theta\geq 0$, $\alpha\geq0$, and $x \in \R$, %\red{[JL: So I guess we should allow $r$ to be negative so that we can use $Z_{q+r, -r}?$. Alternatively we can write here $Z_{q,p}$ or something.  We may want to fix $r > 0$ for the Poisson parameter.  Also, $\theta \geq 0?$]}
\begin{align}
Z_{\alpha}(x, \theta) &:=e^{\theta x} \left( 1 + (\alpha- \k(\theta)) \int_0^{x} e^{-\theta z} W_{\alpha}(z) \diff z	\right) \la{Zs}, \\
%\tag{Z2}
Z_{\alpha,\beta}(x,\th ) &:=%\left\{\begin{array}{ll}\displaystyle
\fr{\beta}{\alpha+\beta-\k(\th)} Z_{\alpha}(x,\th)
+\fr{\alpha-\k(\th)}{\alpha+\beta-\k(\th)} Z_{\alpha}(x,\Phi_{\alpha+\beta} ), \quad \beta \geq - \alpha,%& \theta \neq \Phi_{\alpha+\beta} %\\ \displaystyle (1-\beta\Phi_{\alpha+\beta}'x)Z_{\alpha}(x, \Phi_{\alpha+\beta})+\beta e^{\Phi_{\alpha+\beta}x}\left(\int_0^xe^{-\Phi_{\alpha+\beta}z}(1-\beta\Phi_{\alpha+\beta}'z)W_{\alpha}(z)\diff z\right)  & \theta = \Phi_{\alpha+\beta}  \end{array} \right\}.
\la{Z2}
\end{align}
where the case $\th=\Phi_{\alpha+\beta}$ in \eqref{Z2} should be interpreted in the limiting sense as $\theta \rightarrow \Phi_{\alpha+\beta}$.

%\blue{Hi Kazu: I am probably doing something wrong but I am little confused about \eqref{Z2}, because
%\begin{align*}
%\lim_{\th\to\Phi_{\alpha+\beta}}Z_{\alpha,\beta}(x,\th %)&=\lim_{\th\to\Phi_{\alpha+\beta}}\fr{\beta}{\alpha+\beta-\k(\th)} Z_{\alpha}(x,\th)
%+\fr{\alpha-\k(\th)}{\alpha+\beta-\k(\th)} Z_{\alpha}(x,\Phi_{\alpha+\beta} )\\
%&=\lim_{\th\to\Phi_{\alpha+\beta}}Z_{\alpha}(x,\Phi_{\alpha+\beta} )+\fr{\beta}{\alpha+\beta-\k(\th)} (Z_{\alpha}(x,\th)- Z_{\alpha}(x,\Phi_{\alpha+\beta} )).
%\end{align*}
%Now by L'Hospital we have
%\begin{align*}
%\lim_{\th\to\Phi_{\alpha+\beta}}\fr{\beta}{\alpha+\beta-\k(\th)} (Z_{\alpha}(x,\th)- Z_{\alpha}(x,\Phi_{\alpha+\beta} ))=-\beta\Phi_{\alpha+\beta}'\frac{\partial}{\partial\th} %Z_{\alpha}(x,\th)|_{\th=\Phi_{\alpha+\beta}}.
%\end{align*}
%Is the last term equal to $0$?
%	}
In particular,
\begin{align} \label{Z_special}
\begin{split}
Z_{\alpha}(x, 0) &=Z_{\alpha}(x), \\
Z_{\alpha}(x, \Phi_{\alpha+\beta}) &=e^{\Phi_{\alpha+\beta} x} \left( 1 -\beta \int_0^{x} e^{-\Phi_{\alpha+\beta} z} W_{\alpha}(z) \diff z \right) = \beta \int_0^\infty e^{-\Phi_{\alpha+\beta} z} W_\alpha(z+x) \diff z, \\
Z_{\alpha,\beta}(x) &:=Z_{\alpha,\beta}(x,0) =\fr{\beta}{\alpha+\beta} Z_{\alpha}(x)
+\fr{\alpha}{\alpha+\beta} Z_{\alpha}(x,\Phi_{\alpha+\beta} ).
\end{split}
\end{align}
%\red{[moved these here for now.]
%In order to state our first result we will define the following:
%\green{[Maybe \eqref{Z2} can be defined before the lemma?]}
%\green{[remove this? It is not trivial and I don't think it is used anymore?] Note  that $\lim_{\beta \uparrow \infty}Z_{\alpha,\beta}(x,\th )=Z_{\alpha}(x,\th )$.}
%\green{
Define also, for all $x \in \R$ and $a \leq 0$, %\red{[moved this expression here]}
\begin{align} \label{W_a_def}
\begin{split}
W^a_{\alpha,\beta}(x) &:=W_{\alpha+\beta}(x-a)- \beta\int_{0}^{x}  W_{\alpha}(x-y)W_{\alpha+\beta}(y-a)\diff y =W_{\alpha}(x-a)+\beta\int_0^{-a}W_{\alpha}(x-u-a)W_{\alpha+\beta}(u) \diff u, \\
Z_{\alpha,\beta}^a(x) &:= Z_{\alpha+\beta}(x-a)- \beta \int_0^x W_{\alpha}(x-y)  Z_{\alpha+\beta}(y-a)  \diff y
=Z_{\alpha}(x-a)+\beta\int_0^{-a}W_{\alpha}(x-u-a)Z_{\alpha+\beta}(u) \diff u, \\
\overline{Z}_{\alpha,\beta}^a(x) &:= \overline{Z}_{\alpha+\beta}(x-a)- \beta \int_0^x W_{\alpha}(x-y)  \overline{Z}_{\alpha+\beta}(y-a)  \diff y =\overline{Z}_{\alpha}(x-a)+\beta\int_0^{-a}W_{\alpha}(x-u-a)\overline{Z}_{\alpha+\beta}(u)\diff u,
\end{split}
\end{align}
where the second equalities hold by (5) of \cite{LRZ} and (3.4) of \cite{YP}, and in particular
\begin{align}\label{rel_Z_0}
W^0_{\alpha,\beta}(x)=W_{\alpha}(x), \quad Z^0_{\alpha,\beta}(x)=Z_{\alpha}(x), \quad \textrm{and} \quad \overline{Z}^0_{\alpha,\beta}(x)=\overline{Z}_{\alpha}(x).
\end{align}
%\red{[moved here]}
These functions are related by the following: by  \eqref{W_q_limit} and  \eqref{Z_special}, for $x \in \R$ and $a < 0$, respectively,  %\red{[shortened]}
	\begin{align}\label{nl}
	\lim_{a\downarrow -\infty} \frac{W_{\alpha,\beta}^a(x)}{W_{\alpha+\beta}(-a)}
	%=\frac{W_{\q+\p}(x-a)}{W_{\q+\p}(-a)}-\p\int_0^xW_q(x-z)\frac{W_{\q+\p}(z-a)}{W_{\q+\p}(-a)}\diff z
	%&=e^{\Phi_{\q+\p}x}-r\int_0^xe^{-\Phi_{\q+\p}z}W_{\q}(z)\diff z\notag\\
	= Z_{\alpha}(x,\Phi_{\alpha+\beta}) 	\quad
\textrm{and} \quad \lim_{x\uparrow \infty}\frac{W_{\alpha,\beta}^a(x)}{W_{\alpha}(x)}
	%&=e^{-\Phi_q a}+r\int_0^{-a}e^{-\Phi_q(u+a)}W_{q+r}(u)du
	%&=e^{-\Phi_q a}\left(1+r\int_0^{-a}e^{-\Phi_q u}W_{q+r}(u)du\right)
	=Z_{\alpha+\beta}(-a,\Phi_{\alpha}).
	\end{align}

\section{Main results} \label{section_main}

In this section, we summarize the main results related to the processes $X_r$ and $Y_r^b$ as defined in Section \ref{subsection_process_defined}.
The proofs of Theorems \ref{theorem_laplace} and  \ref{expected_bailouts} are given in Section \ref{section_bounded} for the case $X$ is of bounded variation and in Section \ref{section_unbounded} for the case of unbounded variation.  The proofs of Theorems  \ref{theorem_R_b_r} and  \ref{theorem_L_b_r} are given in Section \ref{section_proof_reflected}.  Those for corollaries are given in Appendix \ref{appendix_proofs}. %\green{[there are no lemmas in this section and so ``those for corollaries"?]}

Throughout, let us fix $r > 0$ and define \begin{align}
\tau_a^-(r) := \inf \left\{ t > 0: X_r(t) < a \right\} \quad \textrm{and} \quad \tau_a^+(r) := \inf \left\{ t > 0: X_r(t) >  a \right\}, \quad a \in \R. \label{def_tau_a_plus_minus}
\end{align}
In particular, the former for $a < 0$ can be understood as the ``absolute ruin" as discussed in Section \ref{section_introduction}.

%\red{JL: I am wondering if it can be $q =0$?  There is currently no description about $q$.} \blue{Kazu: I believe that $q$ can be $0$ except for Corollary 3.2 (ii)? }

\subsection{Identities for the process $X_r$} \label{subsection_joint_Laplace}

%\red{[JL: How about combining these two subsections to ``Identities for $X_r$?  Both are about the bailout $L$ and I do not see the point of separating ]}

We shall first obtain the joint Laplace transform (with killing) of the stopping times \eqref{def_tau_a_plus_minus} and the value of capital injections as in \eqref{def_L_r}. % \green{[delete this sentence]Notice that the case $a = 0$  is already covered by classic theory. }
%Throughout, let us assume $q, \theta \geq 0$.

\beT[Joint Laplace transform with killing]  \label{theorem_laplace}
For all $q, \theta \geq 0$, $a< 0<b$, and $x \leq b$,
%\begin{itemize}
%\item[(i)]
\begin{align}  \la{Parisbailouts}  %\tag{PARISBAILOUTS}
g(x,a,b,\theta)&:=\E_x\left( e^{-q\tau_b^+ (r) - \theta R_r(\tau_b^+ (r))}; \t_b^+(r)< \tau_a^-(r)\right)=\frac{\mathcal{H}^a_{q,r}(x,\theta)}{\mathcal{H}^a_{q,r}(b,\theta)}, \\
h(x,a,b,\theta)&:=\E_x\left(e^{-q\tau_a^-(r) -\theta R_r(\tau_a^- (r) )}; \tau_a^-(r) < \t_b^+(r)\right)=\mathcal{I}_{\q,\p}^a(x)-
\frac{\mathcal{H}^a_{q,r}
(x,\theta)}{\mathcal{H}^a_{q,r}(b,\theta)} \mathcal{I}_{\q,\p}^a(b), \la{Parisbailouts_2}
\end{align}
%\end{itemize}
where, for $y \in \R$, %[changed from $x$ to $y$ below]}
\begin{align} \label{def_H_I}
\begin{split}
\mathcal{H}^a_{q,r}(y,\theta)&:=\p\int_{0}^{-a} e^{-\theta u}\Big[W^a_{\q,\p}(y)\frac{W_{\q+\p} (u)} {W_{\q+\p}(-a)}-W^{-u}_{\q,\p}(y)\Big]\diff u+\frac{W^a_{\q,\p}(y)}{W_{\q+\p}(-a)}, \\
\mathcal{I}_{\q,\p}^a(y)&:=Z^a_{\q,\p}(y)-W^a_{\q,\p}(y)\frac{Z_{\q+\p}(-a)}{W_{\q+\p}(-a)}.
\end{split}
\end{align}
In particular, by  \eqref{Z_special} and \eqref{W_a_def}, simple computation gives
\begin{align*}
\mathcal{H}^a_{q,r}(y,0)
&= (q+r)^{-1} \Big( \frac {W^a_{\q,\p}(y)} {W_{\q+\p}(-a)} \big[ r Z_{q+r}(-a) + q \big] + r \big[ Z_{q}(y) - Z_{q,r}^a (y) \big] \Big) \\ &= (q+r)^{-1} \Big(- r \mathcal{I}^a_{q,r} (y) + q \frac {W^a_{\q,\p}(y)} {W_{\q+\p}(-a)} + r Z_q(y)\Big).
\end{align*}

%In particular, %\red{[simplified]}
%\begin{align*}
%\mathcal{H}^a_{q,r}(x,0)%&:=\p\int_{0}^{-a}\Bigg[W^a_{\q,\p}(x)\frac{W_{\q+\p} (u)} {W_{\q+\p}(-a)}-W^{-u}_{\q,\p}(x)\Bigg]\diff u+\frac{W^a_{\q,\p}(x)}{W_{\q+\p}(-a)} \\
%&=\p\Bigg[W^a_{\q,\p}(x)\frac{\overline{W}_{\q+\p} (-a)} {W_{\q+\p}(-a)}- \int_{0}^{-a} W^{-u}_{\q,\p}(x)\diff u\Bigg]+\frac{W^a_{\q,\p}(x)}{W_{\q+\p}(-a)} \\
%&=W^a_{\q,\p}(x)\frac{r \overline{W}_{\q+\p} (-a)+1} {W_{\q+\p}(-a)}- r \int_{0}^{-a} W^{-u}_{\q,\p}(x)\diff u,
%\end{align*}
%where
%\begin{align*}
%\int_0^{-a}W^{-u}_{\q,\p}(x) \diff u %:= \int_0^{-a}W_{\q+\p}(x+u) \diff u - \p\int_{0}^{x}  W_\q(x-y) \int_0^{-a} W_{\q+\p}(y+u) \diff u \diff y \\
%= \overline{W}_{\q+\p}(x-a) -  \overline{W}_{\q+\p}(x)   - \p\int_{0}^{x}  W_\q(x-y) \Big(\overline{W}_{\q+\p}(y-a) -  \overline{W}_{\q+\p}(y) \Big) \diff y \\
%&= \frac 1 {q+r} \Big[ Z_{\q+\p}(x-a) -  Z_{\q+\p}(x)   - \p\int_{0}^{x}  W_\q(x-y) \Big(Z_{\q+\p}(y-a) -  Z_{\q+\p}(y) \Big) \diff y \Big] \\
%&= \frac 1 {q+r} \Big[ Z_{q,r}^a (x) - Z_{q,r}^0(x)\Big] = \frac 1 {q+r} \Big[ Z_{q,r}^a (x) - Z_{q}(x)\Big].
%\end{align*}
%Hence,
%\begin{align*}
%\mathcal{H}^a_{q,r}(x,0) &= W^a_{\q,\p}(x)\frac{r \overline{W}_{\q+\p} (-a)+1} {W_{\q+\p}(-a)}- \frac r {q+r} \Big[ Z_{q,r}^a (x) - Z_{q}(x)\Big] \\
%&= \frac {W^a_{\q,\p}(x)} {W_{\q+\p}(-a)} \Big( \frac r {q+r} Z_{q+r}(-a) + \frac q {q+r}\Big) - \frac r {q+r} \Big[ Z_{q,r}^a (x) - Z_{q}(x)\Big].
%\end{align*}
\eeT

\begin{remark}  In Theorem \ref{theorem_laplace}, by taking $a \uparrow 0$, we recover \eqref{laplace_in_terms_of_z}: for $\theta \geq 0$, $b > 0$, and $0 \leq x \leq b$,
\begin{align*}
\lim_{a\uparrow 0}g(x,a,b,\th)&=W_q(x)/W_q(b) = \E_x \big(e^{-q \tau_b^+}; \tau_b^+ < \tau_0^- \big), \\
\lim_{a\uparrow 0}h(x,a,b,\th)&=Z_q(x)- {Z_q(b)} \frac {W_q(x)} {W_q(b)} = \E_x \big(e^{-q \tau_0^-}; \tau_b^+ >  \tau_0^- \big).
\end{align*}
Indeed, the former is immediate by the convergence
\begin{align}
\lim_{a\uparrow 0}\mathcal{H}^a_{q,r}(x,\theta)W_{q+r}(-a)=\lim_{a\uparrow 0}W_{q,r}^a(x)=W_q(x). \label{limit_fraction_W_H}
\end{align}
%About Florin's suggestion: Let us take $a\to0$ in (3.22) to this end by multiplying each term by $W_{q+r}(-a)$ we obtain that for all $x<b$
%	\[
%\lim_{a\uparrow 0}\mathcal{H}^a_{q,r}(x,\theta)W_{q+r}(-a)=\lim_{a\uparrow 0}W_{q,r}^a(x)=W_q(x)
%\]
%Hence
%\[
%\lim_{a\uparrow 0}g(x,a,b,\th)=W_q(x)/W_q(b).
%\]
%And we recover the classical two-sided exit problem.
%On the other hand we note that
The latter holds by \eqref{limit_fraction_W_H} and because
\begin{align*}
W_{q+r}(-a) &\big[ \mathcal{I}_{\q,\p}^a(x)\mathcal{H}^a_{q,r}(b,\theta) -\mathcal{I}_{\q,\p}^a(b)\mathcal{H}^a_{q,r}(x,\theta)\big] \\
&= W_{q+r}(-a) \big[ Z_{q,r}^a(x)\mathcal{H}^a_{q,r}(b,\theta) -Z_{q,r}^a(b)\mathcal{H}^a_{q,r}(x,\theta)\big] \\
				&+ Z_{q+r} (-a) r \Big( W^a_{\q,\p}(x) \int_0^{-a}e^{-\th u}W_{q,r}^{-u}(b) \diff u -W^a_{\q,\p}(b) \int_0^{-a}e^{-\th u}W_{q,r}^{-u}(x)\diff u \Big) \\ &\xrightarrow{a \uparrow 0} Z_q(x)W_q(b)- Z_q(b)W_q(x).
\end{align*}
%Hence
%\[
%\lim_{a\uparrow 0}(\mathcal{I}_{\q,\p}^a(b)\mathcal{H}^a_{q,r}(x,\theta)-\mathcal{I}_{\q,\p}^a(x)\mathcal{H}^a_{q,r}(b,\theta))W_{q+r}(-a)=Z_q(b)W_q(x)-Z_q(x)W_q(b)
%\]
%Therefore
%\[
%\lim_{a\uparrow 0}h(x,a,b,\th)=-\lim_{a\uparrow 0}\frac{(\mathcal{I}_{\q,\p}^a(b)\mathcal{H}^a_{q,r}(x,\theta)-\mathcal{I}_{\q,\p}^a(x)\mathcal{H}^a_{q,r}(b,\theta))}{\mathcal{H}^a_{q,r}(b,\theta)}=Z_q(x)-\frac{Z_q(b)}{W_q(b)}W_q(x).
%\]
%And we recover the classical two-sided exit problem again.
\end{remark}
By taking $b \uparrow \infty$ and $a \downarrow  - \infty$ in  Theorem \ref{theorem_laplace}, we have the following.
\begin{corollary} \label{corollary_laplace_tau_a_b_infty}
Fix $q, \theta \geq 0$.

%\green{JL: the overline and underline are not consistent, and I don't know what to do when there is a hat as well.  For Corollaries, I guess we can just not define the notation and start with expectations?  I don't think any of these notations are used.}

(i) Suppose $q > 0$. For $a <  0$ and $x \in \R$,
	\begin{align*}
	%\bar{h}(x,a,\theta):=
	\E_x\left( e^{-q\tau_a^-(r)-\theta R_r(\tau_a^-(r))};\tau_a^-(r)<\infty\right) =\mathcal{I}_{\q,\p}^a(x)-
	\mathcal{J}_{\q,\p}(-a)\frac{\mathcal{H}^a_{q,r}
		(x,\theta)}{\mathcal{G}_{q,r}(-a,\theta)},
	\end{align*}
	%\end{itemize}
%\green{[change to $\bar{h}$ just to be consistent with others?]} where %\red{[change to $g(x,a,b, \theta)$ and $h(x,a,b, \theta)$?]}
where, for $y \in \R$,
	\begin{align*}
	\mathcal{G}_{q,r}(y,\theta)&:=\p\int_{0}^{y} e^{-\theta u}\Big[Z_{q+r}(y,\Phi_q)\frac{W_{\q+\p} (u)} {W_{\q+\p}(y)}-Z_{q+r}(u,\Phi_q)\Big]\diff u+\frac{Z_{\q+\p}(y,\Phi_q)}{W_{\q+\p}(y)}, \\
	\mathcal{J}_{\q,\p}(y)&:=\frac{q}{\Phi_q}Z_{q+r,-r}(y)-Z_{q+r}(y,\Phi_q)\frac{Z_{q+r}(y)}{W_{q+r}(y)}.
	\end{align*}
	%\red{[JL: probably $\mathcal{I}_{\q,\p}(x):=\frac{q}{\Phi_q}Z_{q+r,-r}(x)-Z_{q+r}(x,\Phi_q)\frac {q+r} {\Phi_{q+r}}?$] }\blue{Kazu: I believe it is ok above because we are dividing by $W_q$ and taking limits when $b\to\infty$?} \red{[I see.  Actually I found the notation $\mathcal{H}^a_{q,r}(b,\theta) /W(b) \rightarrow \mathcal{H}_{q,r}(-a,\theta)$ a bit confusing because the location of $a$ moves.  Do you think it is weird to say $\mathcal{H}^a_{q,r}(b,\theta) / W(b) \rightarrow \overline{\mathcal{H}}^a_{q,r}(\theta)$ and $\mathcal{I}^a_{q,r}(b,\theta) / W(b) \rightarrow \overline{\mathcal{I}}^a_{q,r}$ etc]}\blue{I am not sure Kazu, seems natural to consider them like scale functions? } \red{I see.  How about using different character like $\mathcal{G}$?}
	%\blue{[Kazu: changed them for $\mathcal{J}$ and $\mathcal{G}$, is it ok?]}
	Here, in particular, using the fact that $Z'_{q+r}(x, \Phi_q) = \Phi_q Z_{q+r}(x, \Phi_q) + r W_{q+r}(x)$,
	%\begin{align*}
	%\mathcal{H}_{q,r}(x,0) = \frac {Z_{q+r}(x,\Phi_q)} {W_{\q+\p}(x)} \Big( \frac r {q+r} Z_{q+r}(x) + %\frac q {q+r}\Big) - \frac {rq} {\Phi_q(q+r)} \Big( Z_{q+r,-r} (x) - 1\Big).
	%\end{align*}
	%\red{OK to change to?
		\begin{align*}
	\mathcal{G}_{q,r}(y,0) = (q+r)^{-1} \Big[ \frac {Z_{q+r}(y,\Phi_q)} {W_{\q+\p}(y)} \big(r Z_{q+r}(y) + q \big) - \frac {rq} {\Phi_q} \big( Z_{q+r,-r} (y) - 1\big) \Big].
	\end{align*}

For the case $q = 0$, it holds with
%	\begin{align*}
%	\mathcal{J}_{0,\p}(y)&:= \left\{ \begin{array}{ll}r  (e^{\Phi_0 y} -1) / {\Phi_0}-Z_{r}(y,\Phi_q)\frac{Z_{r}(y)}{W_{r}(x)} & \Phi_0 > 0, \\ r [y - \kappa'(0+) \overline{W}_0(y)] + \kappa'(0+) - \frac{Z_{r}(y)^2}{W_{r}(y)} & \Phi_0 = 0,
%	\end{array} \right. \\
%\mathcal{G}_{0,r}(y,0) &:= \left\{  \begin{array}{ll}
%	   \frac {Z_{r}(y,\Phi_0)} {W_{\p}(y)}  Z_{r}(y)  - r  (e^{\Phi_0 y} -1) / {\Phi_0} & \Phi_0 > 0, \\
%	\frac {Z_{r}(y)^2} {W_{\p}(y)}  -  r [y - \kappa'(0+) \overline{W}_0(y)] & \Phi_0 = 0. \\
%	\end{array} \right.
%	\end{align*}
%}
%\blue{Kazu: I get this
	\begin{align*}
	\mathcal{J}_{0,\p}(y)&:= \left\{ \begin{array}{ll}r\Phi_0^{-1}[Z_r(y,\Phi_0)-Z_r(y)]-Z_{r}(y,\Phi_0)\frac{Z_{r}(y)}{W_{r}(y)} & \Phi_0 > 0, \\ r \overline{Z}_r(y)+{\k'(0+)} -\frac {Z_{r}(y)^2} {W_{\p}(y)} & \Phi_0 = 0,
	\end{array} \right. \\
	\mathcal{G}_{0,r}(y,0) &:= \left\{  \begin{array}{ll}
	\frac {Z_{r}(y,\Phi_0)} {W_{\p}(y)}  Z_{r}(y)  - r\Phi_0^{-1}[Z_r(y,\Phi_0)-Z_r(y)] & \Phi_0 > 0, \\
	\frac {Z_{r}(y)^2} {W_{\p}(y)}  - r\overline{Z}_r(y) & \Phi_0 = 0. \\
	\end{array} \right.
	\end{align*}
	
	(ii) For $b > 0$ and $x \leq b$,
\begin{align*}
%\underline{g}(x,b,\theta)&:=
\E_x\left(e^{-q\tau_b^+(r)-\theta R_r(\tau_b^+(r))} ; \tau_b^+(r) < \infty\right)=
\frac{Z_{\q,\p}(x,\th)}{Z_{\q,\p}(b,\th)}.
\end{align*}
%\green{[$\underline{g}$?]}

(iii) For $b > 0$ and $x \leq b$,
		\begin{align*}
		%\underline{g}(x,b,\infty):=
		\E_x \big(e^{- \q \tau_b^+(r)} ;\tau_b^+(r)<T ^-_0(1) \big) =\frac{Z_q(x,\Phi_{\q+\p})}{Z_q(b,\Phi_{\q+\p})},
		\end{align*}
		which matches the result in \cite{AIrisk} for the case $q = 0$.

\end{corollary}

%\green{[added below a remark on the absolute ruin probability.]}

\begin{remark}  (i) By Corollary \ref{corollary_laplace_tau_a_b_infty} (i), for $a < 0$, we have
	\begin{align*}
	\mathbb{P}_x\left( \tau_a^-(r)<\infty\right) =\mathcal{I}_{0,\p}^a(x)-
	\mathcal{J}_{0,\p}(-a)\frac{\mathcal{H}^a_{0,r}
		(x,0)}{\mathcal{G}_{0,r}(-a,0)} =\mathcal{I}_{0,\p}^a(x)-
	\mathcal{J}_{0,\p}(-a)\frac{1 - \mathcal{I}_{0,\p}^a(x)}{\mathcal{G}_{0,r}(-a,0)}.
	\end{align*}

(a) When $\Phi_0 > 0$ (or $X$ drifts to $-\infty$), then $\mathbb{P}_x\left( \tau_a^-(r)<\infty\right) = 1$.
	
(b) When $\Phi_0 = 0$, we have $\mathcal{J}_{0,\p}(y) = -\mathcal{G}_{0,r}(y,0) + \kappa'(0+)$. Therefore,
		\begin{align*}
	\mathbb{P}_x\left( \tau_a^-(r)<\infty\right)=\mathcal{I}_{0,\p}^a(x)+
	[\mathcal{G}_{0,r}(-a,0) - \kappa'(0+)] \frac{1 - \mathcal{I}_{0,\p}^a(x)}{\mathcal{G}_{0,r}(-a,0)} = 1 - \kappa'(0+) \frac{1 - \mathcal{I}_{0,\p}^a(x)}{\mathcal{G}_{0,r}(-a,0)}.
	\end{align*}
%\green{[may be able to simplify further.]}
	
(ii) By Corollary \ref{corollary_laplace_tau_a_b_infty} (ii), for $b > 0$, we have $\mathbb{P}_x\left( \tau_b^+(r) < \infty\right)=
{Z_{0,\p}(x)} / {Z_{0,\p}(b)} = 1$.
\end{remark}

\subsection{Total discounted bailouts} \label{section_bailout}

We now state our results on the total discounted bailouts.  The first result is for the case killed upon exiting $[a,b]$.
%Our second major result is:
%\green{[remove]We shall assume throughout $\kappa'(0+) > -\infty$ (or the first moment of $X$ exists).}

\begin{theorem}[Total discounted capital injections with killing]\label{expected_bailouts}
%Let $X$ be a spectrally L\'evy process, and fix $b>0$.

For  $a < 0 < b$, $q \geq 0$, and $x \leq b$,
 %\red{[we probably did not show for the case $q = 0$ as we need $\kappa'(0+)/q$ is finite.  We should assume $q > 0$ or otherwise we can take $q \downarrow 0$?]}\blue{Kazu: I think everything will work fine for the case $q=0$, perhaps we can mention it in the proof that we can take $q\to 0$ instead than in the statement of the Theorem?} We have for $a\leq x\leq b$
\begin{align*}
f(x,a,b) := \E_x\left( \int_0^{\t_b^{+}(r)\wedge\t_a^-(r)}e^{-qt}\diff R_{\p}(t)\right) =  \frac{\mathcal{H}_{\q,\p}^{a}(x, 0)}{\mathcal{H}_{\q,\p}^{a}(b, 0)}h_{\q,\p}^{a}(b) -h_{\q,\p}^{a}(x),
 \end{align*}
 %\red{JL: To be consistent with \eqref{ParBail} below, I guess we should write so that it equals
% \begin{align*}
% -\frac r {q+r} \Big(h_{\q,\p}^{a,b}(x)-\frac{Z_{\q,\p}^{a,b}(x)}{Z_{\q,\p}^{a,b}(b)}h_{\q,\p}^{a,b}(b) \Big)
% \end{align*}
% So the sign of $h_{\q,\p}^{a,b}(x)$ need to be flipped?
% }
 where, for $y \in \R$,
 %\red{[remove the superscript $b$ because it does not depend on it? But $Z^a_{q,r}$ is already used.]} \blue{ I agree Kazu, do you have any suggestions?}
\begin{align} \label{small_h_def}
\begin{split}
 	h_{\q,\p}^{a}(y) &:=\frac r {q+r}\left(\overline{Z}_q(y)+\frac{aZ_{\q+\p}(-a)+\overline{Z}_{\q+\p}(-a)}{W_{\q+\p}(-a)}W_{\q,\p}^a(y)-aZ_{\q,\p}^a(y)-\overline{Z}_{\q,\p}^a(y)\right) \\
 	&= \frac r {q+r}\left(\overline{Z}_q(y)+\frac{\overline{Z}_{\q+\p}(-a)}{W_{\q+\p}(-a)}W_{\q,\p}^a(y)-\overline{Z}_{\q,\p}^a(y) - a \mathcal{I}_{\q,\p}^a(y)\right). %\\
 	%\mathcal{J}_{\q,\p}^{a}(x)&:=\frac{\p}{\q+\p}\left(Z_{\q}(x)- \mathcal{I}_{\q,\p}^a(x) \right)+\frac{\q}{\q+\p}\frac{W_{\q,\p}^a(x)}{W_{\q+\p}(-a)}.
	\end{split}
 \end{align}
%\red{JL: I guess $\mathcal{J}_{\q,\p}^{a} = \mathcal{H}_{\q,\p}^{a}(\cdot, 0)$? Then, we can just use $\mathcal{H}_{\q,\p}^{a}(\cdot, 0)$?}

\end{theorem}

%\begin{remark} \green{[this is for confirmation only -- no need to include]}
%Note that
%\begin{align*}
%\frac{aZ_{\q+\p}(-a)+\overline{Z}_{\q+\p}(-a)}{W_{\q+\p}(-a)} \sim \frac{Z_{\q+\p}(-a) + a(q+r) W_{q+r}(-a)-Z_{\q+\p}(-a)}{-W_{\q+\p}'(-a)} \xrightarrow{a \uparrow 0} 0.
%\end{align*}
%\begin{align*}
%&W_{q+r}(-a) \big(h_{\q,\p}^a(b)\mathcal{H}^a_{q,r}(x,0)-h_{\q,\p}^a(x) \mathcal{H}^a_{q,r}(b,0) \big)\\=
%&W_{q+r}(-a) \mathcal{H}^a_{q,r}(x,0) \Big[\overline{Z}_q(b)+\frac{aZ_{\q+\p}(-a)+\overline{Z}_{\q+\p}(-a)}{W_{\q+\p}(-a)}W_{\q,\p}^a(b)-aZ_{\q,\p}^a(b)-\overline{Z}_{\q,\p}^a(b) \Big] \\
%&- W_{q+r}(-a) \mathcal{H}^a_{q,r}(b,0)\Big[\overline{Z}_q(x)+\frac{aZ_{\q+\p}(-a)+\overline{Z}_{\q+\p}(-a)}{W_{\q+\p}(-a)}W_{\q,\p}^a(x)-aZ_{\q,\p}^a(x)-\overline{Z}_{\q,\p}^a(x) \Big] \\
%\xrightarrow{a \uparrow 0} &W_q(x) \Big[\overline{Z}_q(b)-\overline{Z}_{\q}(b) \Big] - W_q(b) \Big[\overline{Z}_q(x)-\overline{Z}_{\q}(x) \Big] = 0.
%\end{align*}
%Hence, by \eqref{limit_fraction_W_H}, we have
% \begin{align*}
%\lim_{a \uparrow 0}f(x,a,b) = \frac r {q+r} \Big( \frac {W_q(x)} {W_q(b)} \Big[\overline{Z}_q(b)-\overline{Z}_{\q}(b) \Big] - \Big[\overline{Z}_q(x)-\overline{Z}_{\q}(x) \Big] \Big)
% \end{align*}
%\end{remark}

By taking $b \uparrow \infty$ and $a \downarrow -\infty$ in Theorem \ref{expected_bailouts}, we have the following.

% \red{JL: I think we need $q > 0$ for (i) as well?}
\begin{corollary} \label{corollary_L_b} %\red{[I guess this is necessary here]
	%Assume $\kappa'(0+) > -\infty$.\\
%	\red{[JL: Flipped the order of (i) and (ii), because our proof for (iii) takes $b \uparrow$ in \eqref{ParBail}]}
	
(i)  Fix $a < 0$.  We have, for $x \in \R$, and either $q>0$ or $q = 0$ and $\Phi_0 > 0$, %case holds true.  Should we add this?]}
\begin{align}\label{ParBail_a_b_infty}
%\underline{f}(x, a):=	
\E_x&\left( \int_0^{\t_a^- (r)}e^{-qt}\diff R_{\p}(t)\right) = \frac{\mathcal{H}_{\q,\p}^{a}(x, 0)}{\mathcal{G}_{\q,\p}(-a, 0)}h_{\q,\p}(-a) -h_{\q,\p}^{a}(x)
\end{align} %\red{change from $h$ to $\blue{\tilde{h}_{q,r}}$?}
where, for $y \in \R$,
%\begin{align*}
%h_{\q,\p}(y):=\frac r {q+r} \left(\frac{q}{\Phi_q^2}+\frac{-yZ_{\q+\p}(y)+\overline{Z}_{\q+\p}(y)}{W_{\q+\p}(y)}Z_{q+r}(y,\Phi_q)+(y\Phi_q-1)\frac{q}{\Phi^2_q}Z_{\q+r,-\p}(y)+\frac{r}{\Phi_q}\overline{Z}_{\q+\p}(y) \right).
%\end{align*}
%
%\red{JL: The form for $q = 0$ and $\Phi_0 > 0$ is slightly different and so we should write?  It should be as in the proof of Corollary 3.1 (i),
%\begin{align*}
%h_{0,\p}(y) &:= \frac{-yZ_{r}(y)+\overline{Z}_{r}(y)}{W_{r}(y)}Z_{r}(y,\Phi_0)+  \frac {y\Phi_0-1} {\Phi_0^2} \lim_{q \downarrow 0} q Z_{\q+r,-\p}(y)+\frac{r}{\Phi_0}\overline{Z}_{r}(y) \\
%&= \frac{-yZ_{r}(y)+\overline{Z}_{r}(y)}{W_{r}(y)}Z_{r}(y,\Phi_0)+  \frac {y\Phi_0-1} {\Phi_0^2} r [Z_r(y, \Phi_0) - Z_r(y)]+\frac{r}{\Phi_0}\overline{Z}_{r}(y)
%\end{align*}
%}
\begin{align*}
h_{\q,\p}(y)&:= \left\{ \begin{array}{ll}\displaystyle\frac r {q+r} \Bigg(\frac{q}{\Phi_q^2}+\frac{-yZ_{\q+\p}(y)+\overline{Z}_{\q+\p}(y)}{W_{\q+\p}(y)}Z_{q+r}(y,\Phi_q)+(y\Phi_q-1)\frac{q}{\Phi^2_q}Z_{\q+r,-\p}(y)&+\frac{r}{\Phi_q}\overline{Z}_{\q+\p}(y) \Bigg)\quad  q > 0, \\ \displaystyle\frac{-yZ_{r}(y)+\overline{Z}_{r}(y)}{W_{r}(y)}Z_{r}(y,\Phi_0)+  \frac {y\Phi_0-1} {\Phi_0^2} r [Z_r(y, \Phi_0) - Z_r(y)]+\frac{r}{\Phi_0}\overline{Z}_{r}(y) & \text{$q=0$ and $\Phi_0 > 0$}.
\end{array} \right.
\end{align*}
(ii)  Fix $b > 0$ and suppose $\kappa'(0+) > -\infty$.  We have, for $x \leq b$ and $q\geq0$, %\red{[how about $\overline{f}$ instead of $f$?]}
\begin{align}\label{ParBail}
%\overline{f}(x, b):=	
\E_x&\left( \int_0^{\t_b^+ (r)}e^{-qt}\diff R_{\p}(t)\right)=  \frac {Z_{\q,\p}( x)}  {Z_{\q,\p}( b)} k_{q,r}(b) - k_{q,r}(x)  ,
\end{align} %\red{change from $h$ to $\blue{\tilde{h}_{q,r}}$?}
where, for $y \in \R$,
\begin{align*}
k_{\q,\p}(y):=\frac{r}{q+r}\left(\overline{Z}_q(y)-\kappa'(0+)\overline{W}_{\q}(y)- \frac{\kappa'(0+)}{\q+\p} \big[Z_{\q}(y,\Phi_{\q+\p})-Z_{\q}(y) \big]\right). %= \tilde{h}_{q,r}(x)   -\frac {r\kappa'(0+)} {q(q+r)} Z_{q,r}(x).
\end{align*}
In particular, when $q>0$, \eqref{ParBail} can be simplified by replacing  $k_{q,r}$ with
 %\green{So, this definition of $\tilde{h}_{q,r}$ is ok right?}\blue{I believe it is.}
\begin{align*} \tilde{h}_{q,r}(y)&:=\frac r {q+r} \Big(\overline{Z}_q(y)+\frac{\k'(0+)}{\q}\Big), \quad y \in \R.
%Z_{\q,\p}(x )&:=Z_{\q,\p}(x,0)= \fr{\p}{\q+\p}  Z_{\q}(x)+ \fr{\q}{\q+\p} Z_{\q}(x,\Phi_{q+\p} ).
\end{align*}

(iii) Suppose $\kappa'(0+) > -\infty$. For $q>0$, we have, for $x \in \R$,
\begin{align}\label{eq:localt0}
%f(x):=
\E_x\le( \int_0^{\I} e^{-q t}\diff R_{\p}(t)\ri)  &=\frac{\Phi_{\q+\p}-\Phi_{\q}}{\Phi_{\q+\p}\Phi_{\q}}Z_{\q,\p}(x)- \tilde{h}_{q,r}(x).
\end{align}

\end{corollary}

\subsection{Identities for the process $Y_r^b$} \label{subsection_reflected_case} %\red{[inserted this]}

%\red{[I guess we need to be careful about the derivative throughout]}

%\red{[moved the definition of $\eta_0^-$ to the appendix because it is not used here.]}

We shall now move onto obtaining the fluctuation identities for the process $Y_r^b$ with additional classical reflection above. In addition to the similar identities obtained above for $X_r$, we shall also obtain the results related to the dividends $L_r^b$.  We shall first obtain the cases killed at the absolute ruin
	 \begin{align*}
 \eta_a^-(r) := \inf \{ t > 0: Y^b_r(t) < a \}, \quad a < 0,
 \end{align*}
 and then obtain the infinite horizon case by taking limits.  Analogously to the classical case, the results can be written concisely using the \emph{derivatives} of the functions defined above for $X_r$.

Let the (right-hand) derivative of $W^{a}_{\q,\p}(x)$ defined in \eqref{W_a_def} be %\blue{
	\begin{align}
	(W^{a}_{\q,\p})'(x) &=\frac{\partial_+}{\partial_+x}\left(W_{q+r}(x-a)- r\int_{0}^{x}  W_{q}(x-y)W_{q+r}(y-a)\diff y\right)\notag\\
&=W_{\q+\p}'((x-a)+)- \p \Big[ \int_{0}^{x}  W_\q'(x-y)W_{\q+\p}(y-a)\diff y + W_q(0) W_{q+r} (x-a) \Big], \quad a < 0, \; x \neq a. \label{W_a_derivative}
	\end{align}%}
%\green{JL:  Actually Can we change the second equality to
%	\begin{align*}
%	W_{\q+\p}'((x-a)+)- \p \Big[ \int_{0}^{x}  W_\q'(x-y)W_{\q+\p}(y-a)\diff y + W_q(0) W_{q+r} (x-a) \Big]. \label{W_a_derivative}
%	\end{align}%}
%	so that $\blue{(W^{a}_{\q,\p})'}(x)$ itself denotes the right-derivative.  Then, we do not need to add $+$ for $\mathcal{H}^{a \prime}_{q,r}$ and $\mathcal{I}^{a \prime}_{q,r}]$  But I don't know.  We could instead add $+$ for everything that includes $W'$?
%}

We first obtain a version of Theorem \ref{theorem_laplace}.

\begin{theorem}[Joint Laplace transform] \label{theorem_h_hat} %\red{[changed to Theorem]}
	Fix $a < 0<b$, $q \geq 0$, and $\theta \geq 0$.  For all $x \leq b$,
	\begin{align}\label{ref_dual_laplace}
	\widehat{h}(x,a,b,\th):=\E_x\left(e^{-q \eta_a^-(r) -\theta R^b_r(\eta_a^-(r))}; \eta_a^-(r)< \infty\right)=\mathcal{I}^a_{q,r}(x)-\frac{\mathcal{H}^a_{q,r}(x,\th)}{\mathcal{H}^{a \prime}_{q,r}(b,\theta)}\mathcal{I}_{\q,\p}^{a \prime} (b),
	\end{align}
where $\mathcal{H}^{a \prime}_{q,r}(y,\theta)$ and $\mathcal{I}_{\q,\p}^{a \prime} (y)$ are the (right-hand) derivatives of $\mathcal{H}^{a}_{q,r}(y,\theta)$ and $\mathcal{I}_{\q,\p}^{a} (y)$ with respect to $y$ given by, for $y \neq a$,
\begin{align*}
\mathcal{H}^{a \prime}_{q,r}(y,\theta)&=\p\int_{0}^{-a} e^{-\theta u}\Big((W^{a }_{\q,\p})'(y)\frac{W_{\q+\p} (u)} {W_{\q+\p}(-a)}-(W^{-u}_{\q,\p})'(y)\Big) \diff u+\frac{(W^{a }_{\q,\p})'(y)}{W_{\q+\p}(-a)}, \\
\mathcal{I}_{\q,\p}^{a \prime}(y)&=(q+r)W^a_{\q,\p}(y)-Z_{\q+\p}(-a)\left(rW_q(y)+\frac{(W^{a }_{\q,\p})'(y)}{W_{\q+\p}(-a)}\right).
\end{align*}
%\green{[moved here]
Here, in particular,
				\begin{align}
				\label{H_big_zero}
				\begin{split}
\mathcal{H}^{a \prime}_{q,r}(y,0)% = \blue{\blue{(W^{a}_{\q,\p})'}(x)}\frac{r \overline{W}_{\q+\p} (-a)+1} {W_{\q+\p}(-a)} - r W^a_{q,r} (x) +\frac r {q+r} W_q(x) \Big[  r  Z_{q+r} (-a) + q \Big] \\
&= \frac 1 {q+r}(W^{a}_{\q,\p})'(y)\frac{r Z_{\q+\p} (-a)+q} {W_{\q+\p}(-a)} - r W^a_{q,r} (y) +\frac r {q+r} W_q(y) \Big[  r  Z_{q+r} (-a) + q \Big] \\
&= \frac 1 {q+r} \Big[ q \Big(\frac{(W^{a }_{\q,\p})'(y)}{W_{q+r}(-a)}+r W_q(y) \Big) -r\mathcal{I}^{a \prime}_{q,r}(y) \Big].
\end{split}
\end{align}
%	where
%	\begin{align*}
%	\widehat{\mathcal{H}}^a_{q,r}(x,\theta)&:=\p\int_{0}^{-a} e^{-\theta u}\Bigg[\blue{\blue{(W^{a}_{\q,\p})'}(x)}\frac{W_{\q+\p} (u)} {W_{\q+\p}(-a)}-\widehat{W}^{-u}_{\q,\p}(x)\Bigg]\diff u+\frac{\blue{\blue{(W^{a}_{\q,\p})'}(x)}}{W_{\q+\p}(-a)}, \\
	%\widehat{\mathcal{I}}_{\q,\p}^a(x)&:=(q+r)W^a_{\q,\p}(x)-rW_q(\red{x})Z_{q+r}(-a)-\widehat{W}^a_{\q,\p%}(x)\frac{Z_{\q+\p}(-a)}{W_{\q+\p}(-a)}.
	%\end{align*}
	%\red{JL:  Also, we have $\widehat{\mathcal{H}}^a_{q,r}(x,\theta) = \frac \partial {\partial x}\mathcal{H}^a_{q,r}(x,\theta)$.  Also,
	%\begin{align*}
	%\frac \partial {\partial x}\mathcal{I}_{\q,\p}^a(x)&:= \Big(Z_{\q+\p}(x-a)- r \int_0^x W_q(x-y)  Z_{\q+\p}(y-a)  \diff y \Big)'-\blue{\blue{(W^{a}_{\q,\p})'}(x)}\frac{Z_{\q+\p}(-a)}{W_{\q+\p}(-a)}  \\
	%&=(q+r) W_{q+r}(x-a)- r \int_0^x W_q'(x-y)  Z_{\q+\p}(y-a)  \diff y - r W_q(0) Z_{q+r} (x-a) -\blue{\blue{(W^{a}_{\q,\p})'}(x)}\frac{Z_{\q+\p}(-a)}{W_{\q+\p}(-a)} \\
		%&=(q+r) W_{q+r}(x-a)- r (q+r) \int_0^x W_q'(x-y)  W_{\q+\p}(y-a)  \diff y - r W_q(x) Z_{q+r} (-a) -\blue{\blue{(W^{a}_{\q,\p})'}(x)}\frac{Z_{\q+\p}(-a)}{W_{\q+\p}(-a)} \\
		%= \widehat{\mathcal{I}}_{\q,\p}^a(x)
	%\end{align*}
	%}
\end{theorem}

%\green{[added this remark]
\begin{remark}  By Theorem \ref{theorem_h_hat}, for all $a < 0 < b$ and $x \leq b$,
	\begin{align*}
	\mathbb{P}_x\left( \eta_a^-(r)<\infty\right) =\mathcal{I}_{0,\p}^a(x)-
	\frac{\mathcal{H}^a_{0,r}(x,0)}{\mathcal{H}^{a \prime}_{0,r}(b,0)}\mathcal{I}_{\q,\p}^{a \prime} (b) =\mathcal{I}_{0,\p}^a(x)-
	[1 - \mathcal{I}_{0,\p}^a(x)] \frac{\mathcal{I}_{0,\p}^{a \prime} (b)}{\mathcal{H}^{a \prime}_{0,r}(b,0)} = 1.
	\end{align*}
\end{remark}

Second, we obtain the total expected discounted dividends until the absolute ruin $\eta_a^-(r)$.  As its corollary, by taking $a \downarrow -\infty$, we also obtain the infinite horizon case.

\begin{theorem}[Total discounted dividends with killing] \label{theorem_R_b_r}  %\red{[JL: I added the killed case]}
For $a < 0 < b$ and $q \geq 0$, we have \begin{align*}
		\widehat{j}(x,a, b) &:= \E_x\left( \int_{0}^{\eta_a^-(r)}e^{-qt} \diff L^b_r(t)\right) = \left\{ \begin{array}{ll} \mathcal{H}^a_{q,r}(x,0) / \mathcal{H}^{a \prime}_{q,r}(b,0) & x \leq b, \\ \mathcal{H}^a_{q,r}(b,0) / \mathcal{H}^{a \prime}_{q,r}(b,0) + (x-b) & x > b. \end{array} \right.
		\end{align*}
\end{theorem}

\begin{remark} %\green{[Is it ok if include a short summary of identities for the reflected process when we introduce the scale function?  It is used here and later.]}\blue{Sure Kazu.} \green{JL: Sorry, I actually cannot find an appropriate place to put, so I just add here that it matches the expression in APP.}
We can confirm that these expressions in Theorems \ref{theorem_h_hat} and \ref{theorem_R_b_r} converge to the expressions given in (3.10) and (3.12) of \cite{APP}: \begin{align*}
\lim_{a \uparrow 0}  \widehat{h}(x,a,b,\th) &=Z_q(x)- q \frac {W_q(b)} {W_q'(b+)} W_q(x) = \E_x\left(e^{-q \eta_0^- }\right), \\
\lim_{a \uparrow 0}\widehat{j}(x,a, b) &= \left\{ \begin{array}{ll} W_q(x) / {W_q'(b+)} & x \leq b \\ W_q(b) / {W'_q(b+)} + (x-b) & x > b \end{array} \right\} = \E_x\left( \int_{0}^{\eta_0^-}e^{-qt} \diff L^b(t)\right),
		\end{align*}
		%\red{[moved this here -- it was defined much later.] 
		where
 \begin{align*}
 \eta_0^- := \inf \{ t > 0: Y^b(t) < 0 \}.
 \end{align*}
Indeed, the latter holds because, similarly to \eqref{limit_fraction_W_H}, $\lim_{a\uparrow 0}\mathcal{H}^{a \prime}_{q,r}(x,\theta)W_{q+r}(-a)=W_q'(x+)$.
On the other hand, some algebra gives
%\begin{align*}
%&(q+r) \Big(\mathcal{I}_{\q,\p}^{a \prime}(b)\mathcal{H}^a_{q,r}(x,\theta)-\mathcal{I}_{\q,\p}^a(x)\mathcal{H}^{a \prime}_{q,r}(b,\theta) \Big) \\&= \mathcal{I}_{\q,\p}^{a \prime}(b) \Big[ \Big( \frac {W^a_{\q,\p}(x)} {W_{\q+\p}(-a)} \Big( r Z_{q+r}(-a) + q\Big) - r \Big[ Z_{q,r}^a (x) - Z_{q}(x)\Big] \Big) \Big] \\
%&-\Big[ Z^a_{\q,\p}(x)-W^a_{\q,\p}(x)\frac{Z_{\q+\p}(-a)}{W_{\q+\p}(-a)} \Big] \Big[ q \Big(\frac{\blue{(W^{a }_{\q,\p})'}(b)}{W_{q+r}(-a)}+r W_q(b) \Big) -r\mathcal{I}^{a \prime}_{q,r}(b) \Big]
%\\&= \mathcal{I}_{\q,\p}^{a \prime}(b) \Big[ \Big( \frac {W^a_{\q,\p}(x)} {W_{\q+\p}(-a)} \Big( r Z_{q+r}(-a) + q\Big) - r \Big[ Z_{q,r}^a (x) - Z_{q}(x)\Big] \Big) \Big] \\
%&- Z^a_{\q,\p}(x)\Big[ q \Big(\frac{\blue{(W^{a }_{\q,\p})'}(b)}{W_{q+r}(-a)}+r W_q(b) \Big) -r\mathcal{I}^{a \prime}_{q,r}(b) \Big]
%+ W^a_{\q,\p}(x)\frac{Z_{\q+\p}(-a)}{W_{\q+\p}(-a)}  \Big[ q \Big(\frac{\blue{(W^{a }_{\q,\p})'}(b)}{W_{q+r}(-a)}+r W_q(b) \Big) -r\mathcal{I}^{a \prime}_{q,r}(b) \Big]
%\\&= \mathcal{I}_{\q,\p}^{a \prime}(b) \Big[ \Big( \frac {W^a_{\q,\p}(x)} {W_{\q+\p}(-a)} \Big( q\Big) + r \Big[  Z_{q}(x)\Big] \Big) \Big] \\
%&- Z^a_{\q,\p}(x)\Big[ q \Big(\frac{\blue{(W^{a }_{\q,\p})'}(b)}{W_{q+r}(-a)}+r W_q(b) \Big) \Big]
%+ W^a_{\q,\p}(x)\frac{Z_{\q+\p}(-a)}{W_{\q+\p}(-a)}  \Big[ q \Big(\frac{\blue{(W^{a }_{\q,\p})'}(b)}{W_{q+r}(-a)}+r W_q(b) \Big) \Big].
%\end{align*}
\begin{align*}
\mathcal{I}_{\q,\p}^a(x)\mathcal{H}^{a \prime}_{q,r}(b,\th) &-\mathcal{I}_{\q,\p}^{a \prime}(b)\mathcal{H}^a_{q,r}(x,\th)
=  Z_{q,r}^a(x)\mathcal{H}^{a\prime}_{q,r}(b,\theta) -(q+r)W_{q,r}^a(b)\mathcal{H}^a_{q,r}(x,\theta)\\
&+ r\frac{Z_{q+r} (-a)}{W_{q+r}(-a)} \Big( W^a_{\q,\p}(x) \int_0^{-a}e^{-\th u}(W_{q,r}^{-u})'(b) \diff u -(W^a_{\q,\p})'(b) \int_0^{-a}e^{-\th u}W_{q,r}^{-u}(x)\diff u \Big)\\
&+rZ_{q+r}(-a)W_q(b)\frac{W_{q,r}^a(x)}{W_{q+r}(-a)}.
\end{align*}
Therefore, %\red{[below I guess $0$ can be $\theta$?]}
\begin{align*}
W_{q+r}(-a)\Big(\mathcal{I}_{\q,\p}^a(x)\mathcal{H}^{a \prime}_{q,r}(b,\th) -\mathcal{I}_{\q,\p}^{a \prime}(b)\mathcal{H}^a_{q,r}(x,\th) \Big)
\xrightarrow{a \uparrow 0} W'_{\q}(b +)   Z_{q}(x) - q W_q(b)  W_{\q}(x).
\end{align*}
Hence, the former holds as well.
%\blue{[Kazu: Maybe like this?]
%	\begin{align*}
%	W_{q+r}(-a)\Big(\mathcal{I}_{\q,\p}^a(x)\mathcal{H}^{a \prime}_{q,r}(b,0) -\mathcal{I}_{\q,\p}^{a \prime}(b)\mathcal{H}^a_{q,r}(x,0) \Big)
%	\xrightarrow{a \uparrow 0} W'_{\q}(b +)   Z_{q}(x) - q W_q(b)  W_{\q}(x).
%	\end{align*}
%	}
%	\red{[Yea, I guess we can just focus on the case $\theta = 0$ but then I think we should do the same in Remark 3.1?]}
\end{remark}

\begin{corollary}  \label{corollary_R_b_r}
For $b > 0$ and $q > 0$, %\green{[I guess we need $q > 0$]}
we have
\begin{align*}
%\widehat{j}(x,b):=
\E_x\left(\int_{0}^{\infty}e^{-qt} \diff L^b_r(t)\right)= \left\{ \begin{array}{ll} Z_{q,r}(x) /Z_{q,r}'(b)  & x \leq b, \\ Z_{q,r}(b) /Z_{q,r}'(b) + (x-b) & x > b.
\end{array} \right.
\end{align*}
Here, note that $Z_{q,r}'(y)= q \Phi_{q+r} Z_q(y, \Phi_{q+r}) / (q+r) $ for $y \in \R$.
\end{corollary}

Third, as a version of Theorem \ref{expected_bailouts}, we obtain the total expected discounted capital injections until the absolute ruin $\eta_a^-(r)$.  Again, as its corollary, by taking $a \downarrow -\infty$, we also obtain the infinite horizon case.

\begin{theorem}[Total discounted capital injections with killing] \label{theorem_L_b_r} %\red{JL: I added the killed case}
Suppose %\green{[remove?]$\kappa'(0+) > -\infty$},
$q \geq 0$ and $a < 0 < b$.  We have
 \begin{align*}
\widehat{f}(x,a,b) := \E_x\left(\int_0^{\eta_a^-(r)}e^{-qt}\diff R_{\p}^b(t)\right)= \left\{ \begin{array}{ll}\frac{\mathcal{H}_{\q,\p}^{a}(x, 0)}{\mathcal{H}_{\q,\p}^{a\prime}(b, 0)}h_{\q,\p}^{a \prime}(b) -h_{\q,\p}^{a}(x) & x \leq b, \\
\frac{\mathcal{H}_{\q,\p}^{a}(b, 0)}{\mathcal{H}_{\q,\p}^{a\prime}(b, 0)}h_{\q,\p}^{a \prime}(b) -h_{\q,\p}^{a}(b) & x > b, \end{array} \right.
 \end{align*}
% \green{[denominator is missing a derivative?]}
where $h^{a \prime}_{q,r}(y)$ is the (right-hand) derivatives of $h^{a}_{q,r}(y)$ with respect to $y$ given by, for $y\neq a$,
\begin{multline*}
h^{a \prime}_{q,r}(y) =  \frac r {q+r} \Big[Z_q(y) + \frac{aZ_{\q+\p}(-a)+\overline{Z}_{\q+\p}(-a)}{W_{\q+\p}(-a)} (W^{a}_{\q,\p})'(y) \notag\\
-a\left[(q+r)W^{a}_{\q,\p}(y)-rW_q(y)Z_{q+r}(-a)\right]-
Z_{\q,\p}^a(y)+ rW_q(y)\overline{Z}_{q+r}(-a)\Big].
\end{multline*}
\end{theorem}

\begin{corollary} \label{corollary_L_b_r}
Suppose $\kappa'(0+) > -\infty$, $q >0$, and $b > 0$.
We have
\begin{align*}
%\hat{f}(x,b):=
\E_x\left( \int_{0}^{\infty}e^{-qt} \diff R^b_r(t)\right) = \left\{ \begin{array}{ll}\displaystyle \frac{r}{q+r}\frac{Z_{q,r}(x)}{Z_{q,r}'(b)}Z_q(b)- \tilde{h}_{q,r}(x) & x \leq b, \\ \displaystyle \frac{r}{q+r}\frac{Z_{q,r}(b)}{Z_{q,r}'(b)}Z_q(b)-\tilde{h}_{q,r}(b)  & x > b. \end{array} \right.
\end{align*}
%\red{JL: I guess $\blue{\tilde{h}_{q,r}}'(x)=Z_q(x)$? So we can just use $\blue{Z_q}$?}

\end{corollary}

\section{Proofs for the bounded variation case} \label{section_bounded}

In this section, we show Theorems \ref{theorem_laplace} and  \ref{expected_bailouts} for the case $X$ is of bounded variation. The proofs are direct applications of Lemma 2.1 of \cite{LRZ}.  We shall first review their results and use them to compute the expected values \eqref{def_big_U} defined below, in terms of the scale functions given in \eqref{W_a_def}. These together with an application of the strong Markov property complete the proofs.  For the proofs of their corollaries (both for bounded and unbounded variation cases), see Appendix \ref{appendix_proofs_regular}.

\subsection{Auxiliary results}

Fix $q \geq 0$. Let $\mathbf{e}_{\p}$ be the exponential random variable with parameter $r$ independent of $X$.
Define, for $a < 0$ and $x \leq 0$,
\begin{align*}
u_1(x, a, \theta) &:= \E_{x}
\left(e^{-q\mathbf{e}_{\p}+\theta X(\mathbf{e}_{\p})};
\mathbf{e}_{\p}<\tau_0^+\wedge\tau_a^-\right),  \quad \theta \geq 0, \\
u_2(x, a) &:= \E_{x} \left(e^{-q\tau_0^+};
\tau_0^+ < \mathbf{e}_{\p}\wedge\tau_a^-\right), \\
u_3(x,a) &:= \E_{x}\left(e^{-q\tau_a^-};\tau_a^-<\mathbf{e}_{\p}\wedge\tau_0^+\right), \\
u_4 (x, a) &:= \E_{x}
\left(e^{-q\mathbf{e}_{\p}}X(\mathbf{e}_{\p});
\mathbf{e}_{\p}<\tau_0^+\wedge\tau_a^-\right).
\end{align*}
Also, define for  $a < 0$ and $x \in \R$,
%\begin{align*}
%\tilde{u}_1(x, a, \theta) &:=  r \Big[ W_{q+r}(x-a) \frac {\int_0^{-a} e^{-\theta z} W_{q+r}(z) \diff z} {W_{q+r}(-a)} - e^{\theta x}   \int_0^{x-a} e^{-\theta z} W_{q+r}(z) \diff z \Big], \quad \theta \geq 0,
%\end{align*}
%\red{[I guess we should change this to, in order to handle for the case $x > 0$]
	\begin{align}
	\tilde{u}_1(x,a,\theta):=r \int_{0}^{-a} e^{-\theta u }\Big[ \frac {W_{\q+\p}(x-a) W_{\q+\p} (u)} {W_{\q+\p}(-a)} -W_{\q+\p} (x+u) \Big] \diff u, \label{u_1_tilde}
	\end{align}
where in particular
%\begin{align}
%\tilde{u}_1(x, a, 0) %&=\frac{\p}{\q+\p}\E_x\left(1-e^{-(\q+\p)(\tau_0^+\wedge\tau_a^-)}\right)\\
%&=\frac{\p}{\q+\p}\left(1-\E_x(e^{-(\q+\p)\tau_0^+};\tau_0^+<\tau_a^-)-\E_x(e^{-(\q+\p)\tau_a^-};\tau_a^-<\tau_0^+)\right)\\
%&=\frac{\p}{\q+\p}\left[1-\frac{W_{\q+\p}(x-a)}{W_{\q+\p}(-a)}-\left(Z_{\q+\p}(x-a)-\frac{W_{\q+\p}(x-a)}{%W_{\q+\p}(-a)}Z_{\q+\p}(-a)\right)\right], \label{u_1_tilde_zero}
%\end{align}
%\red{this is wrong above $x > 0$, it should be
\begin{align}
\tilde{u}_1(x, a, 0)
&=r \frac {W_{q+r}(x-a)} {W_{q+r}(-a)} \overline{W}_{q+r} (-a) - \frac r {q+r} [Z_{q+r}(x-a) - Z_{q+r}(x)],\label{u_1_tilde_zero}
\end{align}
%}
and
\begin{align*}
\tilde{u}_2(x, a) &:= \frac{W_{q+\p}(x-a)}{W_{q+\p}(-a)}, \\
\tilde{u}_3(x,a) &:= Z_{\q+\p}(x-a)-W_{\q+\p}(x-a)\frac{Z_{\q+\p}(-a)}{W_{\q+\p}(-a)}.
%\tilde{u}_4 (x, a)
%&:= \frac r {q+r}  \Big[ \overline{Z}_{q+r}(x)+ \frac {W_{q+r}(x-a)} {W_{q+r}(-a)} \Big( a Z_{q+r} (-a) + {\overline{Z}_{q+r}(-a)}  \Big) -   a Z_{q+r} (x-a) -   \overline{Z}_{q+r} (x-a) \Big].
\end{align*}
%\green{[I guess $\tilde{u}_4$ is no longer needed?]}
%\blue{Kazu: For everything to work we need $\overline{Z}_{q}(x)$ instead of $\overline{Z}_{\red{q+r}}(x)$ above.}
In particular, $\tilde{u}_1(x, a, \theta) = \tilde{u}_2(x,a) = 0$ and $\tilde{u}_3(x,a) = 1$ for  $x < a$.

We show below that $u_i$ and $\tilde{u}_i$ match for $x \leq 0$.

\begin{lemma} \label{lemma_u}  %\blue{[put back the proof of this Lemma from the appendix.]}
For $q \geq 0$, $a < 0$ and $x \leq 0$,
 we have $u_1(x, a, \theta) = \tilde{u}_1(x, a, \theta)$ for $\theta \geq 0$,
%\begin{align*}
%u_1(x, a, \theta) &:= \E_{x}
%\left(e^{-q\mathbf{e}_{\p}+\theta X(\mathbf{e}_{\p})};
%\mathbf{e}_{\p}<\tau_0^+\wedge\tau_a^-\right) \\ &=  r \Big[ W_{q+r}(x-a) \frac {\int_0^{-a} e^{-\theta z} W_{q+r}(z) \diff z} {W_{q+r}(-a)} - e^{\theta x}   \int_0^{x-a} e^{-\theta z} W_{q+r}(z) \diff z \Big].
%\end{align*}
%In particular, when $\theta = 0$,
%\begin{align*}
%u_1(x, a, 0) %&=\frac{\p}{\q+\p}\E_x\left(1-e^{-(\q+\p)(\tau_0^+\wedge\tau_a^-)}\right)\\
%%&=\frac{\p}{\q+\p}\left(1-\E_x(e^{-(\q+\p)\tau_0^+};\tau_0^+<\tau_a^-)-\E_x(e^{-(\q+\p)\tau_a^-};\tau_a^-<\tau_0^+)\right)\\
%&=\frac{\p}{\q+\p}\left[1-\frac{W_{\q+\p}(x-a)}{W_{\q+\p}(-a)}-\left(Z_{\q+\p}(x-a)-\frac{W_{\q+\p}(x-a)}{W_{\q+\p}(-a)}Z_{\q+\p}(-a)\right)\right].
%\end{align*}
and $u_i(x, a) = \tilde{u}_i(x, a)$ for $i = 2,3$.
%\begin{align*}
%u_2(x, a) &:= \E_{x} \left(e^{-q\tau_0^+};
%\mathbf{e}_{\p}\wedge\tau_a^->\tau_0^+\right)=\frac{W_{q+\p}(x-a)}{W_{q+\p}(-a)}.
%\end{align*}

%(iii) For $a, x \leq 0$, we have $u_3(x, a, \theta) = \tilde{u}_3(x, a, \theta)$.
%\begin{align*}
%u_3(x,a) &:= \E_{x}\left(e^{-q\tau_a^-};\tau_a^-<\mathbf{e}_{\p}\wedge\tau_0^+\right)=Z_{\q+\p}(x-a)-W_{\q+\p}(x-a)\frac{Z_{\q+\p}(-a)}{W_{\q+\p}(-a)}.
%\end{align*}

%(iv) For $a, x \leq 0$, we have $u_4(x, a, \theta) = \tilde{u}_4(x, a, \theta)$.
%\begin{align*}
%u_4 (x, a) &:= \E_{x}
%\left(e^{-q\mathbf{e}_{\p}}X(\mathbf{e}_{\p});
%\mathbf{e}_{\p}<\tau_0^+\wedge\tau_a^-\right) \\%= r \E_{x}\left(\int_0^{\tau_0^+ \wedge \tau_a^-}e^{-(q+\p)s}X(s)\diff s\right) \\
%&= \frac r {q+r}  \Big[ x+ \frac {W_{q+r}(x-a)} {W_{q+r}(-a)} \Big( a Z_{q+r} (-a) + {\overline{Z}_{q+r}(-a)}  \Big) -   a Z_{q+r} (x-a) -   \overline{Z}_{q+r} (x-a) \Big].
%\end{align*}
%\blue{[Kazu: Wouldn't it help when we mae references to itemize each of the identities? ]}
\end{lemma}
%\green{[It may be a good idea to put the proof of this lemma back here from the appendix. Currently, the other proofs in the appendix are only about taking limits.]}
\begin{proof} %(Proof of Lemma 4.1).\\
(i) By \eqref{resolvent_density},
\begin{align*}
u_1(x, a, \theta) &= r \E_{x-a}\left(\int_0^{\tau_0^-\wedge\tau_{-a}^+}e^{-(q+\p)s+\theta (X(s)+a)}\diff s\right)\\
&=r \int_{0}^{-a} e^{\theta(y+a)}\Big[ \frac {W_{\q+\p}(x-a) W_{\q+\p} (-a-y)} {W_{\q+\p}(-a)} -W_{\q+\p} (x-a-y) \Big] \diff y  = \tilde{u}_1(x,a, \theta).
%&= r \Big[  - e^{\theta x}   \int_0^{x-a} e^{-\theta z} W_{q+r}(z) \diff z + W_{q+r}(x-a) \frac {\int_0^{-a} e^{-\theta z} W_{q+r}(z) \diff z} {W_{q+r}(-a)} \Big].
\end{align*}
By setting $\th=0$ in the above expression we can easily obtain \eqref{u_1_tilde_zero}.
%\green{[we can safely remove this because setting $\theta = 0$, we can easily obtain the result?]The case $\theta = 0$ holds by noticing that
%\begin{align*}
%u_1(x, a, 0) % &=\frac{\p}{\q+\p}\E_x\left(1-e^{-(\q+\p)(\tau_0^+\wedge\tau_a^-)}\right)\\
%&=\frac{\p}{\q+\p}\left(1-\E_x(e^{-(\q+\p)\tau_0^+};\tau_0^+<\tau_a^-)-\E_x(e^{-(\q+\p)\tau_a^-};\tau_a^-<\tau_0^+)\right).
%&=\frac{\p}{\q+\p}\left(1-\frac{W_{\q+\p}(x-a)}{W_{\q+\p}(-a)}-\left(Z_{\q+\p}(x-a)-\frac{W_{\q+\p}(x-a)}{W_{\q+\p}(-a)}Z_{\q+\p}(-a)\right)\right).
%\end{align*}}

(ii), (iii)  For $i = 2,3$, the results are clear by following \eqref{laplace_in_terms_of_z} and noticing that $u_2(x,a) =\E_{x}
(e^{-(q+\p)\tau_0^+};\tau_a^->\tau_0^+)$ and $u_3(x,a) =\E_{x} (e^{-(q+\p)\tau_a^-};\tau_a^-<\tau_0^+)$.
\end{proof}
%\red{[How about changing $v^{(q)}(x)$ to $v_q(x)$?]}
%\red{[We only care about the case $a=0$, and so we change $a$ to $0$?]}\blue{I am not sure, because when we scale functions }

As in \cite{LRZ}, for any $p \geq 0$,  let $\mathcal{V}_0^{(p)}$ be the set of measurable functions $v_{p}: \R \to [0, \infty)$ satisfying
\begin{align*}
\E_x \big(e^{-p\tau_0^-} v_{p} (X(\tau_0^-)); \tau_0^- < \tau_b^+ \big) = v_{p} (x) - \frac {W_{p}(x)} {W_{p}(b)} v_{p} (b), \quad x \leq b.
\end{align*}
We shall further define $\tilde{\mathcal{V}}_0^{(p)}$ to be the set of positive measurable functions $v_{p}(x)$ that satisfy conditions (i) or (ii) in Lemma 2.1 of \cite{LRZ}, which state as follows: %\green{[added this for reader's convenience.]}
\begin{enumerate}
\item[(i)] For the case $X$ is of bounded variation,  $v_p \in \mathcal{V}_0^{(p)}$ and there exists large enough $\lambda$ such that $\int_0^\infty e^{- \lambda z} v_p (z) \diff z < \infty$.
\item[(ii)] For the case $X$ is of unbounded variation, there exist a sequence of functions $v_{p,n}$ that converge to $v_p$ uniformly on compact sets, where $v_{p,n}$ belongs to the class $\tilde{\mathcal{V}}_0^{(p)}$ for the process $X^n$; here $(X^n; n \geq 1)$ is a sequence of spectrally negative \lev processes of bounded variation that converge to $X$ almost surely uniformly on compact time intervals (which can be chosen as in, for example, page 210 of \cite{B}).
\end{enumerate}
%\red{[I guess we should not say $\tilde{\mathcal{V}}_0^{(q)}$ is a subset because (ii) of Lemma 2.1 does not say $v^{(q)} \in \mathcal{V}_0$?  It may be better to write more specifically Lemma 2.1 (i), (ii).]}
Lemma 2.1 of \cite{LRZ} shows that, for all $p, q \geq 0$, $v_{p} \in \tilde{\mathcal{V}}_0^{(p)}$ and $x \leq b$,
\begin{multline} \label{LRZ_identity}
\E_x \big(e^{-q \tau_0^-} v_{p} (X(\tau_0^-)); \tau_0^- < \tau_b^+ \big) \\= v_{p} (x) - (p-q) \int_0^x W_{q} (x-y) v_{p}(y) \diff y - \frac {W_{q}(x)} {W_{q}(b)} \Big( v_{p}(b) - (p-q) \int_0^b W_{q}(b-y) v_{p} (y) \diff y\Big).
\end{multline}

%\green{[This paragraph is no longer needed and remove?  Hence I guess we can stick to this $\tilde{h}$ (we don't need to write $\tilde{h}_{\q+\p,r(q+r)/q}$]}
Fix any $a < 0$.  By Lemma 2.2 of \cite{LRZ} and spatial homogeneity, %\red{[added parenthesis below]}
\begin{align}
(y \mapsto W_{q+r}(y-a)) \in \tilde{\mathcal{V}}_0^{({q+r})} \quad \textrm{and} \quad (y \mapsto Z_{q+r}(y-a)) \in \tilde{\mathcal{V}}_0^{({q+r})}. \label{W_Z_in_V}
\end{align}
%  In addition, if $\kappa'(0+) > -\infty$, by Proposition \green{TODO} in \cite{APP}, the process
% $\{ e^{-(\q+\p)(t\wedge\tau_a^{-}\wedge\tau_b^+)}\blue{\tilde{h}_{\q+\p,r(q+r)/q}}(X(t\wedge\tau_a^{-}\wedge\tau_b^+)),t \geq 0 \}$
%is a martingale. Hence by Remark 2.1 of \cite{LRZ},  we have
%\begin{align}
%y\mapsto \blue{\tilde{h}_{\q+\p,r(q+r)/q}}(y-a)\in\tilde{\mathcal{V}}^{({q+r})}_0.
%\label{Z_bar_in_V}
%\end{align}
%\green{JL: Sorry, looking at this, now I have a second thought about using $\tilde{h}_{q,r}$.  $\tilde{h}_{\q+\p,r(q+r)/q}$ is too hard to see and not so intuitive,  How about just not using a short-hand notation?
%}\blue{Hi Kazu: Yes I agree, I didn't like it much either, but with the chamge of notation its hard to write the proof. Do you have any suggestions?}

%\green{[Before working on the notation, let us make this more precise.  $\tilde{h}$ violates the assumption of LRZ, and so instead of showing $\tilde{u}_4 \in \tilde{V}_0^{(q+r)}$ we should take the derivative of $U_1$ with respect to $\theta$ and take $\theta \rightarrow 0$.] }
%\begin{align}
%y \mapsto  \overline{Z}_{\q+\p}(y-a) + \kappa'(0+)/(q+r) \in \tilde{\mathcal{V}}^{({q+r})}_0. \label{Z_bar_in_V}
%\end{align}
%\red{use $\blue{\tilde{h}_{q,r}}$ above to simplify?}
\begin{lemma} \label{lemma_u_in_V}    For any $q \geq 0$ and $a < 0$, we have %\red{[added parenthesis before $\in$] 
	$(y \mapsto \tilde{u}_1(y, a, \theta)) \in \tilde{\mathcal{V}}_0^{({q+r})}$ for $\theta \geq 0$, and $(y \mapsto \tilde{u}_i(y, a)) \in \tilde{\mathcal{V}}_0^{({q+r})}$ for $i=2,3$. %\green{[remove this?] If $\kappa'(0+) > -\infty$, then $y \mapsto \tilde{u}_4(y, a)\in \tilde{\mathcal{V}}_0^{({q+r})}$. } %\red{[I guess for $\tilde{u}_4$, we need the assumption that $\kappa'(0+) > -\infty$?]}
\end{lemma}
\begin{proof}
 Recall as in  \cite{LRZ} that $\tilde{\mathcal{V}}_0^{({q+r})}$ is a linear space.
%
%$\tilde{\mathcal{V}}_0^{(q)}$ is a linear space, $u_i \in \tilde{\mathcal{V}}_0^{(q)}$ and hence the following is immediate.
%We write
	%\[
	%\tilde{u}_1(x,a,\theta)=r \int_{0}^{-a} e^{\theta(y+a)}\Big[ \frac {W_{\q+\p}(x-a) W_{\q+\p} (-a-y)} %{W_{\q+\p}(-a)} -W_{\q+\p} (x-a-y) \Big] \diff y.
%	\]
%\green{[This is exactly the same as \eqref{u_1_tilde} and so we can just point to \eqref{u_1_tilde}?  %Also, to be consistent with other parts, we could change $-a-y$ to $u$?]}
Hence using \eqref{u_1_tilde} and the fact that $(y \mapsto W_{\q+\p} (y+u)) \in \tilde{\mathcal{V}}_0^{({q+r})}$ for all $u > 0$, % and Lemma 2.1 of  \cite{LRZ} (which gives \eqref{LRZ_identity}), \red{[remove ``and Lemma 2.1 of  \cite{LRZ} (which gives \eqref{LRZ_identity})"? because we are just saying that it is in this class?]} 
 we have $(y \mapsto \tilde{u}_1(y, a, \theta)) \in \tilde{\mathcal{V}}_0^{({q+r})}$.

The proofs for $\tilde{u}_i(\cdot, a)$ for $i=2,3$ hold because these are linear combinations of \eqref{W_Z_in_V}.

\end{proof}

%[JL: I don't understand how this implies it is in $\mathcal{V}^{(\q)}_a$? A more direct proof may be that $e^{-(\q+\p)t}\overline{Z}_{\q+\p}(X_t) + \kappa'(0+)/(q+r)$ is a martingale (which can be found in \cite{APP}). Also, it is in $\tilde{\mathcal{V}}_a^{(q)}$ because $\overline{Z}$ is an integral of $Z$? ]
%\blue{[Kazu: The idea was to express $\overline{Z}$]}

%Hence, using Lemma 2.1 in \cite{LRZ}, we obtain the following result. %\red{[moved here.  Maybe we can shorten the proof or move it to the appendix?]}
%%\red{[Define $\overline{Z}_{\q,\p}^a$ earlier along with $W_{\q,\p}^a$?]}
%\begin{lemma}
%	\begin{align*}
%\E_x&\left(e^{-q\tau_0^-}\overline{Z}_{\q+\p}(X(\tau_0^-)-a);\tau_0^-<\tau_b^+\right)=\overline{Z}_{\q,\p}^a(x)-\kappa'(0+)\overline{W}_{\q}(x)-\frac{W_{\q}(x)}{W_{\q}(b)}\left(\overline{Z}_{\q,\p}^a(b)-\kappa'(0+) \overline{W}_{\q}(b)\right).
%\end{align*}
%\blue{Hi Kazu: With the argument given above I guess we are not using this lemma anymore?}
%%\green{We should write like this?
%%\begin{align*}
%%\overline{Z}_{\q,\p}^a(x)-\kappa'(0+)\overline{W}_{\q}(x)-\frac{W_{\q}(x)}{W_{\q}(b)}\left(\overline{Z}_{\q,\p}^a(b)-\kappa'(0+) \overline{W}_{\q}(b)\right)
%%\end{align*}}
%\end{lemma}
%By Lemma 2.1 in \cite{LRZ},
Using these identities, we shall now compute, for $a < 0 < b$ and $x \leq b$,
\begin{align} \label{def_big_U}
\begin{split}
U_1(x, a,b,\theta) &:= \E_x\Big(e^{-q\tau_0^-} u_1 (X(\tau_0^-), a, \theta);\tau_0^-<\t_b^+\Big), \quad \theta \geq 0, \\
U_i(x, a,b) &:= \E_x\Big(e^{-q\tau_0^-} u_i (X(\tau_0^-), a );\tau_0^-<\t_b^+\Big), \quad i = 2,3, 4.
\end{split}
\end{align}
%\red{JL: I think it holds but we need to make sure that the following holds for $x \leq 0$ as well.}
%\blue{Hi Kazu: So given that for $x<0$
%\[
%U_i(x,a,b)=u_i(x,a,b).
%\]
%Then we have to check that in the following identities, but I think it follows easy by noting that for $x<0$ $W_{\q,\p}^a(x)=W_{\q+\p}(x-a)$, $Z_{\q,\p}^a(x)=Z_{\q+\p}(x-a)$, $\bar{Z}_{\q,\p}^a(x)=\bar{Z}_{\q+\p}(x-a)$.}
%\green{JL: So these expressions are closely related to $U^0$.  For example,
%\[
%U_1(x, a,b,\theta) = W_q(x) U_1^0(a,b,\theta) - W_q(b) U_1^0(a,x,\theta)
%\]
%Probably it holds for other $i=2,3,4$. It may be better to define $U^0$ and use to express $U$?  We can add a sentence saying that $U^0$ will describe later the excursion measure.} \blue{Hi Kazu: yes I believe that it holds for $i=1,3,4$, although I am not sure if its simplifying too much, and it will get confusing?}
%\green{JL: From
%Lemma 2.1 of \cite{LRZ}, it is very natural to obtain $U^0$ (even without our results for the excursion measure case).  In fact, we can then simplify a lot because we can just say $U_i(x, a,b) = W_q(x) U_i^0(a,b) - W_q(b) U_i^0(a,x)$ for $i = ...$?}
%\blue{Ok Kazu.}
For $a < 0$ and $x \in \R$, we define
\begin{align} \label{U_0_summary}
\begin{split}
U_1^0(a,x, \theta)&:=- \p\int_{0}^{-a} e^{-\theta u}\Big[ \frac {W_{\q,\p}^a(x) W_{\q+\p} (u)} {W_{\q+\p}(-a)} -W_{\q,\p}^{-u} (x) \Big]\diff u, \quad \theta \geq 0, \\
U_2^0(a,x) &:= - W_{q,r}^a (x) / W_{q+r}(-a), \\
U_3^0(a,x) &:= - \mathcal{I}_{q,r}^a (x) , \\
U_4^0(a,x) &:= - h^{a}_{q,r}(x),
\end{split}
\end{align}
where $\mathcal{I}_{q,r}^a$ and $h^{a}_{q,r}$ are defined in \eqref{def_H_I} and  \eqref{small_h_def}, respectively.
In particular,
\begin{align} \label{U_0_summary_2}
\begin{split}
U_1^0(a, x, 0) &=- \frac{\p}{\q+\p}\Big( Z_{\q}(x)
- Z_{\q,\p}^a(x) +\frac{ (q+r) \overline{W}_{\q+\p}(-a)}{W_{\q+\p}(-a)}W_{\q,\p}^a(x)\Big).
\end{split}
\end{align}
%\green{[moved here]
Note that
\begin{align}
U_1^0(a,x, \theta) + U_2^0(a,x) = -\mathcal{H}_{q,r}^a(x, \theta), \quad x \in \R, \quad \theta \geq 0, \label{U_1_U_2_sum}
\end{align}
and
\begin{align}
U_1^0(a,0, \theta)  = U_3^0(a,0)  = U_4^0(a,0)  = 0 \quad \textrm{and} \quad U_2^0(a,0) = -1. \label{U_zero_value}
\end{align}

Here, we have the following results.

\begin{corollary} \label{corollary_big_U}
For $q \geq 0$, $a < 0 < b$, and $x \leq b$,
\begin{align*}
U_1(x, a,b,\theta) &=   \frac{W_q(x)}{W_q(b)}U_1^0(a,b,\theta) - U_1^0(a,x,\theta), \quad \theta \geq 0, \\
U_i(x, a,b) &=  \frac{W_q(x)}{W_q(b)}U_i^0(a,b) - U_i^0(a,x), \quad i = 2,3, 4.
\end{align*}
\end{corollary}
\begin{proof}
The case in which $i=1,2,3$ is a direct consequence of Lemma \ref{lemma_u_in_V} and \eqref{LRZ_identity}, and hence we omit the proof.
%\green{[JL: I think we can use this argument (which holds without $u_4 \in \mathcal{V}$ nor $\kappa(0+) > -\infty$)]
For the remaining case, it is clear that $U_4(x, a,b) = \lim_{\theta \downarrow 0} (\partial U_1(x, a,b, \theta) / \partial \theta)$.  We have
\begin{align*}
\lim_{\theta \downarrow 0} \frac \partial {\partial \theta}U_1^0(a,x, \theta)&= \p\int_{0}^{-a} u \Big[ \frac {W_{\q,\p}^a(x) W_{\q+\p} (u)} {W_{\q+\p}(-a)} -W_{\q,\p}^{-u} (x) \Big]\diff u.
\end{align*}
Here integration by parts gives
\begin{align*}
\int_{0}^{-a} u W_{\q+\p} (u+x) \diff u %&= -a \overline{W}_{q+r} (-a+x) - \int_0^{-a} \overline{W}_{q+r} (u+x) \diff u \\
%&= -a \overline{W}_{q+r} (-a+x) - \frac {\overline{Z}_{q+r}(x-a) - \overline{Z}_{q+r} (x) + a} {q+r}
%\\ &
= (q+r)^{-1} \Big( -aZ_{q+r} (x-a) - \overline{Z}_{q+r}(x-a) + \overline{Z}_{q+r} (x)\Big),
\end{align*}
and
\begin{align*}
\int_{0}^{-a} u W_{\q,\p}^{-u} (x) \diff u = \int_{0}^{-a} u \Big( W_{q+r}(x+u)- r\int_{0}^{x}  W_{q}(x-y)W_{q+r}(y+u)\diff y \Big) \diff u
\end{align*}
where another integration by parts and Fubini's theorem give
\begin{align*}
&\int_{0}^{-a} u \int_{0}^{x}  W_{q}(x-y)W_{q+r}(y+u)\diff y \diff u \\
% &= \Big[ u \int_0^x W_q (x-y) \overline{W}_{q+r} (y+u) \diff y \Big]_{u=0}^{u = -a}
% - \int_0^{-a} \int_0^x W_q(x-y) \overline{W}_{q+r} (y+u) \diff y \diff u \\
  &=  -a \int_0^x W_q (x-y) \overline{W}_{q+r} (y-a) \diff y
 - \int_0^x W_q(x-y)  \int_0^{-a} \overline{W}_{q+r} (y+u) \diff u \diff y \\
 &=  (q+r)^{-1} \Big[-a \int_0^x W_q (x-y) Z_{q+r} (y-a) \diff y
 - \int_0^x W_q(x-y)  \Big( \overline{Z}_{q+r}(y-a) - \overline{Z}_{q+r} (y) \Big) \diff y \Big].
\end{align*}
Hence, substituting these and using \eqref{rel_Z_0}, we obtain, %\green{[ok to delete these?]
%\begin{align*}
%&\lim_{\theta \downarrow 0} \frac \partial {\partial \theta}U_1^0(a,x, \theta)
%= \frac {W_{\q,\p}^a(x)} {W_{\q+\p}(-a)}  \Big[  -aZ_{q+r} (-a) - \overline{Z}_{q+r}(-a) \Big] \\
%&-  \Big( -aZ_{q+r} (x-a) - \overline{Z}_{q+r}(x-a) + \overline{Z}_{q+r} (x)\Big)  \\
%&+ r \Big[-a \int_0^x W_q (x-y) Z_{q+r} (y-a) \diff y
% - \int_0^x W_q(x-y)  \Big( \overline{Z}_{q+r}(y-a) - \overline{Z}_{q+r} (y) \Big) \diff y \Big] \\
% &= \frac {W_{\q,\p}^a(x)} {W_{\q+\p}(-a)}  \Big[  -a Z_{q+r} (-a) - \overline{Z}_{q+r}(-a)  \Big]
%-  \Big( -aZ_{q,r}^a (x) - \overline{Z}_{q,r}^a(x) + \overline{Z}^0_{q+r} (x)\Big) \\
%&=\frac  r {q+r}  \Big[  \frac {W_{\q,\p}^a(x)} {W_{\q+\p}(-a)}  \Big( -aZ_{q+r} (-a) - \overline{Z}_{q+r}(-a) \Big)
%-  \Big( -aZ_{q,r}^a (x) - \overline{Z}_{q,r}^a(x) + \overline{Z}_{q} (x)\Big) \Big].
%\end{align*}
%Hence we can confirm that}
\begin{align}\label{new_iden_u4}
\lim_{\theta \downarrow 0} \frac \partial {\partial \theta}U_1^0(a,x, \theta) = - h^a_{q,r} (x)=U_4^0(a,x).
\end{align}
Therefore, putting the pieces together we get the result.
\end{proof}

\subsection{Proofs of Theorem \ref{theorem_laplace} for the bounded variation case}
%\begin{itemize}
%\item[(ii)]
%(ii) %(Two-sided smooth first passage problem)
%\par
%We will use the following notation
%\[
%	\blue{g(x,a, b)} =  \E_x\left[e^{-q\tau_b^+- \theta L_\p(\tau_b^+)};\tau_b^+<\tau_a^-\right].
%\]
\subsubsection{Proof of \eqref{Parisbailouts}}
%\red{[splitted equations.]}
For all $a < 0 <  b$ and $x \leq b$,  by \eqref{X_X_r_the_same} and the strong Markov property %at $\tau_0^-$ on $\{ \tau_0^- < \tau_b^+\}$
together with Corollary \ref{corollary_big_U} and \eqref{U_1_U_2_sum}, %\green{[JL: I am the one who added this, but I guess we can just delete ``at $\tau_0^-$ on $\{ \tau_0^- < \tau_b^+\}$'' because it is trivial and also we are using the Markov property at $\tau_0^- + \tau^+_0 \circ \Theta_{\tau_0^-}$]}
		\begin{align}\label{sdep_1}
		g(x,a, b,\theta) &=  \E_x\left(e^{-q\tau_b^+}; \tau_0^- > \tau_b^+ \right)\notag \\
		&+\E_x\Big[ e^{-q\tau_0^-}\E_{X(\tau_0^-)}
		\left(e^{-q\mathbf{e}_{\p}+\theta X(\mathbf{e}_{\p})};
		\mathbf{e}_{\p}<\tau_0^+\wedge\tau_a^-\right);\tau_0^-<\t_b^+\Big] g(0,a, b,\theta)\notag\\
		&+\E_x\Big[e^{-q\tau_0^-}\E_{X(\tau_0^-)}
		\left(e^{-q\tau_0^+}; \tau_0^+ <
		\mathbf{e}_{\p}\wedge\tau_a^- \right);\tau_0^-<\t_b^+\Big] g(0,a, b,\theta)\notag\\
		&=	\frac{W_q(x)}{W_q(b)} +[ U_1(x, a, b, \theta) + U_2(x,a,b) ] g(0,a, b,\theta) \notag\\
		&=\frac{W_q(x)}{W_q(b)} \Big( 1 - \mathcal{H}_{q,r}^a(b, \theta) g(0,a,b,\theta) \Big) +\mathcal{H}_{q,r}^a(x, \theta) g(0,a, b,\theta).
		\end{align}
Now setting $x=0$ and solving for $g(0,a, b, \theta)$ (using \eqref{U_zero_value}), we get
\begin{align}
g(0,a, b,\theta)=\mathcal{H}_{q,r}^a(b, \theta)^{-1}. \label{g_zero}
\end{align}
Substituting this in (\ref{sdep_1}) we have the result.% -- and remove the rest?}

	%	In particular, for the case $x < 0$, it redues to $g(x,a, b,\theta) = (u_1(x, a, \theta) + u_2(x,b)) g(0,a, b,\theta)$.
		
	%	By  \eqref{laplace_in_terms_of_z} and Corollary \ref{corollary_big_U},  %\red{[how about this?]}
%		\begin{align}
%\label{sdep_1}
%		\begin{split}
%g(x,a, b,\theta) &=\frac{W_q(x)}{W_q(b)} +\left(W_{\q,\p}^a(x)-\frac{W_q(x)}{W_q(b)}W_{\q,\p}^a(b)\right)\frac{g(0,a,b, \theta)} {W_{\q+\p}(-a)} \\
%& + g(0,a, b,\theta)\p\int_{0}^{-a} e^{-\theta u}\Bigg[\left(W_{\q,\p}^a(x)-\frac{W_q(x)}{W_q(b)}W_{\q,\p}^a(b)\right)\frac{W_{\q+\p} (u)} {W_{\q+\p}(-a)}\\
%	&-\left(W_{\q,\p}^{-u}(x)-\frac{W_q(x)}{W_q(b)}W_{\q,\p}^{-u}(b)\right)\Bigg]\diff u.
%	\end{split}
%	\end{align}
%	Now setting $x=0$ and solving for $g(0,a, b, \theta)$, we get
%\begin{align}
%g(0,a, b,\theta)=\mathcal{H}_{q,r}^a(b, \theta)^{-1}. \label{g_zero}
%\end{align}
%Substituting this in (\ref{sdep_1}) we have the result.
%\begin{align*}
%	\blue{g(x,a, b)}&=\blue{g(0,a, b)}\left[\p\int_{0}^{-a} e^{\theta(y+a)}\Bigg[W_{\q,\p}^a(x)\frac{W_{\q+\p} (-a-y)} {W_{\q+\p}(-a)}-\blue{W_{\q,\p}^{a+y}(x)}\Bigg]\diff y+\frac{W_{\q,\p}^a(x)}{W_{\q+\p}(-a)}\right]\\
%%	\end{align*}
%\green{remove?} where
%\[
%\blue{\mathcal{H}^a_{q,r}}(x,\theta):=\p\int_{0}^{-a} e^{\theta(y+a)}\Bigg[W_{\q,\p}^a(x)\frac{W_{\q+\p} (-a-y)} {W_{\q+\p}(-a)}-\blue{W_{\q,\p}^{a+y}(x)}\Bigg] \diff y+\frac{W_{\q,\p}^a(x)}{W_{\q+\p}(-a)}.
%\]
%and
%\[
%\blue{W_{\q,\p}^u(x)=W_{\q+\p}(x-u)- \p\int_{0}^{x}  W_\q(x-z)W_{\q+\p}(z-u)\diff z}.
%\]
%\item[(iii)]
\subsubsection{Proof of \eqref{Parisbailouts_2}}
Similarly, for all $a < 0 <  b$ and $x \leq b$,  %\green{[combine the first two expectations below?]}
%Again, by \eqref{X_X_r_the_same} and the strong Markov property at $\tau_0^-$,
%Again, by an application of the Markov property and because $X_r(t) = X(t)$ a.s.\ for $0 \leq t \leq \tau_0^-$,
%
%(Two sided non-smooth first passage problem)
%\par Finally in this section we will compute the following quantity
%\[
%\blue{h(x,a,b)}:=\E_x\left[e^{-q\tau_a^- -\theta L^{(r)}(\tau_a^-)}; \tau_a^- < \t_b^+\right],\qquad a \leq 0 < b.
%\]
%To this end we proceed like in the previous sections, hence by an application of the Markov property we obtain that
%\begin{align*}
%\blue{h(x,a,b)}&=\E_x\left(e^{-q\tau_0^-}1_{\{X(\tau_0^{-})< a\}};\tau_0^-<\tau_b^+\right)\\
%&+\Big( \E_x\left(e^{-q\tau_0^-} \E_{X(\tau_0^{-})}\left(e^{-q\tau_a^-};\tau_a^-<\mathbf{e}_{\p}\wedge\tau_0^+\right) + u_1(x, a, \theta) + u_2(x,a)\Big);\tau_0^-<\tau_b^+\right)h_a(0,b)
%%\\
%%&+\E_x\left(e^{-q\tau_0^-}1_{\{X(\tau_0^{-})\geq a\}} u_2(X(\tau_0^-, a));\tau_0^-<\tau_b^+\right)h_a(0,b)\\
%%&+\E_x\left(e^{-q\tau_0^-}1_{\{X(\tau_0^{-})\geq a\}}\E_{X(\tau_0^{-})}\left(e^{-q\mathbf{e}_{\p}+\theta X(\mathbf{e}_{\p})};\mathbf{e}_{\p}<\tau_0^+\wedge\tau_a^-\right);\tau_0^-<\tau_b^+\right)h_a(0,b).
%\end{align*}
%Now we note that
%\begin{align*}
%\E_x&\left(e^{-q\tau_0^-}1_{\{X(\tau_0^{-})\geq a\}}\E_{X(\tau_0^-)}\left(e^{-q\tau_a^-};\tau_a^-<\mathbf{e}_{\p}\wedge\tau_0^+\right);\tau_0^-<\tau_b^+\right)\\
%&=\E_x\left(e^{-q\tau_0^-}\E_{X(\tau_0^-)}\left(e^{-q\tau_a^-};\tau_a^-<\mathbf{e}_{\p}\wedge\tau_0^+\right);\tau_0^-<\tau_b^+\right)-\E_x\left(e^{-q\tau_0^-}1_{\{X(\tau_0^-)<a\}}\;\tau_0^-<\tau_b^+\right).
%\end{align*}
%Therefore
\begin{align}\label{Exit_prob_a_new1}
h(x,a, b,\theta)&=\E_x\left[ e^{-q\tau_0^-}\E_{X(\tau_0^{-})}\left(e^{-q\tau_a^-};\tau_a^-<\mathbf{e}_{\p}\wedge\tau_0^+\right);\tau_0^-<\tau_b^+\right] \notag\\
&+\E_x\left[ e^{-q\tau_0^-}\E_{X(\tau_0^{-})}\left(e^{-q\mathbf{e}_{\p}+\theta X(\mathbf{e}_{\p})};\mathbf{e}_{\p}<\tau_0^+\wedge\tau_a^-\right);\tau_0^-<\tau_b^+\right] h(0,a, b,\theta)\notag\\
&+\E_x\left[ e^{-q\tau_0^-}\E_{X(\tau_0^{-})}\left(e^{-q\tau_0^+};\tau_0^+<\mathbf{e}_{\p}\wedge\tau_a^-\right);\tau_0^-<\tau_b^+\right] h(0,a, b,\theta)\notag\\
&=U_3 (x, a,b)+ [U_1(x, a,b, \theta) + U_2 (x, a,b)]  h(0,a, b,\theta)\notag \\
&=-\frac{W_q(x)}{W_q(b)} \Big( \mathcal{I}_{q,r}^a(b) + \mathcal{H}_{q,r}^a(b, \theta) h(0,a,b,\theta) \Big) + \mathcal{I}_{q,r}^a(x) +\mathcal{H}_{q,r}^a(x, \theta) h(0,a, b,\theta).
\end{align}
Now setting $x=0$ and solving for $h(0,a,b, \theta)$ (using \eqref{U_zero_value}), we get
\begin{align}
h(0,a, b,\theta)=-\mathcal{I}_{\q,\p}^a(b)/\mathcal{H}_{q,r}^a(b, \theta). \label{h_zero}
\end{align}
Substituting \eqref{h_zero} in (\ref{Exit_prob_a_new1}), we obtain the result.% -- and remove the rest?
\subsection{Proof of Theorem \ref{expected_bailouts} for the bounded variation case}
%\green{[JL: I guess the case $q=0$ is also covered in the following arguments.  So remove this sentence?]We shall show for the case $q > 0$; the case $q=0$ holds by monotone convergence.}
By \eqref{X_X_r_the_same}, the strong Markov property, Corollary \ref{corollary_big_U}, and \eqref{U_1_U_2_sum},
%By the Markov property and because $X_r(t) = X(t)$ a.s.\ for $0 \leq t \leq \tau_0^-$,
\begin{align}\label{MPa}
f(x,a,b)&=-\E_x\left[ e^{-q\tau_0^-}
\E_{X(\tau_0^-)}
\left(e^{-q\mathbf{e}_{\p}}X(\mathbf{e}_{\p});
\mathbf{e}_{\p}<\tau_0^+\wedge\tau_a^-\right);
%\left(e^{-q\mathbf{e}_{\p}}X(\mathbf{e}_{\p});
%\mathbf{e}_{\p}<\tau_0^+\wedge\tau_a^-\right);
\tau_0^-<\t_b^+\right]\notag\\
&+\E_x\left[e^{-q\tau_0^-}\E_{X(\tau_0^-)} \left(e^{-q\tau_0^+
};\tau_0^+<\mathbf{e}_{\p}\wedge\tau_a^-\right);\tau_0^-<\t_b^+\right]f(0,a,b)\notag\\
&+\E_x\left[e^{-q\tau_0^-}\E_{X(\tau_0^-)}\left(e^{-q\mathbf{e}_{\p}
};\mathbf{e}_{\p}<\tau_0^+\wedge\tau_a^-\right);\tau_0^-<\t_b^+\right]f(0,a,b)\notag\\
&=- U_4 (x, a,b)
+ [U_1(x, a, b, 0) +U_2(x, a,b)] f(0,a,b) \notag \\
&=\frac{W_q(x)}{W_q(b)} \Big( h^{a}_{q,r}(b) - \mathcal{H}_{q,r}^a(b, 0) f(0,a,b) \Big) -h^{a}_{q,r}(x)+\mathcal{H}_{q,r}^a(x, 0) f(0,a, b).
\end{align}
Now setting $x=0$ and solving for $f(0,a,b)$ (using \eqref{U_zero_value}), we get $f(0,a, b)=   h^{a}_{q,r}(b)  /\mathcal{H}_{q,r}^a(b, 0)$.
Substituting this in (\ref{MPa}), we obtain the result.% -- and remove the rest?

\section{Proofs for the unbounded variation case} \label{section_unbounded}

This section shows Theorems \ref{theorem_laplace} and  \ref{expected_bailouts} for the case $X$ is of unbounded variation. The proof is via excursion theory, and in particular we use the recent results obtained in \cite{PPR2}.  Toward this end, we shall first obtain a key formula (Theorem \ref{lemma_key_excursion}), which is an analogue of Lemma 2.1 in \cite{LRZ} (as discussed and used in the last section) under the excursion measure. Using this and well-known results in excursion theory, we derive the same expressions as in the bounded variation case.  We refer the reader to \cite{PPR2} for detailed introduction and definitions regarding excursions away from zero for the case of spectrally negative \lev processes.

\subsection{Key identities}

We assume throughout this section that $X$ is of unbounded variation so that $0$ is  a regular point.  Let $\mathbf{n}$ be the excursion measure away from zero for $X$. %\green{[swapped these two sentences.]}

%\green{JL: In your paper with JC and Victor, it is assumed that the discount is strictly positive and so I set here that $p, q > 0$; I extend to the case $q=0$ by monotone convergence in the proof of the theorems. Is this ok?}

%\green{JL: I think this is an important result.  Change from Lemma to Theorem?}
\begin{theorem} \label{lemma_key_excursion} Fix $p, q > 0$ and  $b > 0$.
	Consider functions $w_p, v_q: \R \rightarrow [0, \infty)$ %\green{[JL: LRZ is assuming the positivity and so I change this way; below, I removed absolute values in places because it would not be necessary by this assumption. ]}
	that satisfy the following: %\red{[it may be confusing to use $v$ for both?  How about $w$ for the second one?]} which satisfy the following
	\begin{enumerate}
		\item  $w_p$ and $v_q$ belong to the classes  $\tilde{\mathcal{V}}^{(p)}_0$ and $\tilde{\mathcal{V}}^{(q)}_0$, respectively. %\red{[change to $\tilde{\mathcal{V}}^{(p)}_0$ and $\tilde{\mathcal{V}}^{(q)}_0$?]}\blue{Kazu: I don't think is necessary, and also that is the notation followed in \cite{LRZ}}. \red{So \cite{LRZ} requires conditions additionally to $\mathcal{V}^{(p)}_0$?  I defined $\tilde{\mathcal{V}}^{(p)}_0$ to be such that these are also satisfied. What do you think? I also think it may be a good idea to flip here $p$ and $q$ to be consistent with other parts?}
		\item We have
		\begin{align}
		(w_p-v_q)(0)=0 \label{v_diff_zero}
		\end{align}
		 and the following limits are well defined and finite:
		\begin{align}
		\lim_{x \downarrow 0} \frac {(w_p-v_q)(x)} {W_q(x)}, \quad \lim_{x \uparrow 0} \frac {(w_p-v_q)(x)} {|x|}. \label{derivative_exist_left_right}
		\end{align}
		\item There exist bounded, $\mathcal{F}_{\tau_b^+ \wedge\tau_0^-}$-measurable functionals $F$ and $G$ such that  %\red{$\mathcal{F}_{\tau_0^- \wedge\tau_b^+}$?} \blue{I believe most of the identities are expectations of functionals that are $\mathcal{F}_{\tau_0^- \wedge\tau_b^+}$}
		\[w_p(x)=\E_x(F) \qquad\text{and}\qquad v_q(x)=\E_x(G), \quad x < 0.
		\]
	\end{enumerate}
	Then we have %\green{[``for all $b > 0$" to be removed because in 3, $b$ is fixed.]}
	\begin{align}\label{LRZ_exc}
	\mathbf{n}\left(e^{-q\tau_0^-}(w_p-v_q)(X(\tau_0^-));\tau_0^-<\tau_b^+\right)
	&=\lim_{x\downarrow 0}\frac{(w_p-v_q)(x)}{W_q(x)}+\frac{\sigma^2}{2}\lim_{x\uparrow 0}\frac{(w_p-v_q)(x)}{|x|}\notag\\
	&-\frac{1}{W_q(b)}\left((w_p-v_q)(b)-(p-q)\int_0^bW_q(b-y) w_p(y) \diff y\right).
	\end{align}
	%\red{[I guess the font of $\textbf{n}$ should be changed? It changes in theorems.]}
\end{theorem}
%\red{JL: Can we define $(w_p- v_q) (x) := w_p(x) - v_q(x)$ and simplify a bit?}
\begin{proof}
%\red{[JL: Changed the order of arguments. What do you think?]}
First, we decompose
	\begin{align*}
	\mathcal{N} :=\mathbf{n}&\left(e^{-q\tau_0^-}(w_p-v_q)(X(\tau_0^-));\tau_0^-<\tau_b^+\right)= \mathcal{N}_1 + \mathcal{N}_2,
	\end{align*}
	where
	\begin{align*}
	\mathcal{N}_1&:=\mathbf{n}\left(e^{-q\tau_0^-}(w_p-v_q)(X(\tau_0^-));\tau_0^-<\tau_b^+,\tau_0^->0\right), \\
	\mathcal{N}_2 &:= \mathbf{n}\left(e^{-q\tau_0^-}(w_p-v_q)(X(\tau_0^-));\tau_0^-<\tau_b^+,\tau_0^-=0\right).
	\end{align*}
	Now since under $\mathbf{n}$,  the jumps of $X$  constitute a Poisson point process (see  for instance Section 4 in \cite{PPR2}), an application of the Master's formula under $\mathbf{n}$ (see  for instance identity (14) in \cite{PPR2}) and Lemma 1 in \cite{PPR1} allow us to deduce, %By Lemma 1 in \cite{PPR1}, the former becomes %\green{[changed $y \to \theta$ and $x \to y$]}
	with $\zeta$ being the length of the excursion from the point it leaves $0$,
	\begin{align*}
%	\textbf{n}&\left(e^{-p\tau_0^-}\left(v^{(q)}(X(\tau_0^-))-v^{(p)}(X(\tau_0^-))\right);\tau_0^-<\tau_b^+,\tau_0^->0\right)\\
	\mathcal{N}_1 &=\mathbf{n}\left(\sum_{0<s<\zeta}e^{-qs}1_{\{X(s-)>0, \,  \sup_{u\leq s}X(u)<b\}} (w_p-v_q)(X(s-)+\Delta X(s))\right)\\
	&=\mathbf{n}\left(\int_0^{\infty}\int_{(-\infty, 0)}e^{-qs}1_{\{X(s-)>0, \,  \sup_{u\leq s}X(u)<b\}} (w_p-v_q)(X(s-)+\theta)\Pi(\diff \theta)\diff s\right)\\
	&=\int_0^b\frac{W_q(b-y)}{W_q(b)}\int_{(-\infty, -y)} (w_p-v_q)(y+\theta)\Pi(\diff \theta)\diff y.
	\end{align*}
	%where the third equality follows from By Lemma 1 in \cite{PPR1}.}
%\green{[added this paragraph]}
To see why this is finite, by our assumptions  \eqref{v_diff_zero} and \eqref{derivative_exist_left_right},
%2 that $(w_p-v_q)(0)=0$ and that its left hand derivative exists at zero,
there exist $K > 0$ and $\bar{y} < 0$ such that %\green{[added absolute]}
\[
|(w_p-v_q)(y+\th) | \leq K | y+\th |, \quad  \bar{y} < y + \theta < 0.
\]
Hence, for $\varepsilon$ small enough,
\[
\int_0^{\varepsilon}\int_{(-\varepsilon,-y)} |(w_p-v_q)(y+\th) | \Pi(\diff \th) \diff y\leq K\int_{(-\varepsilon,0)} \int_0^{-\th} | y+\th| \diff y\Pi( \diff \th)=\frac{K}{2}\int_{(-\varepsilon,0)}\th^2\Pi(\diff \th)<\infty.
\]
%\green{[JL: I get this
%\begin{align*}
%\int_0^{\varepsilon}\int_{-\varepsilon}^{-y} |(w_p-v_q)(\th+y) | \Pi(\diff \th) \diff y \leq K \int_0^{\varepsilon}\int_0^{-z} |z+\th | \Pi(\diff \th) \diff z?
%\end{align*}
%]}
%\blue{Yes Kazu, then by a change of the order of integration
%\begin{align*}
%\int_0^{\varepsilon}\int_0^{-z} |z+\th | \Pi(\diff \th) \diff z=\int_{-\varepsilon}^{0}\int_0^{-\th} | y+\th| \diff y\Pi( \diff \th)=\int_{-\varepsilon}^{0}\th^2\Pi(\diff \th)<\infty?
%\end{align*}	}	
%\green{I see (actually my computation above is wrong) and also, I guess the final coefficient would be different?}
%\green{[remove?]Just by noting that $W_q(b-y)$ and $W_q'(b-y)$ is bounded, and the fact that $(w_p-v_q)(0)=0$ and the left hand derivative exists at zero then
%\[
%(w_p-v_q)(\th+y)\leq K(y+\th)
%\]
%when $y$ and $\th$ are close to zero and $K$ is a constant, }
	
On the other hand, using Theorem 3 (i) in \cite{PPR2} we obtain that
	\begin{align*}
%	\textbf{n}&\left(e^{-p\tau_0^-}\left(v^{(q)}(X(\tau_0^-))-v^{(p)}(X(\tau_0^-))\right);\tau_0^-<\tau_b^+,\tau_0^-=0\right)
	\mathcal{N}_2=\frac{\sigma^2}{2}\left[\bar{\mathbf{n}}(F)-\bar{\mathbf{n}}(G)\right]
	\end{align*}
%\begin{align*}
%h_2 &=\mathbf{n}\left(e^{-q\mathbf{e}_{\p} + \theta X(\mathbf{e}_{\p})}; \zeta \wedge \tau_a^- >\mathbf{e}_{\p},\tau_0^-=0\right)=\frac{\sigma^2}{2}\bar{\textbf{n}}\left(e^{-q\mathbf{e}_{\p}+ \green{\theta}X(\mathbf{e}_{\p})}; \zeta \wedge \tau_a^->\mathbf{e}_{\p}\right),
%\end{align*}
where $\bar{\mathbf{n}}$ is the excursion measure of the spectrally negative L\'evy process reflected at its supremum. %\green{[don't know the following arguments... below it is just a copy of the arguments for the bailout]}In order to compute the above term we will use one of the main results of \cite{CD} which establishes that the measure $\bar{\textbf{n}}$ can be constructed as a limit of the law of $X$ killed at its first passage time above $0$, therefore using L'Hospital rule and \eqref{kazuiden} we get %\green{[JL: I guess we are taking $x \rightarrow 0-$?]}
By using the results of \cite{CD} (see, in particular, page 6 of \cite{PPR2} for the spectrally negative case),
%\begin{align*}
%\bar{\textbf{n}}\left(e^{-q\mathbf{e}_{\p}+\theta X(\mathbf{e}_{\p})}; \zeta \wedge \tau_a^->\mathbf{e}_{\p}\right)&=\lim_{x\to0-}\frac{r}{|x|}\E_x\left(\int_0^{\tau_0^+\wedge\tau_a^-}e^{-(q+\p)s+\theta X(s)} \diff s\right)\\
%&=\lim_{x\to0-}\frac{\p}{|x|}\int_{0}^{-a} e^{\theta(y+a)}\Big[ \frac {W_{\q+\p}(x-a) W_{\q+\p} (-a-y)} {W_{\q+\p}(-a)} -W_{\q+\p} (x-a-y) \Big]\diff y \\
%&=-\p\int_{0}^{-a} e^{\theta(y+a)}\Big[ \frac {W_{\q+\p}'(-a) W_{\q+\p} (-a-y)} {W_{\q+\p}(-a)} -W_{\q+\p}' (-a-y) \Big]\diff y.
%\end{align*}
%	On the other hand
	\begin{align*}
%	\textbf{n}&\left(e^{-p\tau_0^-}\left(v^{(q)}(X(\tau_0^-))-v^{(p)}(X(\tau_0^-))\right);\tau_0^-<\tau_b^+,\tau_0^-=0\right)=\frac{\sigma^2}{2}\left(\bar{\textbf{n}}(F)+\bar{\textbf{n}}(G)\right)\\
	\mathcal{N}_2 &=\frac{\sigma^2}{2}\lim_{x \uparrow  0}\frac{\E_x(F)-\E_x(G)}{|x|}=\frac{\sigma^2}{2}\lim_{x\uparrow  0}\frac{(w_p-v_q)(x)}{|x|}.
	\end{align*}
	%\red{[JL: I guess $F$ and $G$ need to be bounded?  We may need to localize first?]}
	%So putting the pieces together we obtain that
	Hence, summing up these,
	\begin{align} \label{eq_N_mathcal}
	\mathcal{N} =\int_0^b\frac{W_q(b-y)}{W_q(b)}\int_{(-\infty,-y)}(w_p-v_q)(y+\th)\Pi(\diff \th) \diff y +\frac{\sigma^2}{2}\lim_{x\uparrow 0}\frac{(w_p-v_q)(x)}{|x|}.
%	&=\lim_{x\downarrow 0}\frac{v^{(q)}(x)-v^{(p)}(x)}{W_p(x)}+\frac{\sigma^2}{2}\lim_{x\to 0}\frac{1}{|x|}(v^{(q)}(x)-v^{(p)}(x))\\
%	&-\frac{1}{W_p(b)}\left(v^{(q)}(b)-(q-p)\int_0^bW_p(b-y)v^{(q)}(y)dy-v^{(p)}(b)\right).
	\end{align}
	
We shall now simplify \eqref{eq_N_mathcal} using the identities by \cite{LRZ} (as reviewed in the last section).
	Using the identities \eqref{undershoot_expectation} and  \eqref{v_diff_zero}, we obtain
	%\green{[JL: Here I think we can just modify \eqref{undershoot_expectation} for the unbounded variation case (with creeping term) and reference \eqref{undershoot_expectation}?]}
	\begin{align*}
	\mathcal{V}_{p,q} (x)  &:= \E_x\left(e^{-q\tau_0^-}(w_p-v_q)(X(\tau_0^-));\tau_0^-<\tau_b^+\right)\\
%	&=\int_0^b\int_{\red{-\infty}}^{-y}(v^{(q)}(y+\theta)-v^{(p)}(y+\theta))\left(\frac{W_p(b-y)}{W_q(b)}W_p(x)-W_p(x-y)\right)\Pi(d\theta)dy\\
%	&+(v^{(q)}(0)-v^{(p)}(0))\frac{\sigma^2}{2}\left(W^{(q)\prime}(x)-W_p(x)\frac{W_p'(b)}{W_q(b)}\right)\\
	&=\int_0^b\int_{(-\infty, -y)} (w_p-v_q)(y+\theta) \left(\frac{W_q(b-y)}{W_q(b)}W_q(x)-W_q(x-y)\right)\Pi(\diff \theta)\diff y, \quad 0 <  x \leq b.
	\end{align*}
	On the other hand, by \eqref{LRZ_identity} applied to $w_p \in \tilde{\mathcal{V}}^{(p)}_0$ and $v_q\in \tilde{\mathcal{V}}^{(q)}_0$,
	%by identity (15) in \cite{LRZ},  we also have
	\begin{align*}
%	\E_x&\left(e^{-p\tau_0^-}\left(v^{(q)}(X(\tau_0^-))-v^{(p)}(X(\tau_0^-))\right);\tau_0^-<\tau_b^+\right)\\
	\mathcal{V}_{p,q} (x)
	&=(w_p-v_q)(x)-(p-q)\int_0^x W_q(x-y)w_p(y)\diff y\\&-\frac{W_q(x)}{W_q(b)}\left((w_p-v_q)(b)-(p-q)\int_0^bW_q(b-y)w_p(y)\diff y\right).
	\end{align*}
	%\red{[the $b$ in the integrand not necessary?]}
	By matching these,
	\begin{align*}
	\int_0^{b}\int_{(-\infty, -y)}&(w_p-v_q)(y+\theta)\left(\frac{W_q(b-y)}{W_q(b)}-\frac{W_q(x-y)}{W_q(x)}\right)\Pi(\diff \theta)\diff y\\
	&=\frac{1}{W_q(x)}\left((w_p-v_q)(x)-(p-q)\int_0^x W_q(x-y)w_p(y)\diff y\right)\\
	&-\frac{1}{W_q(b)}\left((w_p-v_q)(b)-(p-q)\int_0^bW_q(b-y)w_p(y)\diff y\right).
	\end{align*}
	%Now we note the following limits
	%\begin{align*}
	%\lim_{x\to 0}\frac{v^{(q)}(x)}{W_p(x)}=\frac{v^{(q)\prime}(0+)}{W_p'(0+)}=\frac{\sigma^2}{2}v^{(q)\prime}(0+).
	%\end{align*}
	We shall now take $x \downarrow 0$ on both sides.
Using that $W_q(x-y)/W_q(x)$ is increasing in $x$ by Remark \ref{remark_smoothness_zero} (3), monotone convergence gives
\begin{align*}
	%\frac{1}{W_q(x)}\int_0^xW_q(x-y)|w_p(y) - v_q(y)|dy=
	 \int_0^{b} \frac {W_q(x-y)} {W_q(x)}w_p(y) \diff y \xrightarrow{x \downarrow 0} 0 %\]
	%\green{JL: Should it be $\int_0^{b} \frac {W_q(x-y)} {W_q(x)}|w_p(y)|\diff y \xrightarrow{x \downarrow 0} 0$?}and
%	\begin{align*}
\quad \textrm{and} \quad
	\int_0^b \frac{W_q(x-y)}{W_q(x)} \int_{(-\infty, -y)} |(w_p-v_q)(y+\theta)| \Pi(\diff \theta)\diff y \xrightarrow{x \downarrow 0} 0.
	\end{align*}
	Hence, %\green{JL: Below, I guess there is no point of multiplying both sides by $W_q(b)$?}
	\begin{multline*}
	\int_0^b\frac{W_q(b-y)}{W_q(b)} \int_{(-\infty, -y)}(w_p-v_q)(y+\theta)\Pi(\diff \theta)\diff y-\lim_{x\downarrow 0}\frac{(w_p-v_q)(x)}{W_q(x)}\\
	=-\frac{1}{W_q(b)}\left((w_p-v_q)(b)-(p-q)\int_0^bW_q(b-y)w_p(y) \diff y \right).
	\end{multline*}
	Substituting this in \eqref{eq_N_mathcal}, we have the result.
\end{proof}
\begin{remark} \label{remark_key_excursion}
	In particular, when $w_p-v_q$ is differentiable at $0$, we have, by \eqref{eq:Wqp0},
	\begin{align*}
	\lim_{x\downarrow 0}\frac{(w_p-v_q)(x)}{W_q(x)}&=\frac{(w_p-v_q)'(0)}{W_q'(0+)}=\frac{\sigma^2}{2}(w_p-v_q)'(0), \\
	\lim_{x\uparrow 0}\frac{(w_p-v_q)(x)}{|x|}&= -(w_p-v_q)'(0).
	\end{align*} %\red{[$w'_p(0)$ and $w'_p(0)$ may not exist individually; how about $(w_p-v_q)'(0)$?]}
	Therefore \eqref{LRZ_exc} simplifies to
	\begin{align*}
	\mathbf{n}\Big(e^{-q\tau_0^-}&(w_p-v_q)(X(\tau_0^-));\tau_0^-<\tau_b^+\Big)=-\frac{1}{W_q(b)}\left((w_p-v_q)(b)-(p-q)\int_0^bW_q(b-y)w_p(y)\diff y\right).
	\end{align*}
\end{remark}

%\green{[remove this?]Analogously to \eqref{def_big_U}, we define, for $a < 0 < b$,
%\begin{align*}
%U_1^0(a,b,\theta) &:= \mathbf{n} \Big(e^{-q\tau_0^-} u_1 (X(\tau_0^-), a, \theta);\tau_0^-<\t_b^+\Big), \quad \theta \geq 0, \\
%U_i^0(a,b) &:= \mathbf{n} \Big(e^{-q\tau_0^-} u_i (X(\tau_0^-), a );\tau_0^-<\t_b^+\Big), \quad i = 2,3, 4.
%\end{align*}}

%\green{[moved here instead of making the proof.]}
%Using the previous identities we obtain the following result .
Using Theorem \ref{lemma_key_excursion} (in particular Remark \ref{remark_key_excursion}), we obtain the excursion measure version of Corollary \ref{corollary_big_U}.
\begin{corollary} \label{corollary_big_U_zero}
For $q > 0$ and $a < 0 < b$,
\begin{align*}
 \mathbf{n} \Big(e^{-q\tau_0^-} u_1 (X(\tau_0^-), a, \theta);\tau_0^-<\t_b^+\Big) &= U_1^0(a,b,\theta)/W_q(b), \quad \theta \geq 0, \\
 \mathbf{n} \Big(e^{-q\tau_0^-} u_i (X(\tau_0^-), a );\tau_0^-<\t_b^+\Big) &= U_i^0(a,b)/W_q(b), \quad i = 3, 4.
\end{align*}
\end{corollary}
\begin{proof}

By Lemma \ref{lemma_u_in_V}, %\red{[added tilde for u]}
	$\tilde{u}_1, \tilde{u}_3 \in \tilde{\mathcal{V}}_0^{(q+r)}$. In addition, $\tilde{u}_1$ and $\tilde{u}_3$ vanish and are differentiable at $0$.  Therefore we can apply Theorem \ref{lemma_key_excursion} (with $p=q+r$ and $v_q=0$) to obtain the result for $i=1,3$.
	
For the remaining case, it is not difficult to check, using monotone convergence, that %(by the convexity of $\exp$ and that $X(\tau_0^-)$ is negative)?  $\mathbf{n}$ is not a finite measure and so I guess we should not use dominated convergence?]}\blue{Hi Kazu: Yeah $\mathbf{n}$ is not a finite measure but I was thinking that we can bound it by
%\[
%|a|\mathbf{n} \Big(e^{-q\tau_0^-};\tau_0^-<\t_b^+\Big).
%\]	
%which is finite. But I think monotone convergence works fine too so we can go either way? I am ok with both.} \green{JL: So you mentioned before that $\mathbf{n} \Big(e^{-q\tau_0^-};\tau_0^-<\t_b^+\Big)$ may not be finite?}
\begin{align*}
 \mathbf{n} \Big(e^{-q\tau_0^-} u_4 (X(\tau_0^-), a );\tau_0^-<\t_b^+\Big) = \lim_{\theta \downarrow 0} \frac \partial {\partial \theta} \mathbf{n} \Big(e^{-q\tau_0^-} u_1 (X(\tau_0^-), a, \theta );\tau_0^-<\t_b^+\Big).
\end{align*}
Therefore, using \eqref{new_iden_u4},
\[
\mathbf{n} \Big(e^{-q\tau_0^-} u_4 (X(\tau_0^-), a );\tau_0^-<\t_b^+\Big) =\frac{1}{W_q(b)}\lim_{\theta \downarrow 0} \frac \partial {\partial \theta}U_1^0(a,b,\theta)=\frac{1}{W_q(b)}U_4^0(a,b).
\]
%\green{[JL: Here too, I think we can obtain $\mathbf{n} \Big(e^{-q\tau_0^-} u_4 (X(\tau_0^-), a );\tau_0^-<\t_b^+\Big) = \lim_{\theta \downarrow 0} \frac \partial {\partial \theta} \mathbf{n} \Big(e^{-q\tau_0^-} u_1 (X(\tau_0^-), a, \theta );\tau_0^-<\t_b^+\Big) = \lim_{\theta \downarrow 0} \frac \partial {\partial \theta}U_1^0(a,b,\theta)/W_q(b)$, which has been already computed above.
%]}
%So putting the pieces together we get the result.
\end{proof}

\subsection{Auxiliary results}

%We are now ready to show the proof of the theorems in the unbounded variation case.

In order to show Theorems \ref{theorem_laplace} and  \ref{expected_bailouts} for the unbounded variation case, we shall first obtain some auxiliary results.
%\green{[Here, I copied the facts common for all the proofs below. Is it ok to use this and shorten the individual proof?]}

Fix $b > 0$ and $q > 0$.
Let us consider the event
\begin{align}
E_B:=\{\zeta^->\mathbf{e}_{\p}\}\cup\{\zeta>\tau_b^+ \}\cup\{\zeta>\tau_a^{-}\}, \label{E_B_laplace}
\end{align}
where $\mathbf{e}_{\p}$ is an independent exponential clock with rate $\p$,  $\zeta$ is the length of the excursion from the point it leaves $0$ and returns back to $0$, and $\zeta^-$ denotes the length of the negative component of the excursion.
 That is, $E_B$ is the event in which (1) the exponential clock $\mathbf{e}_{\p}$ that starts once the excursion becomes negative  rings before the excursion ends, (2) the excursion exceeds the level $b>0$, or (3) it goes below $a < 0$.  Due to the fact that $X$ is spectrally negative, once an excursion gets below zero, it stays until it ends at $\zeta$; consequently, if we denote, by $\Theta_t$, the shift operator at time $t\ge 0$, then $\zeta^- = \tau_0^+ \circ \Theta_{\tau_0^-}$. %\green{JL: I forgot to write $\Theta$; changed the other parts.}

%\par \green{[I think these sentences look too similar to your paper with JC and Victor.  Maybe we can simplify by avoiding new notations?]} \blue{I agree Kazu, can you propose whcih notations to avoid?} \green{``bail-out excursion and $T_B^e$ can be safely avoidable?}\blue{Yeah I think so Kazu.}
%\green{[JL: Could you check if I oversimplified?] }
Now let us denote by $T_{E_B}$ the first time an excursion in the event $E_B$  occurs, and also denote by
\begin{align}
l_{T_{E_B}} :=\sup\{t< T_{E_B}: X(t)=0\}, \label{g_E_B}
\end{align}
%\green{[avoid using $g$?  How about $\underline{T}_{E_B}$?]}
%\blue{Kazu: I think $g$ is used because of gauche (left in french) so I used $l$ for left, is this ok?}
the left extrema of the first excursion on $E_B$.
%and  we call {\it bail-out excursion} \green{[I guess we should name it when $E_B$ is defined?]} for the excursion where  bail-out occurs or has exceeded the level $b$.
%We  also denote \green{ [about $\Theta_t$ removed]}
%by  $T^e_B$ \green{[I guess we don't need to define$T^e_B$?]} for   the time when the bail-out occurs in the bail-out excursion and observe that  we can write  the time to the first bail-out $T_{E_B}$ %\green{[should this be called bankruptcy?]}
On the event %$\{\tau_b^+>T_{E_B}\},$ %\green{[do we need to change to $\{\tau_b^+ \wedge \tau_a^- >T_{E_B}\}$.  Actually do we need only need
$\{l_{T_{E_B}} < \infty \}$, we have
\begin{align}
T_{E_B}=l_{T_{E_B}}+T_{E_B}\circ\Theta_{l_{T_{E_B}}}. \label{g_E_B_recursion} %=g_{T_{E_B}}+T^e_B.
\end{align}

Let $(e_{t}; t\geq 0)$ be the point process of excursions away from $0$ and $V:=\inf\{t>0: e_{t}\in E_B\}.$  %\green{[I rephrased]}
By, for instance, Proposition 0.2 in  \cite{B},  $(e_{t}, t<V)$ is independent of $(V,e_{V})$. The former is a Poisson point process with characteristic measure $\mathbf{n}(\cdot \cap E_B^c)$ and ${V}$ follows an exponential distribution with parameter $\mathbf{n}(E_B).$ Moreover, we have that $l_{T_{E_B}}=\sum_{s<V}\zeta(e_{s})$, where $\zeta(e_{s})$ denotes the lifetime of the excursion $e_{s}$. % \green{[So each time an excursion ends it may not end up at $0$. Is this ok?]}\blue{Kazu: No I think each excursion must end at zero. But we are modifying the process by the bail out so when exponential clock rings in each excursion to do the bail-out we only need to know
%	\[ \mathbf{n}\left(e^{-q\tau_0^-}\E_{X_{\tau_0^-}}\left(e^{-q\mathbf{e}_{\p}}(-X_{\mathbf{e}_{\p}});\tau_0^+>\mathbf{e}_{\p}\right);\tau_0^-<\tau_b^+\Big | E_B\right)
%	\]
%	and
%	\[
%	\mathbf{n}\left(e^{-q\tau_0^-}\E_{X_{\tau_0^-}}\left(e^{-q\mathbf{e}_{\p}};\tau_0^+>\mathbf{e}_{\p}\right);\tau_0^-<\tau_b^+\Big | E_B\right)
%	\]
%	and apply the Markov property.
% 	} \green{I see; thanks for clarification.}
Therefore,  the exponential formula  for Poisson point processes (see for instance Section 0.5 in \cite{B} or Proposition 1.12 in Chapter XII in \cite{RY}) and the independence between $(e_{t}, t<V)$ and  $(V,e_{V})$ implies %\green{[removed the line changes below]}
\begin{align} \label{laplace_T_E_B}
\begin{split}
\e\Big( e^{-q l_{T_{E_B}}}\Big) &=\e\Big( \exp\Big\{-q \sum_{s<V}\zeta(e_{s})\Big\}\Big)=\mathbf{n}(E_B)\int_0^\infty e^{-s[\mathbf{n}(E_B)+\mathbf{n}(1-e^{-q\zeta}; E_B^c)]}\ud s\\
&=\frac{\mathbf{n}(E_B)}{\mathbf{n}(E_B)+\mathbf{n}\left(\mathbf{e}_q<\zeta, E_B^c\right)}=\frac{\mathbf{n}(E_B)}{\mathbf{n}(E_1)+\mathbf{n}(E_2)+\mathbf{n}(E_3)},
 \end{split}
\end{align}
 where $\mathbf{e}_q$ is an exponential random variable with parameter $q$ that is independent of $\mathbf{e}_r$ and $X$, and
 %\green{
% \begin{align*}
% E_1 &:= \{\mathbf{e}_q<\zeta\}\cup\{\tau_b^+<\zeta\}, \\ E_2&:= \{ \mathbf{e}_q>\zeta,\tau_a^-<\zeta,\tau_b^+>\zeta \}, \\ E_3 &:= \{\mathbf{e}_q>\zeta,\zeta^->\mathbf{e}_{\p},\tau_a^->\zeta,\tau_b^+>\zeta \}.
 %\end{align*}%}
 %\green{change to?
  \begin{align*}
 E_1 &:= \{\mathbf{e}_q<\zeta\}\cup\{\tau_b^+<\zeta\}, \\ E_2&:= \{ \mathbf{e}_q>\zeta,\tau_a^-<\zeta < \tau_b^+ \}, \\ E_3 &:= \{\mathbf{e}_q>\zeta,\zeta^->\mathbf{e}_{\p},\tau_a^- \wedge \tau_b^+>\zeta \}.
 \end{align*}%}
 To see how the last equality of \eqref{laplace_T_E_B} holds, we have
 \begin{multline*}
 \mathbf{n}(E_B)+\mathbf{n}\left(\mathbf{e}_q<\zeta, E_B^c\right) =  \mathbf{n}(\mathbf{e}_q < \zeta) + \mathbf{n}(\mathbf{e}_q > \zeta, E_B) \\= \mathbf{n}(E_1) - \mathbf{n} (\mathbf{e}_q > \zeta, \tau_b^+ < \zeta) + \mathbf{n}(\mathbf{e}_q > \zeta, E_B)
 =\mathbf{n}(E_1)+\mathbf{n}(E_2)+\mathbf{n}(E_3).
 \end{multline*}

 \begin{lemma} \label{lemma_E_123} % \blue{[Kazu: I modified the Lemma and add (iv).]}
For $q > 0$ and $b > 0$,  we have  \begin{itemize}
 \item[(i)] $\displaystyle\mathbf{n}(E_1)
 %\Big(\{\zeta>\mathbf{e}_{q}\}\cup \{\tau_{b}^+<\zeta\}\Big)
 = e^{\Phi_{\q}b} / W_{\q}(b)$,
\item[(ii)]$\displaystyle\mathbf{n}(E_2)=-\frac{1}{W_q(b)}\left(e^{\Phi_q b}-\frac{W_q(b-a)}{W_q(-a)}\right)$,
\item[(iii)]$\displaystyle \mathbf{n}\left(E_3\right)=-\frac{1}{W_q(b)}\left(\frac{W_q(b-a)}{W_q(-a)}-\frac{W_{q,r}^a(b)}{W_{q+r}(-a)}\right)$,
\item[(iv)] $\mathbf{n} \big(e^{-q\tau_b^+}; \tau_0^- > \tau_b^+ \big) = W_q(b)^{-1}$.
\end{itemize}
 \end{lemma}
 \begin{proof}
(i)  Let $\mathbf{L} = (\mathbf{L}(t); t \geq 0)$ %\red{[change to $\mathbf{L}(t)$?]}
be the local time at zero of the process $X$ and observe that $\mathbf{L}(\tau^+_b\land \mathbf{e}_q)$ is the first time where the Poisson point process $(e_{t}, t\geq 0)$ enters the set $E_1$.
 % $\{\epsilon: h(\epsilon)>b \textrm{ or } \zeta(\epsilon)>\mathbf{e}_q\}$, %\green{[does this include an excursion where it stays on $(0,b)$ for more than $e_q$?]} \blue{Kazu: Yes I think it does.} \green{I see.}
 % where $h$ denotes the height of the excursion.
 In other words,  $\mathbf{L}(\tau^+_b\land \mathbf{e}_q)$ is the first time  an excursion that goes above the level $b$ or such that its length is bigger than $\mathbf{e}_q$ starts. Using Proposition 0.2 in  \cite{B}, we deduce that  $\mathbf{L}(\tau^+_b\land \mathbf{e}_q)$ is exponentially distributed with parameter
 \begin{equation}
 \mathbf{n}(E_1)
 %\Big(\{\zeta>\mathbf{e}_{q}\}\cup \{\tau_{b}^+<\zeta\}\Big)
 =\e\Big( \mathbf{L}(\tau_{b}^+\wedge\mathbf{e}_q)\Big)^{-1}.\notag
 \end{equation}
%Recall \eqref{u_q_resolvent}. 
By the identity  (2.25) of \cite{PPR1} we have, with $u_q$ defined in \eqref{u_q_resolvent},
 \begin{equation}\label{localtb}
 \e\Big( \mathbf{L}(\tau_{b}^+\wedge\mathbf{e}_q)\Big)=\mathbb{E}\left( \int_0^{\tau_b^+}e^{-qs}\ud \mathbf{L}(s)\right)=u_q(0)-\frac{u_q(b)u_q(-b)}{u_q(0)}=e^{-\Phi_{\q}b}W_{\q}(b).
 \end{equation}
 Hence, we have the result.

 (ii) By the memoryless property of $\mathbf{e}_q$, \begin{align*}
\mathbf{n}(E_2) &= \mathbf{n}\left(e^{-q\tau_0^-}\mathbb{P}_{X(\tau_0^-)}
\left(  \tau_a^- < \tau_0^+ < \tilde{\mathbf{e}}_q \right); \tau_0^-<\t_b^+\right),
\end{align*}
where $\tilde{\mathbf{e}}_q$ is an independent copy of $\mathbf{e}_q$.
%\green{JL: Above, do we need $\tau_0^- > 0$ in the indicator?}
In this case, if we define $w_q(x) := e^{\Phi_q x}- W_q(x-a) / W_q(-a)$, then $w_q(x) = \mathbb{P}_{x}
\left(\tilde{\mathbf{e}}_q > \tau_0^+ >  \tau_a^- \right)$ for all $x \leq 0$, $w_q \in\mathcal{V}_0^{(q)}$, $w_q(0)=0$, and is differentiable at $0$. Hence by Theorem \ref{lemma_key_excursion} for $p=q$ and $v_q = 0$, we obtain
%therefore we can apply Lemma \ref{lemma_key_excursion} (with $p=q$ and $v_q=0$) we obtain
\begin{align*}
\mathbf{n}(E_2)=-\frac{1}{W_q(b)}w_q(b)=-\frac{1}{W_q(b)}\left(e^{\Phi_q b}-\frac{W_q(b-a)}{W_q(-a)}\right).
\end{align*}

(iii) Again, by the memoryless property of $\mathbf{e}_q$,
\begin{align*}
\mathbf{n}\left(E_3\right)& =
\mathbf{n}\left(e^{-q\tau_0^-}\mathbb{P}_{X(\tau_0^-)}
\left( \tilde{\mathbf{e}}_q> \tau_0^+, \tau_a^- > \tau_0^+ > \mathbf{e}_r \right); \tau_0^-<\t_b^+\right)\\
&=\mathbf{n}\left(e^{-q\tau_0^-}\E_{X(\tau_0^-)}
\left(e^{-q\tau_0^+}-e^{-(\q+\p)\tau_0^+};\tau_0^+<\tau_a^-\right); \tau_0^-<\t_b^+\right).
\end{align*}
If we set
\begin{align*}
v_q(x):=\frac{W_q(x-a)}{W_q(-a)}\qquad\text{and}\qquad w_{q+r}(x):=\frac{W_{q+r}(x-a)}{W_{q+r}(-a)},
\end{align*}
then, for $x\leq0$, we have $v_q(x)=\E_{x}
\big( e^{-q\tau_0^+};\tau_0^+<\tau_a^-\big)$ and $w_{q+r}(x)=\E_{x}
\big(e^{-(\q+\p)\tau_0^+};\tau_0^+<\tau_a^-\big)$,
which satisfy the conditions of Lemma 2.1 (with $p=q+r$). Therefore, Theorem \ref{lemma_key_excursion}  shows the result.
%by applying this result we get the result. %\green{[I guess we need to rewrite this too.]}
%\begin{align*}
%\mathbf{n}\left(E_3\right)=-\frac{1}{W_q(b)}\left(\frac{W_q(b-a)}{W_q(-a)}-\frac{W_{q,r}^a(b)}{W_{q+r}(-a)}\right).
%\end{align*}

(iv) By a small modification of the proof of Theorem 3 (ii) in \cite{PPR2}, it is not difficult to see that
\begin{align*}
%\hat{g}_0(0,b)=
\mathbf{n}\left(e^{-q\tau_b^+};  \tau_b^+ < \tau_0^- \right) = \mathbf{n}\left(e^{-q\tau_b^+};\tau_b^+<\zeta\right)=\underline{\mathbf{n}}\left(e^{-q\tau_b^+};\tau_b^+<\zeta\right),
\end{align*}
where $\underline{\mathbf{n}}$ is the excursion measure of the process reflected at its infimum. Now by Proposition 1 in \cite{CD}, which establishes that the measure $\underline{\mathbf{n}}$ can be constructed as a limit of the law of $X$ killed at its first passage time above $0$, and by  \eqref{laplace_in_terms_of_z}, %\green{[we should also refer to \cite{PPR2} how the denominator $W(x)$ is chosen]} \blue{Is not sufficient to use \cite{CD}?} \green{[it does not specifically say the denominator is $W(x)$ does it?]}
\begin{align*}
%\textbf{n}\left(e^{-q\tau_b^+};\tau_b^+<\zeta\right)=
\underline{\mathbf{n}}\left(e^{-q\tau_b^+};\tau_b^+<\zeta\right)=\lim_{x \downarrow 0}\frac{\mathbb{E}_x\left(e^{-q\tau_b^+};\tau_b^+<\tau_0^-\right)}{W(x)}=\lim_{x \downarrow 0}\frac{W_q(x)}{W(x)W_q(b)}=\frac{1}{W_q(b)}.
\end{align*}
To see how the last equality holds, by for instance the identity (3.6) in page 132 of \cite{KKR}, the scale function admits a series representation
$W_q(x)=\sum_{j\geq 0}q^{j}W^{*(j+1)}(x),$
where $W^{*(k)}$ denotes the $k$-th convolution of $W$ with itself.  In addition, because $W$ is increasing, $W^{*(k+1)}(x)\leq x W(x)W^{*(k)}(x)$, $x>0$. Hence, we deduce that
\begin{align*}
0 \leq \frac {W_q(x) - W(x)} {W(x)}= W(x)^{-1}\sum_{j\geq 1}q^{j}W^{*(j+1)}(x)\leq x  \sum_{j\geq 0}q^{j}W^{*(j+1)}(x) \leq x  W_q(x)\xrightarrow{x \downarrow 0} 0.
\end{align*}
In sum, $\mathbf{n} \big(e^{-q\tau_b^+};   \tau_b^+ < \tau_0^- \big) = W_q(b)^{-1}$.
 \end{proof}

 By Lemma \ref{lemma_E_123}, we have %\green{[changed to inline.]}
 $\mathbf{n}(E_1) + \mathbf{n}(E_2) + \mathbf{n}\left(E_3\right)
 = {W_{q,r}^a(b)} / [W_q(b) W_{q+r}(-a)]$.
Hence, by \eqref{laplace_T_E_B},\begin{align}\label{laplace_T_E_B_2}
\begin{split}
\e\Big( e^{-q l_{T_{E_B}}}\Big) &=\frac  {\mathbf{n}(E_B) W_q(b)  W_{q+r}(-a)} {W_{q,r}^a(b)}.
 \end{split}
\end{align}

\subsection{Proof of Theorem \ref{theorem_laplace} for the  unbounded variation case}

We show for the case $q > 0$.  The case $q = 0$ holds by monotone convergence.

\subsubsection{Proof of \eqref{Parisbailouts}}
On the event %$\{\tau_b^+>T_{E_B}\},$ %\green{[do we need to change to $\{\tau_b^+ \wedge \tau_a^- >T_{E_B}\}$.  Actually do we need only need
$\{l_{T_{E_B}} < \infty \}$,
%\begin{align}
%T_{E_B}=l_{T_{E_B}}+T_{E_B}\circ\Theta_{l_{T_{E_B}}}. \label{g_E_B_recursion} %=g_{T_{E_B}}+T^e_B.
%\end{align}
before $T_{E_B}$, there is no contribution to the cumulative bail-outs (because no exponential clock has rung before the excursion ends) and hence $X_r(t) = X(t)$, $0 \leq t \leq l_{T_{E_B}}$.  %\green{$l_{T_{E_B}}$?}
	Therefore, using the memoryless property of the Poisson arrival times,
  \begin{align}\label{eqwexc1_g}
 g(0,a,b,\theta)&= g_0(0,b) + [g_1(0,a,b,\theta)+ g_2(0,a,b,\theta)] g(0,a,b,\theta),
 \end{align}
where, with $\tilde{T}_0^- :=l_{T_{E_B}}+\tau_0^- \circ \Theta_{l_{T_{E_B}}}$, %\red{[change $T_0^-$ to say $\blue{\tilde{T}_0^-}$? It is sort of used already.]}
 \begin{align*}
 \begin{split}
    g_0(0, b) &:= \E\Big( e^{-q \tau_b^+}; \tilde{T}_0^- > \tau_b^+ \Big), \\
   g_1(0,a,b,\theta) &:= \E\Big( e^{-q (\tilde{T}_0^- +\mathbf{e}_{\p}) + \theta X(\tilde{T}_0^-+\mathbf{e}_{\p})}; \tilde{T}_0^- <\tau_b^+, (\tau_0^+ \circ \Theta_{\tilde{T}_0^-}) \wedge (\tau_a^- \circ \Theta_{\tilde{T}_0^-}) > \mathbf{e}_{\p} \Big), \\
   g_2(0,a,b,\theta) &:=  \E\Big( e^{-q (\tilde{T}_0^- + \tau_0^+ \circ \Theta_{\tilde{T}_0^-} ) %+ \theta X(\tilde{T}_0^-+\tau_0^+ \circ \Theta_{\tilde{T}_0^-} )
   	}; \tilde{T}_0^- <\tau_b^+, \mathbf{e}_{\p} \wedge (\tau_a^- \circ \Theta_{\tilde{T}_0^-}) > \tau_0^+ \circ \Theta_{\tilde{T}_0^-} \Big).
   \end{split}
\end{align*}
%\green{JL:In the second line $X(\mathbf{e}_{\p})$ in the exponent should be $X(\tilde{T}_0^-+\mathbf{e}_{\p})$? In the third line, $X(\tilde{T}_0^-+\tau_0^+ \circ \Theta_{\tilde{T}_0^-})$ is always zero?}

%\green{ To simplify the indicator a bit, can we change to?
 %\begin{align*}
%    f_0(0, b) &:= \E\Big[e^{-q \tau_b^+}; T_0^- > \tau_b^+ \Big], \\
%   f_1(0, b) &:= \E\Big[e^{-q (T_0^- +\mathbf{e}_{\p}) + \theta X(\mathbf{e}_{\p})}; T_0^- <\tau_b^+, (\tau_0^+ \circ \Theta_{T^-_0}) \wedge (\tau_a^- \circ \Theta_{T^-_0}) > \mathbf{e}_{\p} \Big], \\
 %  f_2(0, b) &:= \E\Big[e^{-q (T_0^- + \tau_0^+ \circ \Theta_{T^-_0} ) + \theta X(\green{T_0^-+}T^-_0+\tau_0^+ \circ \Theta_{T^-_0} )}; T_0^- <\tau_b^+, \mathbf{e}_{\p} \wedge (\tau_a^- \circ \Theta_{T^-_0}) > \tau_0^+ \circ \Theta_{T^-_0} \Big].
%\end{align*}
%}
In particular, $g_2(0,a,b,\theta) = 0$  because %$\mathbb{P} ( T_0^- <\tau_b^+, \mathbf{e}_{\p} > \tau_0^+ \circ \Theta_{T^-_0}, \tau_a^- \circ \Theta_{T^-_0} > \tau_0^+ \circ \Theta_{T^-_0} ) = 0$
 $\mathbb{P} ( \tilde{T}_0^- <\tau_b^+, \mathbf{e}_{\p} \wedge (\tau_a^- \circ \Theta_{\tilde{T}_0^-}) > \tau_0^+ \circ \Theta_{\tilde{T}_0^-} ) = 0$ in view of the definition of $\tilde{T}_0^-$ and \eqref{E_B_laplace}.
\par A simple application of the so-called Master's formula  at time $l_{T_B}$ (see for instance excursions straddling a terminal time in Chapter XII in Revuz-Yor \cite{RY}) and Corollary \ref{corollary_big_U_zero} imply %\red{[removed the shorthand notation $g_0$]}
\begin{align}
g_0 (0,b) = \e\Big( e^{-q l_{T_{E_B}}}\Big) \mathbf{n}\big(e^{-q\tau_b^+}; \tau_0^- > \tau_b^+ \big) / \mathbf{n} (E_B)\quad\text{and}\quad g_1(0,a,b,\theta) = \e\Big( e^{-q l_{T_{E_B}}}\Big) U_1^0(a,b, \theta) / [W_q(b) \mathbf{n} (E_B)]. \label{g_0_1}
\end{align}

Substituting these, \eqref{laplace_T_E_B_2}, Lemma \ref{lemma_E_123}(iv), and \eqref{U_0_summary} in \eqref{eqwexc1_g}, we obtain
 \begin{align}\label{excu_descomp}
g(0,a,b,\theta) &=  \frac {W_q(b)\mathbf{n}\big(e^{-q\tau_b^+}; \tau_0^- > \tau_b^+ \big)}  {W_{q,r}^a(b)/W_{q+r}(-a) - U_1^0(a,b, \theta)} =  \mathcal{H}^{\q,\p}_a(b , \theta)^{-1}.
 \end{align}
%Now substituting (iv) in Lemma \ref{lemma_E_123} and \eqref{U_0_summary} in \eqref{excu_descomp}, we obtain $g(0,a,b,\theta) = \mathcal{H}^{\q,\p}_a(b)^{-1}$.
Using this expression in (\ref{sdep_1}) (which also holds for the unbounded variation case), we get the result.
\subsubsection{Proof of \eqref{Parisbailouts_2}}
%We will use the following notation
%\begin{equation*}
%h(x,a,b)=\E\left[e^{-q\tau_a^-- \theta L(\tau_a^-)};\tau_a^-<\tau_b^+\right].
%\end{equation*}
By modifying the above arguments, we have
 \begin{align} \label{h_a_expression_n}
h(0,a,b,\theta) &=  \frac {U_3^0(a,b)} {W_{q,r}^a(b)/W_{q+r}(-a) - U_1^0(a,b, \theta)}.
 \end{align}
This together with \eqref{U_0_summary} gives $h(0,a,b,\theta) = - \mathcal{I}_{q,r}^a(b)/ \mathcal{H}^a_{\q,\p}(b,\th)$. %\red{[JL: I don't remember what $\mathcal{Z}^{q,r}_a$ was but it has been changed to something else?]}
%[simplified this way, is it ok?]}
%\begin{align*}
%h(0,a,b,\theta)&=\Bigg(\p\int_{0}^{-a} e^{\theta(y+a)}\Bigg[W_{\q,\p}^a(b)\frac{W_{\q+\p} (-a-y)} {W_{\q+\p}(-a)}-W_{\q,\p}^a(b-y)\Bigg]\diff y+\frac{W_{\q,\p}^a(b)}{W_{\q+\p}(-a)}\Bigg)^{-1}\\
%&\times\left(W_{\q,\p}^a(b)\frac{Z_{\q+\p}(-a)}{W_{\q+\p}(-a)}-Z_{\q,\p}^a(b)\right)\\
%&=\left(W_{\q,\p}^a(b)\frac{Z_{\q+\p}(-a)}{W_{\q+\p}(-a)}-Z_{\q,\p}^a(b)\right)\frac{1}{\mathcal{H}_a^{\q,\p}(b,\th)}.
%\end{align*}
Using this in (\ref{Exit_prob_a_new1}), we obtain the result.
\subsection{Proof of Theorem \ref{expected_bailouts} for the  unbounded variation case}
%In this section we are interested in computing the expected present value of cumulative bailouts in the unbounded variation case that is
%\[
%f(x,a,b)=\E_x\left[\int_0^{\tau_b^+\wedge\tau_a^-}e^{-qt} \diff L_{r}(t)\right].
%\]
%\green{[change to $L_r(t)$ to be consistent with other parts?]}

We shall show for the case $q > 0$; the case $q=0$ holds by monotone convergence.
%\par To this end we will use excursion theory. Let us denote by
With $g_1$ defined in \eqref{g_0_1}, using the memoryless property of the Poisson arrival times,
\begin{align}\label{eqwexc1_ga}
f(0,a,b)&= f_0(0,a,b) + g_1(0,a,b,0) f(0,a,b),
\end{align}
where%, with $T_0^- := l_{T_{E_B}}+\tau_0^- \circ \Theta_{l_{T_{E_B}}}$, %\green{[is this the same as $\tilde{T}_0^-$ above?]}
\begin{align*}
f_0(0,a, b) &:= \E\Big( e^{-q (\tilde{T}_0^- + \mathbf{e}_{\p})}(-X(\tilde{T}_0^- + \mathbf{e}_{\p})); \tilde{T}_0^- <\tau_b^+, (\tau_0^+ \circ \Theta_{\tilde{T}_0^-})\wedge(\tau_a^- \circ \Theta_{\tilde{T}_0^-}) > \mathbf{e}_{\p}\Big).% \\
%f_1(0,a, b) &:= \E\Big(e^{-q (T_0^- +\mathbf{e}_{\p})}; T_0^- <\tau_b^+, \tau_0^+ \circ \Theta_{T^-_0} > \mathbf{e}_{\p},  \tau_a^- \circ \Theta_{T^-_0} > \mathbf{e}_{\p} \Big).
\end{align*}
%\green{[I guess we can just use $g_1(0,a,b,0)$ instead of $f_1$?]}
Again, by the Master's formula in excursion theory and Corollary \ref{corollary_big_U_zero},
%similarly to our arguments in Section \ref{bailout_proof_unbounded_variation_case}, we have
\begin{align*}
f_0 (0, a,b) = \e\Big( e^{-q l_{T_{E_B}}}\Big) \widehat{f}_0(0,a,b) / \mathbf{n} (E_B),
\end{align*}
where
\begin{align} \label{g_partsa}
\widehat{f}_0(0,a,b)&:=\mathbf{n}\left(e^{-q (\t_0^- + \mathbf{e}_{\p})}(-X(\t_0^- + \mathbf{e}_{\p})); \t_0^- <\tau_b^+, (\tau_0^+ \circ \Theta_{\t^-_0})\wedge(\tau_a^- \circ \Theta_{\t^-_0}) > \mathbf{e}_{\p}\right)\notag\\
%&=\mathbf{n}\left(e^{-q\t_0^-}\E_{X(\tau_0^-)}\left(e^{-q\mathbf{e}_{\p}}(-X(\mathbf{e}_{\p})); \mathbf{e}_{\p}<\tau_0^+\wedge\tau_a^-\right);\t_0^- <\tau_b^+\right) \\
&= -U_4^0(a,b){/W_q(b)}.
\end{align}
%and
%\[
%g_1(0,a,b,0) = \e\Big[e^{-q l_{T_{E_B}}}\Big]\frac{U_1^0(a,b,0)}{W_q(b)\mathbf{n} (E_B)}.
%\]
Using this, \eqref{laplace_T_E_B_2}, and \eqref{g_0_1}, we obtain  %[deleted the above because it is exactly \eqref{laplace_T_E_B}], we obtain}
%Substituting this and \eqref{g_partsa} in \eqref{eqwexc1_ga}, we obtain
\begin{align}\label{mainexclow}
f(0,a,b) &=  \frac {-U_4^0(a,b)} {W_{q,r}^a(b)/W_{q+r}(-a) - U_1^0(a,b, 0)}. %= \frac {\hat{f}_0(0,a,b)} {\frac{1}{W_{\q}(b)}\frac{W_{\q,\p}^a(b)}{W_{\q+\p}(b-a)} - \hat{f}_1(0,a,b)}.
\end{align}
Using the above identity, \eqref{U_0_summary}, and \eqref{U_0_summary_2}, we get $f(0,a,b)=  {h_{\q,\p}^{a}(b)} / {\mathcal{H}_{\q,\p}^{a}(b,0)}$.
Substituting this in (\ref{MPa}), we obtain the result.

\section{Proofs for the case with classical reflection above} \label{section_proof_reflected}

In this section, we prove the theorems in Section \ref{subsection_reflected_case}.  Toward this end, we first show Theorem \ref{lemma_simplifying_reflected}, which is a version of Lemma 2.1 of \cite{LRZ} (see \eqref{LRZ_identity}), when the process $X$ is replaced with the reflected process $Y^b$. Using this and the results for the process $X_r$ as obtained in previous sections, the theorems can be proven by arguments via the strong Markov property. These proofs hold for both the bounded and unbounded variation cases.

\subsection{Simplifying formula for the spectrally negative \lev process reflected from above}

%\red{[moved here]}

%\red{[JL: For this, without loss of generality, we can set $a = 0$?]}\blue{Kazu: yeah sure.}

%\green{JL: Should we assume below that $v_p$ is differentiable at $b$?  Also, for the expectation to be finite, I guess we need that $\int_{(-\infty, -1]} |v_p (y)| \Pi (\diff y) < \infty$ or something; is this a condition given in $\mathcal{V}$? I think this is needed when we interchange the derivative with respect to $b$ after ``Differentiating this''.}
%\blue{[Kazu: I see, yeah I think also in the proof of Lemma 5.1 we need to assume the same. Also it is important to remark that everything is finite near $0$. Just by noting that $W_q(b-y)$ and $W_q'(b-y)$ is bounded, and the fact that $(w_p-v_q)(0)=0$ and the left hand derivative exists at zero then
%\[
%(w_p-v_q)(\th+y)\leq K(y+\th)
%\]
%when $y$ and $\th$ are close to zero and $K$ is a constant, then for $\varepsilon$ small enough
%\[
%\int_0^{\varepsilon}\int_{-\varepsilon}^{-y}(w_p-v_q)(\th+y)\Pi(d\th)dy\leq K\int_{-\varepsilon}^{0}\int_0^{-\th}(y+\th)dy\Pi(d\th)=K\int_{-\varepsilon}^{0}\th^2\Pi(d\th)<\infty.
%\]
%Or something like this?} \green{For Lemma 5.1, we are already assuming boundedness of $F$ and $G$. I think this is stronger and so, we do not need the tail integrability condition?  }\blue{[Kazu:I agree.]}

%\green{JL: I think this is an important result.  Change from Lemma to Theorem?}
%\green{[JL: I relaxed the condition from differentiability to right-hand differentiability of $v$ as we had to deal with the right-hand differentiability of $W$ anyways. Could you check?]}

	\begin{theorem} \label{lemma_simplifying_reflected} Fix $p \geq 0$ and $b > 0$. Suppose $v_p: \R \to [0, \infty)$ and belongs to $\tilde{\mathcal{V}}_0^{(p)}$.  Assume also that $v_p$ is locally bounded, right-hand differentiable at $b$ and 
	\begin{align} \label{integrability_cond_formula_reflected}
	\sup_{0 \leq y \leq b}\int_{(-\infty, -1]} v_p (y + \theta) \Pi (\diff \theta) < \infty.
	\end{align}
	 In addition, for the case of unbounded variation, in (ii) for the definition of $\tilde{\mathcal{V}}_0^{(p)}$ above, $v_{p,n}'(b+)$ converges to $v_{p}'(b+)$.
	Then, for $x \leq b$ and $q \geq 0$, %\red{[JL: deleted the indicator because $\tau_0^-$ is almost surely finite.]}
	%	\begin{align}\label{LRZ_identity_reflected}
	%\E_x \big[ e^{-q \eta_0^-} v_p(Y^b(\eta_0^-)); \eta_0^- < \infty \big]  	&= - \frac {W_q(x)} {W_q'(b +)}\Big( v_p'(b +) - (p-q) \Big[ \int_0^b W_q(b-y) v_p'(y) \diff y + W_q(b) v_p(0) \Big] \Big) \notag\\
	%&+ v_p(x) - (p-q) \int_0^x W_q(x-y) v_p(y) \diff y.
	%\end{align}
	%\green{JL: Maybe it is more intuitive to write?
	%\begin{align*}
		%\E_x \big[ e^{-q \eta_0^-} v_p(Y^b(\eta_0^-)); \eta_0^- < \infty \big]  	&= - \frac {W_q(x)} %{W_q'(b +)} \frac {\partial_+} {\partial_+ b} \Big( v_p(b) - (p-q) \int_0^b W_q(b-y) v_p(y) %\diff y \Big) \notag\\
%	&+ v_p(x) - (p-q) \int_0^x W_q(x-y) v_p(y) \diff y.
%	\end{align*}}
	%\green{Actually, do you use $\partial_+$ for right hand derivative? or prefer $ \ldots |_{y = b+}$ or something?}
	%\blue{Hi Kazu: I like the notation $\partial_+$ for the r-h derivative, for the other point maybe we should put both, to avoid computing later the derivative? Something like..
	\begin{align}\label{LRZ_identity_reflected}
	\E_x \big[ e^{-q \eta_0^-} v_p(Y^b(\eta_0^-)) \big]  	&= - \frac {W_q(x)} {W_q'(b +)} \frac {\partial_+} {\partial_+ b} \Big( v_p(b) - (p-q) \int_0^b W_q(b-y) v_p(y) \diff y \Big) \notag\\
	&+ v_p(x) - (p-q) \int_0^x W_q(x-y) v_p(y) \diff y\notag\\
	&= - \frac {W_q(x)} {W_q'(b +)}\Big( v_p'(b +) - (p-q) \Big[ \int_0^b W_q'(b-y) v_p(y) \diff y + W_q(0) v_p(b)\Big] \Big) \notag\\
	&+ v_p(x) - (p-q) \int_0^x W_q(x-y) v_p(y) \diff y.
	\end{align}	
	%Or is it too much?}  \green{Sure. It is probably better to write that way.  I applied integration by parts to change from $(p-q) \Big[ \int_0^b W_q'(b-y) v_p(y) \diff y + W_q(0) v_p(b) \Big]$ to $(p-q) \Big[ \int_0^b W_q(b-y) v_p'(y) \diff y + W_q(b) v_p(0) \Big] $ but it is probably better to stay with
%\green{JL:  better to write $(p-q) \Big[ \int_0^b W_q'(b-y) v_p(y) \diff y + W_q(0) v_p(b) \Big]$
%	because integration by parts will require the continuity of $v_p$. }
	\end{theorem}
\begin{proof}	(i) We first consider the case of bounded variation.   We also focus on the case $0 \leq x \leq b$; the case $x < 0$ is immediate.

	Using the resolvent given in Theorem 1 of \cite{pistorius2007excursion} and the compensation formula, we have
	\begin{align} \label{resolvent_expression_reflected}
	\begin{split}
	\E_x \big[ e^{-q \eta_0^-} v_p(Y^b(\eta_0^-))\big] &= \int_0^\infty \int_{(-\infty, -y)} v_p(y + \theta) \Pi (\diff \theta)  \Big[ \frac {W_q'(b-y)} {W_q'(b +)} W_q(x) - W_q(x-y)\Big] \diff y \\
	&+  W_q(x) \frac {W_q(0)} {W_q'(b+)} \int_{(-\infty, -b )} v_p(b + \theta ) \Pi (\diff \theta).
	\end{split}
	\end{align}
	By (19) of \cite{LRZ}, we have
	\begin{align}
\int_0^{b} W_q(b-y) \int_{(-\infty, -y)} v_p(y + \theta) \Pi (\diff \theta)  \diff y = c v_p (0) W_q(b) - v_p(b) + (p-q) \int_0^b W_q(b-y) v_p(y) \diff y. \label{LRZ_simplifying}
	\end{align}
	%Differentiating this,
Taking the right-hand derivative with respect to $b$,
		\begin{align} \label{LRZ_simplifying_reflected}
		\begin{split}
&\int_0^{b} W_q'(b-y)  \int_{(-\infty, -y)} v_p(y + \theta) \Pi (\diff \theta)  \diff y + W_q(0)  \int_{(-\infty, -b )} v_p(b + \theta ) \Pi (\diff \theta) \\
&=\frac {\partial_+} {\partial_+ b}\int_0^\infty W_q(b-y) \int_{(-\infty, -y)} v_p(y + \theta) \Pi (\diff \theta)  \diff y
 \\&= c v_p (0) W_q'(b+) - v_p'(b +) + (p-q) \Big[ \int_0^b W_q'(b-y) v_p(y) \diff y + W_q(0) v_p(b) \Big].
 \end{split}
	\end{align}

Here the right-hand derivative on the left-hand side can be interchanged over integrations by the following arguments.
For $\epsilon > 0$ and $0 \leq \delta < b$,  define
\begin{align*}
		K_1(\delta, \epsilon)%&:=\int_{0}^{\gamma-\beta-\delta}\int_{(y,\infty)}W^{(\tilde{q})}(y-u+\beta-\alpha)  \frac {\mathbb{W}^{(\tilde{p})}(\gamma+\epsilon-\beta-y) - \mathbb{W}^{(\tilde{p})}(\gamma-\beta-y) } \epsilon  \Pi(\diff u)\diff y, \\ 
		&:= \int_0^{b-\delta} \frac {W_q(b+\epsilon-y)  - W_q(b-y)} \epsilon \int_{(-\infty, -y)} v_p(y + \theta) \Pi (\diff \theta)  \diff y, \\
		K_2(\delta, \epsilon)%&:= \int_{\gamma-\beta-\delta}^{\gamma-\beta}\int_{(y,\infty)}W^{(\tilde{q})}(y-u+\beta-\alpha)  \frac {\mathbb{W}^{(\tilde{p})}(\gamma+\epsilon-\beta-y) - \mathbb{W}^{(\tilde{p})}(\gamma-\beta-y) } \epsilon  \Pi(\diff u)\diff y \\
		&:= \int_{b-\delta}^b \frac {W_q(b+\epsilon-y)  - W_q(b-y)} \epsilon \int_{(-\infty, -y)} v_p(y + \theta) \Pi (\diff \theta)  \diff y.
	\end{align*}
Note that for all $0 < \delta < b$, we have 
\begin{align} \label{K_sum}
K_1(0, \epsilon) = K_1(\delta, \epsilon) + K_2(\delta, \epsilon).
\end{align}  
Here, we show how $\lim_{\epsilon \downarrow 0}K_1(0, \epsilon)$ can be computed.
For the first term, for any $0 < \epsilon < \bar{\epsilon}$ for fixed $\bar{\epsilon} > 0$ and $0 < y < b - \delta$, we have a bound: $| W_q (b+\epsilon-y) - W_q(b-y) | / \epsilon \leq \sup_{\delta < z < b + \bar{\epsilon}} W_q' (z) < \infty$ (because $W_q'(z)$ is finite if $z > 0$, which is clear from \eqref{eq:Wqp0} [see  also identity (8.26) in \cite{K}]), and 
%	\begin{multline*}
%\int_{0}^{\gamma-\beta-\delta}\int_{(y,\infty)}W^{(\tilde{q})}(y-u+\beta-\alpha) \sup_{\delta < z < \gamma - \beta + \bar{\epsilon}} \mathbb{W}_+^{(\tilde{p}) \prime}(z)  \Pi(\diff u)\diff y \\
%		 \leq W^{(\tilde{q})}(\beta-\alpha) \sup_{\delta < z < \gamma - \beta + \bar{\epsilon}} \mathbb{W}_+^{(\tilde{p}) \prime}(z)  \int_{0}^{\gamma-\beta-\delta}\int_{(y,y + \beta - \alpha]} \Pi(\diff u)\diff y, 
%	\end{multline*}
%	
	\begin{align*}
	\int_0^{b-\delta} \sup_{\delta < z < b + \bar{\epsilon}} W_q' (z) \int_{(-\infty, -y)} v_p(y + \theta) \Pi (\diff \theta)  \diff y = \sup_{\delta < z < b + \bar{\epsilon}} W_q' (z) \int_0^{b-\delta}  \int_{(-\infty, -y)} v_p(y + \theta) \Pi (\diff \theta)  \diff y
	\end{align*}
is finite by \eqref{integrability_cond_formula_reflected} and because, for sufficiently small $c > 0$, by the assumption that $X$ is of bounded variation,
	\begin{align*}	
\int_{0}^{c}\int_{(-\infty, -y)} 1_{\{u > -c \}}\Pi(\diff u)\diff y = \int_{0}^c \int_{(-c, -y)} \Pi(\diff u)\diff y = \int_{(-c,0)} |u| \Pi(\diff u) < \infty.
\end{align*}
%\red{[JL: Above, I think we are assuming that $v_p$ is locally bounded. Should we assume for this theorem. Or it may be guaranteed by other assumed conditions?]}\blue{Kazu: I think we need to assume it.}

%\red{[I guess the first one should be $b$; we need to choose something larger than $b-\delta$ because we need the finiteness of $\int_0^{b-\delta}  \int_{(-\infty, -y)} v_p(y + \theta) \Pi (\diff \theta)  \diff y$.]}
%\blue{Kazu: It is not important but for the first integral to make sense shouldn't we also have the indicator that $y\leq c$ in the first integral?}
%\blue{[Kazu: I believe its not that clear why $\int_0^{b-\delta}  \int_{(-\infty, -y)} v_p(y + \theta) \Pi (\diff \theta)  \diff y<\infty$]
%\begin{align*}
%\int_0^{b-\delta}  \int_{(-\infty, -y)} v_p(y + \theta) \Pi (\diff \theta)  \diff y=\int_0^{b-\delta}  \int_{(-\infty, -1)} v_p(y + \theta) \Pi (\diff \theta)  \diff y+\int_0^{b-\delta}  \int_{(-1,-y)} v_p(y + \theta) \Pi (\diff \theta)  \diff y\\
%\end{align*}		
%The second integral is finite by the argument above and the fact that $v_p$ is bounded in $(-1,0)$ (right hand differentiable). But for the second integral we only know that $\int_{(-\infty, -1)} v_p(y + \theta) \Pi (\diff \theta)$ is finite for each $y\in[0,b]$. Should we ask that
%\[
%\sup_{0\leq y\leq b}\int_{(-\infty, -1)} v_p(y + \theta) \Pi (\diff \theta)<\infty?
%\]
%	}
Therefore, by dominated convergence, the limit as $\epsilon \downarrow 0$ can be interchanged over the integral and hence,
%	\begin{align*}
%		K:= \lim_{\delta \downarrow 0}\lim_{\epsilon \downarrow 0} K_1(\delta, \epsilon) &=\lim_{\delta \downarrow 0}\int_{0}^{\gamma-\beta-\delta}\int_{(y,\infty)}W^{(\tilde{q})}(y-u+\beta-\alpha)  \mathbb{W}_+^{(\tilde{p}) \prime}(\gamma-\beta-y) \Pi(\diff u)\diff y\\
%		&=\int_{0}^{\gamma-\beta}\int_{(y,\infty)}W^{(\tilde{q})}(y-u+\beta-\alpha)  \mathbb{W}_+^{(\tilde{p}) \prime}(\gamma-\beta-y) \Pi(\diff u)\diff y.
%	\end{align*}
%	
		\begin{align*}
		K:= \lim_{\delta \downarrow 0}\lim_{\epsilon \downarrow 0} K_1(\delta, \epsilon) &=\lim_{\delta \downarrow 0}\int_0^{b-\delta} W_q'(b-y)   \int_{(-\infty, -y)} v_p(y + \theta) \Pi (\diff \theta)  \diff y \\
		&=\int_0^{b} W_q'(b-y)   \int_{(-\infty, -y)} v_p(y + \theta) \Pi (\diff \theta)  \diff y.
	\end{align*}
On the other hand, by Fubini's theorem and because $k(\delta) := \sup_{0 \leq u \leq b}\int_{(-\infty, -b + \delta)} v_p(u + \theta)  \Pi(\diff \theta) < \infty$, $0 < \delta < b$, by \eqref{integrability_cond_formula_reflected},
%	\begin{align*}
%	0 \leq K_2(\delta,\epsilon) & \leq W^{(\tilde{q})}(\beta-\alpha)\Pi(\gamma-\beta-\delta,\gamma-\alpha) \frac 1 \epsilon \int_{\gamma-\beta-\delta}^{\gamma-\beta} \int_0^\epsilon \mathbb{W}^{(\tilde{p})\prime}_+(\gamma+z-\beta-y) \diff z \diff y\\
%	&= W^{(\tilde{q})}(\beta-\alpha)\Pi(\gamma-\beta-\delta,\gamma-\alpha)\frac{1}{\epsilon}\int_0^{\epsilon}( \mathbb{W}^{(\tilde{p})}(\delta+z) - \mathbb{W}^{(\tilde{p})}(z)) \diff z\\
%	&\leq W^{(\tilde{q})}(\beta-\alpha)\Pi(\gamma-\beta-\delta,\gamma-\alpha)\sup_{0\leq z\leq \epsilon}|\mathbb{W}^{(\tilde{p})}(\delta+z) - \mathbb{W}^{(\tilde{p})}(z)|.
%	\end{align*}
%	
\begin{multline*}
	0 \leq K_2(\delta, \epsilon)
		%&= \int_{b-\delta}^b \frac {W_q(b+\varepsilon-y)  - W_q(b-y)} \varepsilon \int_{(-\infty, -y)} v_p(y + \theta) \Pi (\diff \theta)  \diff y \\
		\leq  k(\delta) \int_{b-\delta}^b \frac {W_q(b+\epsilon-y)  - W_q(b-y)} \epsilon \diff y = k(\delta)  \frac 1 \epsilon \int_{b-\delta}^b  \int_0^\epsilon W_q'(b+z-y)  \diff z  \diff y \\
= k(\delta)\frac 1 \epsilon   \int_0^\epsilon (W_q(\delta+z) - W_q(z))  \diff z  \leq k(\delta) \sup_{0 \leq z \leq \epsilon}|W_q(\delta+z) - W_q(z)|  \diff z.   
		\end{multline*}
	Hence, noting that $\lim_{\delta \downarrow 0} \lim_{\epsilon \downarrow 0}K_2(\delta,\epsilon)=0$ and by \eqref{K_sum},
	\begin{multline*}
K  = \lim_{\delta \downarrow 0}\liminf_{\epsilon \downarrow 0} (K_1(\delta, \epsilon) + K_2 (\delta, \epsilon)) = \liminf_{\epsilon\downarrow 0} K_1(0, \epsilon)  \leq  \limsup_{\epsilon\downarrow 0} K_1(0, \epsilon) \\ = 
	\lim_{\delta \downarrow 0}\limsup_{\epsilon\downarrow 0} (K_1(\delta, \epsilon) + K_2 (\delta, \epsilon)) = K,
	\end{multline*}
implying $K_1(0, \epsilon) \xrightarrow{\epsilon \downarrow 0} K$, as desired.

Now, substituting \eqref{LRZ_simplifying} and \eqref{LRZ_simplifying_reflected} in \eqref{resolvent_expression_reflected}, we have %\green{[deleted the second and third equalities as it was trivial.]}
	\begin{align*}
	\E_x \big[ e^{-q \eta_0^-} v_p(Y^b(\eta_0^-)) \big] &=  \frac {W_q(x)} {W_q'(b+)}\Big( c v_p (0) W_q'(b+) - v_p'(b +) + (p-q) \Big[ \int_0^b W_q'(b-y) v_p(y) \diff y + W_q(0) v_p(b) \Big] \Big) \\
	&- \Big(c v_p (0) W_q(x) - v_p(x) + (p-q) \int_0^x W_q(x-y) v_p(y) \diff y \Big),
	%&=  \frac {W_q(x)} {W_q'(b\green{+})}\Big( - v_p'(b) + (p-q) \Big[ \int_0^b W_q'(b-y) v_p(y) \diff y + W_q(0) v_p(b) \Big] \Big) \\
	%&+ \Big(v_p(x) - (p-q) \int_0^x W_q(x-y) v_p(y) \diff y \Big) \\
%		&=  \frac {W_q(x)} {W_q'(b \green{+})}\Big( - v_p'(b \green{+}) + (p-q) \Big[ \int_0^b W_q(b-y) v_p'(y) \diff y + W_q(b) v_p(0) \Big] \Big) \\
%	&+ \Big(v_p(x) - (p-q) \int_0^x W_q(x-y) v_p(y) \diff y \Big).
	\end{align*}
%which reduces to the desired expression.
as desired.

(ii) %[JL: I elaborated the proof even more.  Could you check?]
The unbounded variation case can be similarly shown by approximation methods by \cite{LRZ}.  With the convergent sequence $X^n$ for $X$, the corresponding reflected processes $Y^{b,n}$ also converge uniformly in compacts to $Y^b$ (using the fact that $Y^b$ can be written as the difference between $X$ and its running supremum). Hence, it suffices to take the limit in the equality \eqref{LRZ_identity_reflected} for $Y^{b,n}$: 
	\begin{align*}
	\E_x \big[ e^{-q \eta_{0,n}^-} v_{p,n}(Y^{b,n}(\eta_{0,n}^-)) \big]  
	&= - \frac {W_{q,n}(x)} {W_{q,n}'(b +)}\Big( v_{p,n}'(b +) - (p-q) \Big[ \int_0^b W_{q,n}'(b-y) v_{p,n}(y) \diff y + W_{q,n}(0) v_{p,n}(b)\Big] \Big) \notag\\
	&+ v_{p,n}(x) - (p-q) \int_0^x W_{q,n}(x-y) v_{p,n}(y) \diff y,
	\end{align*}	
where $\eta_{0,n}^-$ and $W_{q,n}$ correspond to those for $Y^{b,n}$. Similarly to the arguments in the proof of Lemma 2.1 of \cite{LRZ}, the left hand side converges to $\E_x \big[ e^{-q \eta_{0}^-} v_{p}(Y^{b}(\eta_0^-))\big]$ %\red{[remove $; \eta_{0}^- < \infty$?]}
 thanks to the uniform convergence assumed in (ii) in the definition of $\tilde{\mathcal{V}}^{(p)}$.  Regarding the convergence on the right hand side, it is known that $W_{q,n}(x) \rightarrow W_{q}(x)$ %\red{[$W_{q}(x)$?]} 
 for $x \in \R$, and $W_{q,n}'(x) \rightarrow W_{q}'(x)$ %\red{[$W_{q}'(x)$?]} 
 for $x > 0$, as in Remark 3.2 of \cite{PY20015b}.  In addition, triangle equality gives
 \begin{align*}
& \Big| \int_0^b [W_{q,n}'(b-y) v_{p,n}(y) - W_{q}'(b-y) v_{p}(y)] \diff y  \Big|\\
 &\leq  \int_0^b W_{q,n}'(b-y) |v_{p,n}(y) - v_{p}(y)| \diff y 
 + \Big| \int_0^b (W_{q,n}'(b-y) - W_{q}'(b-y)) v_{p}(y) \diff y \Big| \\
 &\leq  (W_{q,n}(b) - W_{q,n} (0)) \sup_{0 \leq y \leq b} |v_{p,n}(y) - v_{p}(y)|
 + \big( |W_{q,n}(b) - W_{q}(b) | +  |W_{q,n}(0) - W_{q}(0) | \big) \sup_{0 \leq y \leq b}v_{p}(y),
 \end{align*}
 which vanishes as $n \uparrow \infty$ by the assumed uniform convergence of $v_{p,n}$.
	\end{proof}
	
%		\begin{example} \red{[this is for confirmation; to be removed.  ]}\blue{remove it Kazu?} Suppose $v_q (x)= Z_q(x) \in \mathcal{V}_0^{(q)}$.  Then,
%	\begin{align*}
%	\E_x [e^{-q \eta_0^-} ] = \E_x [e^{-q \eta_0^-} Z_q(Y^b_{\eta_0^-})] =  -\frac {W_q(x)} {W_q'(b)} q W_q(b) + Z_q(x).
%	\end{align*}
%	Suppose $v_q (x)= \overline{Z}_q(x) + \kappa'(0+) / q \in \mathcal{V}_0^{(q)}$.  Then,
%			\begin{align*}
%	\E_x \Big[e^{-q \eta_a^-}  (\overline{Z}_q(Y^b(\eta_0^-)) + \frac {\kappa'(0+)} q \Big]
%	= \E_x \Big[e^{-q \eta_a^-}  (Y^b(\eta_0^-) + \frac {\kappa'(0+)} q \Big] \\
%		=  -\frac {W_q(x)} {W_q'(b)} Z_q(b) + \overline{Z}_q(x) + \frac {\kappa'(0+)} q.
%	\end{align*}
%	Hence,
%	\begin{align*}
%	\E_x \Big[e^{-q \eta_a^-}  Y^b(\eta_0^-)\Big] = -\frac {W_q(x)} {W_q'(b)} Z_q(b) + \overline{Z}_q(x) + \frac {\kappa'(0+)} q - \frac {\kappa'(0+)} q \Big(-\frac {W_q(x)} {W_q'(b)} q W_q(b) + Z_q(x) \Big) \\
%	= -\frac {W_q(x)} {W_q'(b)} Z_q(b) + \overline{Z}_q(x)  +\kappa'(0+) \Big(\frac {W_q(x)} {W_q'(b)}  W_q(b) -  \overline{W}_q(x) \Big)
%\end{align*}
%	\end{example}

By Theorem \ref{lemma_simplifying_reflected} together with Lemma \ref{lemma_u_in_V}, we can compute, for $a < 0 < b$ and $x \leq b$,
\begin{align*}
\begin{split}
\widehat{U}_1(x, a,b,\theta) &:= \E_x\Big(e^{-q\eta_0^-} u_1 (Y^b(\eta_0^-), a, \theta); \eta_0^- < \infty \Big), \quad \theta \geq 0, \\
\widehat{U}_i(x, a,b) &:= \E_x\Big(e^{-q\eta_0^-} u_i (Y^b(\eta_0^-), a ); \eta_0^- < \infty \Big), \quad i = 2,3, 4.
\end{split}
\end{align*}
For $a < 0$ and $x \in \R$, we define the derivatives of \eqref{U_0_summary} with respect to $x$:
\begin{align} \label{U_0_hat}
\begin{split}
\widehat{U}_1^{0}(a,x, \theta)&:=-\left(\p\int_{0}^{-a} e^{-\theta u}\Big[ \frac {(W^{a }_{\q,\p})'(x) W_{\q+\p} (u)} {W_{\q+\p}(-a)} -(W^{-u }_{\q,\p})' (x) \Big]\diff u\right), \quad \theta \geq 0, \\
\widehat{U}_2^{0}(a,x) &:= - (W^{a }_{\q,\p})' (x) / W_{q+r}(-a), \\
\widehat{U}_3^{0}(a,x) &:= - \mathcal{I}_{q,r}^{a \prime} (x) , \\
\widehat{U}_4^{0}(a,x) &:= - h^{a \prime}_{q,r}(x).
\end{split}
\end{align}
In particular,
\begin{align*}
\widehat{U}_1^{0}(a, x, 0) &=- r \frac { \overline{W}_{q+r} (-a)} {W_{q+r}(-a)} (W^{a}_{\q,\p})'(x)+ \frac r {q+r} \Big[ (q+r)W^{a}_{\q,\p}(x)-rW_q(x)Z_{q+r}(-a) - qW_q(x)  \Big].
\end{align*}

Similarly to Corollary \ref{corollary_big_U}, we have the following result.
\begin{lemma} \label{A3}
For $a < 0 < b$ and $x \leq b$,
\begin{align*}
\widehat{U}_1(x, a,b,\theta) &=   \frac{W_q(x)}{W_q'(b+)}\widehat{U}_1^0(a,b,\theta) - U_1^0(a,x,\theta), \quad \theta \geq 0, \\
\widehat{U}_i(x, a,b) &=  \frac{W_q(x)}{W_q'(b+)}\widehat{U}_i^0(a,b) - U_i^0(a,x), \quad i = 2,3, 4.
\end{align*}
%\green{[we don't need this any more?]where we assume for $\widehat{U}_4$, $\kappa'(0+) > -\infty$.}
\end{lemma}
\begin{proof}
For the case $i=1,2,3$, this is a direct consequence of Lemma \ref{lemma_u_in_V} and Theorem \ref{lemma_simplifying_reflected}.%\eqref{LRZ_identity_reflected}.

On the other hand, for the remaining case, we note that
%\green{[JL, Here we can have this argument for $\widehat{U}_4$.]Here
\begin{align*}
\widehat{U}_4(x, a,b) =  \frac{W_q(x)}{W_q'(b+)} \lim_{\theta \downarrow 0} \frac \partial {\partial \theta}\widehat{U}_1^0(a,b,\theta) - \lim_{\theta \downarrow 0} \frac \partial {\partial \theta} U_1^0(a,x,\theta).
\end{align*}
Similarly to the computation for $\lim_{\theta \downarrow 0} (\partial U_1^0(a,x,\theta)/  {\partial \theta})$ as in the proof of Corollary \ref{corollary_big_U}, we have \\ $\lim_{\theta \downarrow 0} (\partial \widehat{U}_1^0(a,b,\theta) / {\partial \theta})  =- h^{a \prime}_{q,r}(x)$. Hence putting the pieces together we obtain the result.
\end{proof}

Note that by Lemma \ref{A3}, \eqref{U_0_summary},  \eqref{U_1_U_2_sum}, \eqref{U_0_hat}, and $\widehat{U}_1^0(a,x, \theta) + \widehat{U}_2^0(a,x) = -\mathcal{H}_{q,r}^{a \prime}(x, \theta)$ for $x \in \R$ and $\theta \geq 0$,
%and hence together with
		\begin{align} \label{H_H_derivative_relation}
		\begin{split}
		 \mathcal{H}^a_{q,r}(b,\theta) - \widehat{U}_1(b,a,b,\theta) - \widehat{U}_2(b,a,b)  &= \frac {W_q(b)} {W_q'(b+)} \mathcal{H}^{a \prime}_{q,r}(b,\theta), \quad \theta \geq 0, \\
		 \mathcal{I}^a_{q,r}(b) - \widehat{U}_3(b,a,b)  &= \frac {W_q(b)} {W_q'(b+)} \mathcal{I}^{a \prime}_{q,r}(b), \\
		  h_{\q,\p}^{a}(b)  - \widehat{U}_4(b,a,b) &=\frac {W_q(b)} {W_q'(b+)}  h_{\q,\p}^{a \prime}(b).
			\end{split}
		\end{align}

%\green{[remove below?]}
%In particular,
%\begin{align*}
%\widehat{U}_1(x, a, b, 0)&+ \widehat{U}_2(x, a, b)
%=\frac {r \overline{W}_{q+r} (-a)+1} {W_{q+r}(-a)} \Big(W^{a}_{q,r} (x)  - \frac {W_q(x)} {W_q'(b+)}\blue{(W^{a}_{\q,\p})'}(b) \Big) \\
%&+ r \frac {W_q(x)} {W_q'(b+)} W^{a}_{\q,\p}(b) + \frac r {r+q} \Big( \frac {W_q(x)} {W_q'(b+)}\left[-rW_q(b)Z_{q+r}(-a) - qW_q(b)\right] - Z^{a}_{q,r} (x) + Z_q (x)  \Big).
%\end{align*}

\subsection{Proof of Theorem \ref{theorem_h_hat}}
%\red{[JL: I think we should do similarly to the original proof because then we do not need to work on the unbounded variation case separately.  Could you see what I added for the proof of Lemma 3.2?]}\blue{Hi Kazu: But doesn't this proof work for the unbounded variation case also? The difficult part has already been computed by the functions $h$ and $g$?. } \red{I see. Sorry, I did not read this clearly, I think it is ok.}

%Now let us denote by
%\begin{align*}
%\widehat{h}(x,a,\th)=\E_x\left[e^{-q\tau_a^-(r) -\theta L^b_r(\tau_a^-)}\right].
%\end{align*}
Note that for $x \leq b$, $\mathbb{P}_x$-a.s.
\begin{align} \label{Y_X_matches}
Y^b_r(t) = X_r(t), \quad 0 \leq t \leq \tau_b^+(r).
\end{align}
By an application of the Markov property at $\tau_b^+(r)$ on $\{ \tau_b^+(r) < \tau_a^-(r)\}$, we have
\begin{align}
\widehat{h}(x,a,b,\th)=h(x,a,b,\th)+g(x,a,b,\th)\widehat{h}(b,a, b, \th), \quad x \leq b, \label{h_hat_recursion}
\end{align}
where  $g$ and $h$ are as defined in \eqref{Parisbailouts} and \eqref{Parisbailouts_2}.
Again by the strong Markov property and \eqref{Y_matches},
\begin{align*}
\widehat{h}(b,a,b,\th)&=\E_b\left[e^{-q\eta_0^-}\E_{Y^b(\eta_0^{-})}\left(e^{-q\tau_a^-};\tau_a^-<\mathbf{e}_{\p}\wedge\tau_0^+\right)\right]\\
&+\E_b\left[e^{-q\eta_0^-}\E_{Y^b(\eta_0^{-})}\left(e^{-q\tau_0^+};\tau_0^+<\mathbf{e}_{\p}\wedge\tau_a^-\right)\right]\widehat{h}(0,a,b, \theta)\\
&+\E_b\left[ e^{-q\eta_0^-}\E_{Y^b(\eta_0^{-})}\left(e^{-q\mathbf{e}_{\p}+\theta X(\mathbf{e}_{\p})};\mathbf{e}_{\p}<\tau_0^+\wedge\tau_a^-\right)\right] \widehat{h}(0,a, b, \theta)\\
&=\widehat{U}_3 (b, a,b)+ [\widehat{U}_1(b, a,b, \theta) + \widehat{U}_2 (b, a,b)] \widehat{h}(0,a, b, \theta).
\end{align*}
Substituting this in \eqref{h_hat_recursion},
%Therefore we obtain
\begin{align}\label{h_0_reflected_1}
\widehat{h}(x,a,b, \th)=h(x,a,b,\th)+g(x,a,b,\th)\left(\widehat{U}_3 (b, a,b)+ [\widehat{U}_1(b, a,b, \theta) + \widehat{U}_2 (b, a,b)] \widehat{h}(0,a, b, \theta)\right).
\end{align}
Setting $x=0$ and solving for $\widehat{h}(0,a,b, \th)$, we obtain by \eqref{g_zero}, \eqref{h_zero}, and \eqref{H_H_derivative_relation},
\begin{align}\label{h_0_reflected_2}
\widehat{h}(0,a, b, \theta)=\frac{\widehat{U}_3 (b, a,b)-\mathcal{I}^a_{q,r}(b)}{\mathcal{H}^a_{q,r}(b,\th)-[\widehat{U}_1 (b, a,b,\th)+\widehat{U}_2 (b, a,b)]}=-\frac{\mathcal{I}_{\q,\p}^{a \prime}(b)}{\mathcal{H}^{a \prime}_{q,r}(b,\theta)}.
\end{align}
Finally, substituting \eqref{h_0_reflected_2} in \eqref{h_0_reflected_1} and using  \eqref{g_zero}, \eqref{h_zero}, and \eqref{U_0_hat} allows us to obtain the result.
%\begin{align*}
%\widehat{h}(x,a, b,\th)=\blue{\mathcal{I}^a_{q,r}(x)}-\frac{\mathcal{H}^a_{q,r}(x,\th)}{\blue{\frac \partial {\partial x}\mathcal{H}^a_{q,r}(b,\theta)}}\blue{\frac \partial {\partial x}\mathcal{I}_{\q,\p}^a(b)}.
%\end{align*}
	\subsection{Proof of Theorem \ref{theorem_R_b_r}}
	%\begin{proof}(Proof of Lemma 3.1).
We shall prove for the case $x \leq b$.  The case $x > b$ is then immediate by reflection.
		By \eqref{Y_X_matches}, the strong Markov property at $\tau_b^+(r)$ on $\{ \tau_b^+(r) < \tau_a^-(r)\}$, and \eqref{Parisbailouts}, we have
	%	\begin{align}
	%	\widehat{j}(x,b) &=\frac{Z_{\q,\p}(x)}{Z_{\q,\p}(b)} \widehat{j}(b,b), \quad x \leq b. \label{j_relation}
	%	\end{align}
		%		\red{Sorry this was a mistake -- it should be
				\begin{align}
		\widehat{j}(x,a, b) &=g(x,a,b,0)\widehat{j}(b,a,b)=\frac{\mathcal{H}^a_{q,r}(x,0)}{\mathcal{H}^a_{q,r}(b,0)}\widehat{j}(b,a,b), \quad x \leq b. \label{j_relation}
		\end{align}
	%	}
By the strong Markov property and \eqref{Y_matches}, together with (3.12) of \cite{APP},
		\begin{align}
		\widehat{j}(b,a,b) &= \E_b \Big( \int_0^{\eta_0^-} e^{-qt} \diff L^b (t) \Big) +\E_b\Big[e^{-q\eta_0^-}\E_{Y^b(\eta_0^-)}\left(e^{-q(\tau_0^+\wedge \mathbf{e}_{\p})}; \tau_a^- > \tau_0^+\wedge \mathbf{e}_{\p} \right)\Big] \widehat{j}(0,a,b)\notag \\
		&= \frac {W_q(b)} {W_q'(b+)} +\big[ \widehat{U}_1(b,a,b,0) + \widehat{U}_2(b,a,b) \big] \widehat{j}(0,a,b).
		 \label{j_b_b}
		\end{align}
		Substituting  \eqref{j_b_b} in \eqref{j_relation}, and setting $x=0$,
%		\begin{align*}
%		\widehat{j}(0,b) &=\frac{1}{Z_{\q,\p}(b)} \Big( \frac {W_q(b)} {W_q'(b+)}+\Big( \widehat{U}_1(b,a,0) + \widehat{U}_2(b,a)\Big) \widehat{j}(0,b) \Big).
%		\end{align*}
	%%	\red{it should be
				\begin{align*}
		\widehat{j}(0,a,b) &=\frac{1}{\mathcal{H}^a_{q,r}(b,0)} \Big( \frac {W_q(b)} {W_q'(b+)}+\big[ \widehat{U}_1(b,a,b,0) + \widehat{U}_2(b,a,b)\big] \widehat{j}(0,a,b) \Big),
		\end{align*}
		and hence, together with \eqref{H_H_derivative_relation},
						\begin{align*}
		\widehat{j}(0,a,b)  &=  \frac {W_q(b)} {W_q'(b+)} \Big( \mathcal{H}^a_{q,r}(b,0) - \widehat{U}_1(b,a,b,0) - \widehat{U}_2(b,a,b) \Big)^{-1} =  \mathcal{H}^{a \prime}_{q,r}(b,0)^{-1}.
		\end{align*}
Substituting this in \eqref{j_b_b}  allows us to obtain that $\widehat{j}(b,a,b) =  \mathcal{H}^a_{q,r}(b,0) / \mathcal{H}^{a \prime}_{q,r}(b,0)$. This together with  \eqref{j_relation} completes the proof.

\subsection{Proof of Theorem \ref{theorem_L_b_r}}
%\green{[remove? the argument should hold for $q = 0$ as well.]We shall show for the case $q > 0$; the case $q=0$ holds by monotone convergence.}
We focus on  the case $x \leq b$.  The case $x > b$ is then immediate by reflection.

By \eqref{Y_X_matches}, the strong Markov property at $\tau_b^+(r)$ on $\{ \tau_b^+(r) < \tau_a^-(r)\}$, and  Theorems \ref{theorem_laplace} and \ref{expected_bailouts},		\begin{align}%\label{mp_b}
		\widehat{f}(x,a, b) &= -h_{\q,\p}^{a}(x) + \frac{\mathcal{H}^a_{q,r}(x,0)}{\mathcal{H}^a_{q,r}(b,0)} [\widehat{f}(b,a, b) + h_{\q,\p}^{a}(b)]. \label{f_hat_relation}
		\end{align}
By the strong Markov property and \eqref{Y_matches},
\begin{align} \label{f_hat_relation_b}
\begin{split}
		\widehat{f}(b,a,b) &= -\E_b\Big[ e^{-q\eta_0^-}\E_{Y^b(\eta_0^-)}\left(e^{-q \mathbf{e}_{\p} }X(\mathbf{e}_{\p});\mathbf{e}_{\p} <
		\tau_0^+ \wedge \tau_a^- \right)\Big] \\ &+\E_b\Big[e^{-q\eta_0^-}\E_{Y^b(\eta_0^-)}\left(e^{-q(\tau_0^+\wedge \mathbf{e}_{\p})}; \tau_a^- > \tau_0^+\wedge \mathbf{e}_{\p} \right)\Big] \widehat{f}(0,a,b)\\
		&= - \widehat{U}_4(b,a,b) +\big[ \widehat{U}_1(b,a,b,0) + \widehat{U}_2(b,a,b)\big] \widehat{f}(0,a,b).
		\end{split}
		\end{align}
Substituting \eqref{f_hat_relation_b} in \eqref{f_hat_relation} and setting $x = 0$ and noticing that $h_{\q,\p}^{a}(0) = 0$,
\begin{align*}%\label{mp_b}
		\widehat{f}(0,a, b) &=  \frac 1 {\mathcal{H}^a_{q,r}(b,0)} \Big[ - \widehat{U}_4(b,a,b) +\big[ \widehat{U}_1(b,a,b,0) + \widehat{U}_2(b,a,b)\big] \widehat{f}(0,a,b) +  h_{\q,\p}^{a}(b) \Big].
		\end{align*}
		Hence, by \eqref{H_H_derivative_relation}, %\green{[simplified the signs below.]}
		\begin{align*}%\label{mp_b}
		\widehat{f}(0,a, b)  &=  \frac {h_{\q,\p}^{a}(b)  - \widehat{U}_4(b,a,b)} {\mathcal{H}^a_{q,r}(b,0)  -  \widehat{U}_1(b,a,b,0) - \widehat{U}_2(b,a,b)} =  \frac {h_{\q,\p}^{a \prime}(b)} {\mathcal{H}^{a \prime}_{q,r}(b,0)}.
		\end{align*}
		Substituting this in \eqref{f_hat_relation_b} and then in \eqref{f_hat_relation} and using \eqref{U_0_hat}, we have the claim.

\appendix

\section{Proofs} \label{appendix_proofs}
%First we will provide some auxiliary results: %\blue{Kazu: I reincluded this lemmas because they are used, maybe we can clean them up?}
%%\red{Hi JL, where are these used? Do we need all?}\blue{I think only d) is used in Remark 3.1, I changed it.}
%	\begin{lemma}
%		When $b \to \I$, it holds that
%		\begin{align}\label{limits}
%		\begin{split}
%		&\lim_{b\to\infty}\frac{Z_{\q}(b)}{W_{\q}(b)}
%		=\frac{q}{\Phi_q}, \quad \lim_{b\to\infty}\frac{\overline{Z}_{\q}(b)}{W_q(b)}
%		=\frac{q}{\Phi_q^2}\\&
%		\lim_{b\to\infty}\frac{Z_{\q}(b,\Phi_{\q+\p})}
%		{W_{\q}(b)}=\frac{\p}{\Phi_{\q+\p}-\Phi_{\q}}.
%		\end{split}
%		\end{align}
%	\end{lemma}

\subsection{Proofs of Corollaries in Section \ref{subsection_joint_Laplace}} \label{appendix_proofs_regular}

We shall first summarize the limits needed to show the corollaries.

%\green{[JL: We need a special care fo the case $q =0$ and so here we assume $q > 0$.]}
\begin{lemma} \label{lemma_limits_b}
Fix $q > 0$ and $a < 0$. We have, as $x\uparrow\infty$, %\green{[swapped (i) and (ii)]}
(i) $ \mathcal{I}_{q,r}^a(x) / {W_q(x)} \rightarrow \mathcal{J}_{q,r}(-a)$, (ii) ${\mathcal{H}_{q,r}^a(x,\th)} /{W_q(x)} \rightarrow \mathcal{G}_{q,r}(-a,\th)$ for  $\theta \geq 0$, \quad (iii) $h_{q,r}^a(x) / {W_q(x)} \rightarrow h_{q,r}(-a)$.
\end{lemma}
\begin{proof}
%\red{[shortened a little]}
%	Using  (5) in \cite{LRZ} and (3.4) in \cite{YP}  we can write
%\begin{align*}
%	W_{q,r}^a(x)=W_q(x-a)+r\int_0^{-a}W_q(x-u-a)W_{q+r}(u) \diff u, \\
%Z_{q,r}^a(x)=Z_q(x-a)+r\int_0^{-a}W_q(x-u-a)Z_{q+r}(u) \diff u, \\
%\overline{Z}_{q,r}^a(x)=\overline{Z}_q(x-a)+r\int_0^{-a}W_q(x-u-a)\overline{Z}_{q+r}(u)\diff u.
%	\end{align*}
Recall  the second limit in \eqref{nl}. 
Also, by Exercise 8.5 of \cite{K},
		\begin{align}\label{limits}
		\begin{split}
		&\lim_{b\uparrow \infty}\frac{Z_{\q}(b)}{W_{\q}(b)}
		=\frac{q}{\Phi_q}, \quad \lim_{b\uparrow \infty}\frac{\overline{Z}_{\q}(b)}{W_q(b)}
		=\frac{q}{\Phi_q^2},
%		\quad \textrm{and} \quad
%		\lim_{b\uparrow \infty}\frac{Z_{\q}(b,\Phi_{\q+\p})}
%		{W_{\q}(b)}=\frac{\p}{\Phi_{\q+\p}-\Phi_{\q}}.
		\end{split}
		\end{align}
By this,  \eqref{W_q_limit} and the second equality of \eqref{W_a_def},
%	In a similar way using again identity (5) in \cite{LRZ} we obtain
%
%	Hence,
\begin{align} \label{Z_a_W_limit}
\begin{split}
		\lim_{x \uparrow \infty}\frac{Z_{q,r}^a(x)}{W_q(x)}&=\frac{q}{\Phi_q}e^{-\Phi_q a}+r\int_0^{-a}e^{-\Phi_q(u+a)}Z_{q+r}(u) \diff u \\
		%&=e^{-\Phi_q a}\frac{1}{\Phi_q}\left(q+r(1-e^{\Phi_q a}Z_{q+r}(-a))+(r+q)\left(1+r\int_0^{-a}e^{-\Phi_q u}W_{q+r}(u)\diff u \right)-(r+q)\right)\\
		&=\frac{1}{\Phi_q}\left[(q+r)Z_{q+r}(-a,\Phi_q)-rZ_{q+r}(-a)\right]=\frac{q}{\Phi_q}Z_{q+r,-r}(-a),
		\end{split}
		\end{align} %\red{[above: remove the second line and say we used integration by parts? Actually the way I do integration by parts seems different from yours and the order of terms appear to be different.]}
		where in the second equality we used integration by parts. In a similar way, using, \eqref{limits}, \eqref{Z_a_W_limit} and integration by parts allows us to obtain
	\begin{align*}
	\lim_{x \uparrow \infty}\frac{\overline{Z}_{q,r}^a(x)}{W_q(x)}&=\frac{q}{\Phi_q^2}e^{-\Phi_q a}+r\int_0^{-a}e^{-\Phi_q(u+a)}\overline{Z}_{q+r}(u)\diff u\\
	%&=\frac{q}{\Phi_q^2}e^{-\Phi_q a}-\frac{r}{\Phi_q}\overline{Z}_{q+r}(-a)+\frac{r}{\Phi_q}\int_0^{-a}e^{-\Phi_q(u+a)}Z_{q+r}(u)\diff u \\
	&=\frac{q}{\Phi_q^2}e^{-\Phi_q a}-\frac{r}{\Phi_q}\overline{Z}_{q+r}(-a)+\frac{1}{\Phi_q}\left(\frac{q}{\Phi_q}e^{-\Phi_q a}+r\int_0^{-a}e^{-\Phi_q(u+a)}Z_{q+r}(u)\diff u \right)-\frac{qe^{-\Phi_q a}}{\Phi_q^2}\\
	%&=\frac{1}{\Phi_q^2}\left((r+q)Z_{q+r}(-a,\Phi_q)-rZ_{q+r}(-a)\right)-\frac{r}{\Phi_q}\overline{Z}_{q+r}(-a)\\
	&=\frac{q}{\Phi^2_q}Z_{q+r,-r}(-a)-\frac{r}{\Phi_q}\overline{Z}_{q+r}(-a).
	\end{align*}
%	\red{[remove second and fourth lines? And say we use integration by parts and use \eqref{Z_a_W_limit}?]}
Using these limits, (i), (ii), and (iii) are immediate.
\end{proof}
		
	%	In order to prove these Corollaries we first state the following lemma:

%\green{[JL: Below, the results should hold for $q = 0$.]}
\begin{lemma} \label{convergence_summary}
Fix $q \geq 0$ and $x \in \R$.  We have, as $a \downarrow -\infty$, %\red{[combined (i) and (i')]}

(i) $ \mathcal{I}_{\q,\p}^a (x) \rightarrow 0$, \quad
(ii) $\mathcal{H}^a_{q,r}(x,\theta) \rightarrow Z_{\q,\p}(x,\th)$ for $\theta \geq 0$,
%$\theta \in \{0\} \cup (\Phi_{q+r}, \infty)$,
\quad
(iii) if $\kappa'(0+) > -\infty$, $h_{\q,\p}^{a}(x) \rightarrow k_{\q,\p}(x)$.

%where
%\begin{align*}
%k_{\q,\p}(x):=\frac{r}{q+r}\left(\overline{Z}_q(x)-\kappa'(0+)\overline{W}_{\q}(x)- \frac{\kappa'(0+)}{\q+\p} \big(Z_{\q}(x,\Phi_{\q+\p})-Z_{\q}(x) \big)\right) %= \tilde{h}_{q,r}(x)   -\frac {r\kappa'(0+)} {q(q+r)} Z_{q,r}(x).
%	\end{align*}
%\green{[let's delete the last equality, as it may not make sense when $q = 0$.]}
			\end{lemma}

\begin{proof}
(i)	By  \eqref{laplace_in_terms_of_z}, \eqref{W_a_def}, and dominated convergence,
\begin{align*}
\lim_{a \downarrow -\infty} \mathcal{I}_{\q,\p}^a (x)
&=\lim_{a \downarrow -\infty} \Big[ \E_x \big(e^{-(\q+\p)\tau_{a}^-};\tau_a^-<\tau_{0}^+ \big)-r\int_0^xW_{\q}(x-z)\E_z(e^{-(\q+\p)\tau_{a}^-};\tau_a^-<\tau_{0}^+) \diff z \Big] =0.
\end{align*}

	(ii)  First suppose $\theta > \Phi_{q+r}$.
	By \eqref{scale_function_laplace} and \eqref{nl},
	\begin{align*}
	\lim_{a \downarrow -\infty }\frac{W^a_{\q,\p}(x)}{W^a_{\q+\p}(-a)}\int_{0}^{-a} e^{-\theta u}W_{\q+\p} (u)\diff u&=Z_{\q}(x,\Phi_{\q+\p})\int_0^{\infty} e^{-\theta u}W_{\q+\p} (u)\diff u=\frac{Z_{\q}(x,\Phi_{\q+\p})}{\k(\th)-(\q+\p)}.
	\end{align*}
	%\blue{Kazu: I rewrote the proof, can you check?
	In addition, we have that
	\begin{align*}
	\lim_{a\downarrow -\infty}\int_{0}^{-a}e^{-\th u} W^{-u}_{\q,\p}(x)\diff u&=\int_0^{\infty}e^{-\th u}\left(W_{\q+\p}(x+u)-r\int_{0}^{\infty}W_{\q}(x-z)W_{\q+\p}(z+u)\diff z\right)\diff u.
	\end{align*}
	%Now it is not difficult to check that
	Here, by \eqref{scale_function_laplace},
	\[
	\int_0^{\infty}e^{-\th u}W_{\q+\p}(x+u)\diff u=e^{\th x}\left(\frac{1}{\k(\th)-(\q+\p)}-\int_0^xe^{-\th u}W_{\q+\p}(u)\diff u\right),
	\]
	and %\green{[below, remove the middle?]}
	\begin{align*}
\int_{0}^{\infty}W_{\q}(x-z)W_{\q+\p}(z+u)\diff z%&=\int_{u}^{\infty}W_{\q}(x+u-y)W_{\q+\p}(y)\diff y\\
&=\int_{0}^{\infty}W_{\q}(x+u-y)W_{\q+\p}(y)\diff y-\int_{0}^{u}W_{\q}(x+u-y)W_{\q+\p}(y)\diff y.
	\end{align*}	
	Now by the identity (5) in \cite{LRZ}, we have that
	\[
	\int_{0}^{\infty}W_{\q}(x+u-y)W_{\q+\p}(y)\diff y=\frac{1}{r}\left[W_{\q+\p}(x+u)-W_{\q}(x+u)\right].
	\]
	%For the second integral we note that
	In addition,
	\begin{align*}
	\int_0^{\infty}e^{-\th u}&\int_{0}^{u}W_{\q}(x+u-y)W_{\q+\p}(y)\diff y\diff u=\int_0^{\infty}W_{\q+\p}(y)\int_{y}^{\infty}e^{-\th u}W_{\q}(x+u-y)\diff u\diff y\\
	&=\int_0^{\infty} e^{-\th y} W_{\q+\p}(y)\diff y\int_0^{\infty}e^{-\th z}W_{\q}(x+z)\diff z = [\kappa(\theta)-(q+r)]^{-1}\int_0^{\infty}e^{-\th z}W_{\q}(x+z)\diff z.
	\end{align*}
	Therefore, we get
	\begin{align*}
	r\int_0^{\infty}e^{-\th u}\int_{0}^{\infty}W_{\q}(x-z)W_{\q+\p}(z+u)\diff z \diff u&=-\frac{\kappa(\th)-q}{\k(\th)-(\q+\p)}e^{\th x}\left(\frac{1}{\k(\th)-q}-\int_0^xe^{-\th z}W_{\q}(z)\diff z\right)\\
	&+e^{\th x}\left(\frac{1}{\k(\th)-(\q+\p)}-\int_0^xe^{-\th u}W_{\q+\p}(u)\diff u\right).
	\end{align*}%}
	Hence $\int_{0}^{-a}e^{-\th u}W^{-u}_{\q,\p}(x)\diff u \xrightarrow{a \downarrow -\infty}Z_{\q}(x,\th) / [\k(\th)-(\q+\p)]$.
	Putting the pieces together, we obtain that
	\begin{align}\label{H_limit_-a}
	\lim_{a \downarrow -\infty}\mathcal{H}^a_{q,r}(x,\theta)&=\p\frac{Z_{\q}(x,\Phi_{\q+\p})}{\k(\th)-(\q+\p)}-\frac{\p}{\k(\th)-(\q+\p)} Z_{\q}(x,\th)+Z_{\q}(x,\Phi_{\q+\p})=Z_{\q,\p}(x,\th).
	\end{align}%}
	
We shall now extend this to $\theta \geq 0$. By Corollary \ref{corollary_big_U_zero},  \eqref{nl}, \eqref{U_0_summary}, \eqref{U_1_U_2_sum}, and monotone convergence, for $\theta > \Phi_{q+r}$,
\begin{multline} \label{U_1_excursion_convergence}
- W_q(x) \mathbf{n} \Big(e^{-q\tau_0^-} \E_{X(\tau_0^-)}
\left(e^{-q\mathbf{e}_{\p}+\theta X(\mathbf{e}_{\p})};
\mathbf{e}_{\p}<\tau_0^+ \right);\tau_0^-<\t_x^+\Big) = - \lim_{a \downarrow -\infty}U_1^0(a,x, \theta)  \\
= \lim_{a \downarrow -\infty} \Big( \mathcal{H}^a_{q,r}(x,\theta) - \frac {W^a_{q,r} (x)} {W_{q+r}(-a)} \Big) = Z_{\q,\p}(x,\th) - Z_q (x, \Phi_{q+r}), %= W_q(b) \mathbf{n} \Big(e^{-q\tau_0^-} u_1 (X(\tau_0^-), -\infty, \theta);\tau_0^-<\t_b^+\Big),
\end{multline}
which can be analytically extended to $\theta \geq 0$  as in the proof of Theorem 3.1 in \cite{AIZ}.
This shows (ii) for $\theta \geq 0$.% [and remove the rest]

(iii) By \eqref{U_1_excursion_convergence}, %[shortend below]}
\begin{multline*}
-W_q( x) \mathbf{n} \Big(e^{-q\tau_0^-} \E_{X(\tau_0^-)}
\left(e^{-q\mathbf{e}_{\p}}X(\mathbf{e}_{\p});
\mathbf{e}_{\p}<\tau_0^+ \right);\tau_0^-<\t_x^+\Big) = \lim_{\theta \downarrow 0}\frac \partial {\partial \theta}Z_{\q,\p}(x,\th) \\ = \fr{r}{q+r} \lim_{\theta \downarrow 0}\frac \partial {\partial \theta}Z_{q}(x,\th) + \fr{r \kappa'(0+)}{(q+r)^2} [Z_{q}(x)- Z_{q}(x,\Phi_{q+r} ) ],
\end{multline*}
where simple algebra gives $\lim_{\theta \downarrow 0} (\partial Z_{q}(x,\th) / {\partial \theta} )= - \k'(0+)  \overline{W}_{q}(x) + \overline{Z}_q(x)$.
%Note that
%\begin{align*}
%\frac \partial {\partial \theta}Z_{\q,\p}(x,\th) = \fr{r}{q+r-\k(\th)} \frac \partial {\partial \theta}Z_{q}(x,\th) + \fr{r \kappa'(\theta)}{(q+r-\k(\th))^2} Z_{q}(x,\th)
%\\ +\fr{-\kappa'(\theta) (q+r-\kappa(\theta))+(q-\k(\th)) \kappa'(\theta)}{(q+r-\k(\th))^2} Z_{q}(x,\Phi_{q+r} )  \\
%\xrightarrow{\theta \downarrow 0}\fr{r}{q+r} \lim_{\theta \downarrow 0}\frac \partial {\partial \theta}Z_{q}(x,\th) + \fr{r \kappa'(0+)}{(q+r)^2} (Z_{q}(x)- Z_{q}(x,\Phi_{q+r} ) )
%\end{align*}
%Here,
%\begin{align*}
%\frac \partial {\partial \theta}Z_{q}(x,\th)  &= x e^{\theta x} \left( 1 + (q- \k(\theta)) \int_0^{x} e^{-\theta z} W_{q}(z) \diff z	\right) \\
%&+ e^{\theta x} \left(  - \k'(\theta) \int_0^{x} e^{-\theta z} W_{q}(z) \diff z	- (q- \k(\theta)) \int_0^{x} z e^{-\theta z} W_{q}(z) \diff z\right) \\
%&\xrightarrow{\theta \downarrow 0} x  Z_{q}(x) + \left(  - \k'(0+)  \overline{W}_{q}(x)	- q \int_0^{x} z  W_{q}(z) \diff z\right)
%=  - \k'(0+)  \overline{W}_{q}(x) + \overline{Z}_q(x).
%\end{align*}
Hence,
\begin{align*}
\lim_{\theta \downarrow 0}\frac \partial {\partial \theta}Z_{\q,\p}(x,\th) = \fr{r}{q+r} \Big(  - \k'(0+)  \overline{W}_{q}(x) + \overline{Z}_q(x) \Big) + \fr{r \kappa'(0+)}{(q+r)^2} [Z_{q}(x)-Z_{q}(x,\Phi_{q+r} ) ] = k_{q,r} (x).
\end{align*}
Now, by monotone convergence and Corollary \ref{corollary_big_U_zero}, we obtain
\begin{multline*}
h^a_{q,r}(x) = -W_q(x) \mathbf{n} \Big(e^{-q\tau_0^-} u_4 (X(\tau_0^-), a);\tau_0^-<\t_x^+\Big) \\ \xrightarrow{a \downarrow -\infty} -W_q(x) \mathbf{n} \Big(e^{-q\tau_0^-} \E_{X(\tau_0^-)}
\left(e^{-q\mathbf{e}_{\p}}X(\mathbf{e}_{\p});
\mathbf{e}_{\p}<\tau_0^+\right);\tau_0^-<\t_x^+\Big) = k_{q,r}(x),
\end{multline*}
as desired.
\end{proof}

\subsubsection{Proof of Corollary \ref{corollary_laplace_tau_a_b_infty} } %\red{[JL: Combined and simplified this way.]}

(i) For $q > 0$, by applying dominated convergence in \eqref{Parisbailouts_2}, upon taking $b \uparrow \infty$, it is immediate by Lemma \ref{lemma_limits_b}.

For $q = 0$, by monotone convergence, it suffices to show the convergence as $q \downarrow 0$ of
%\begin{align*}
%\frac q {\Phi_q} Z_{q+r, -r} (y) = \Phi_q^{-1} [-r Z_q (y) + (q+r) Z_q(y, \Phi_q)] = \frac  {-r Z_q (y) + (q+r) e^{\Phi_q y}} {\Phi_q}.
%\end{align*}
%For the case $\Phi_{0} > 0$, it is immediate that the above converges to $\frac r {\Phi_0} [e^{\Phi_0 y} -1]$.
%For the case $\Phi_0 = 0$,
%\begin{align*}
%\frac q {\Phi_q} Z_{q+r, -r} (y) = r \frac  {- Z_q (y) + e^{\Phi_q y}} {\Phi_q} +  \frac  {q e^{\Phi_q y}} {\Phi_q}.
%\end{align*}
%Here, because $q/\Phi_q \xrightarrow{q \downarrow 0} \kappa'(0+)$,
%\begin{align*}
%\frac  {- Z_q (y) + e^{\Phi_q y}} {\Phi_q}  = \frac  {- q \overline{W}_q (y) + e^{\Phi_q y}-1} {\Phi_q} \xrightarrow{q \downarrow 0} y - \kappa'(0+) \overline{W}_0(y),
%\end{align*}
%and  ${q e^{\Phi_q y}} / {\Phi_q} \xrightarrow{q \downarrow 0} \kappa'(0+)$.
%Therefore,
%\begin{align*}
%\lim_{q \downarrow 0}\frac q {\Phi_q} Z_{q+r, -r} (y) = r (y - \kappa'(0+) \overline{W}_0(y)) + \kappa'(0+).
%\end{align*}}
%\blue{Kazu: I get this
		\begin{align*}
		\frac q {\Phi_q} Z_{q+r, -r} (y) = \Phi_q^{-1} [ (q+r) Z_{q+r}(y, \Phi_q)- r Z_{q+r} (y)], \quad y \in \R.
		\end{align*}
		%\green{[thanks JL. I simplified a bit here.]}
		For the case $\Phi_{0} > 0$, it is immediate that the above converges to $r  [Z_r(y,\Phi_0)-Z_r(y)] / {\Phi_0}$.
		For the case $\Phi_0 = 0$, because ${q e^{\Phi_q y}} / {\Phi_q} \xrightarrow{q \downarrow 0} \kappa'(0+)$,
		%\begin{align*}
		%\frac q {\Phi_q} Z_{q+r, -r} (y) = r \frac  {- Z_q (y) + e^{\Phi_q y}} {\Phi_q} +  \frac  {q e^{\Phi_q y}} {\Phi_q}.
		%\end{align*}
		\begin{align*}
		&\Phi_q^{-1} [(q+r) Z_{q+r}(y, \Phi_q)-r Z_{q+r} (y) ]=r\frac{e^{\Phi_q y}-1}{\Phi_q}+r(r+q)\int_0^y\frac{e^{\Phi_q(y-z)}-1}{\Phi_q}W_{q+r}(z)\diff z+\frac{qe^{\Phi_q y}}{\Phi_q} \\
%		\end{align*}
%		Here,
%		\begin{align*}
%		r\frac{(e^{\Phi_q y}-1)}{\Phi_q}+r(r+q)\int_0^y\frac{(e^{\Phi_q(y-z)}-1)}{\Phi_q}W_{q+r}(z)dz+\frac{qe^{\Phi_q y}}{\Phi_q}
		 &\xrightarrow{q \downarrow 0} ry+r^2\int_0^y(y-z)W_{r}(z)\diff z +\k'(0+) =r \overline{Z}_r(y)+\k'(0+).
		\end{align*}
		%and  ${q e^{\Phi_q y}} / {\Phi_q} \xrightarrow{q \downarrow 0} \kappa'(0+)$.
		Therefore,
		\begin{align*}
		\lim_{q \downarrow 0}\frac q {\Phi_q} Z_{q+r, -r} (y) =r \overline{Z}_r(y)+{\k'(0+)}.
		\end{align*}
%\subsection{Proof of Corollary  \ref{corollary_laplace_tau_b_a_infty}}

(ii) By applying dominated convergence in \eqref{Parisbailouts}, upon taking $a \downarrow -\infty$, it is immediate by Lemma \ref{convergence_summary} (ii) for $\theta \geq 0$. %\green{$\theta \geq 0$?}%\green{[Now, Lemma \ref{convergence_summary} \green{(ii)} holds for $\theta \geq 0$ and so we omit the following sentence?]} The case $\theta \leq \Phi_{q+r}$ can be shown by analytic continuation as in the proof of Theorem 3.1 in \cite{AIZ}.

%We shall show for the case $\theta > \Phi_{q+r}$.  The case $\theta \leq \Phi_{q+r}$ can be shown by analytic continuation as in the proof of Theorem 3.1 in \cite{AIZ}.
%By Theorem \ref{theorem_laplace} and  monotone convergence, it is sufficient to show the convergence:
%$\mathcal{H}^a_{q,r}(x,\theta) \xrightarrow{a \downarrow -\infty} Z_{\q,\p}(x,\th)$ for $x \in \R$.

	%	Therefore
%	\begin{align*}
%	\E_x\left[e^{-q\tau_b^+(r)- \theta L^{(r)}(\tau_b^+)}\right]=\lim_{\red{a\to-\infty}}\E_x\left[e^{-q\tau_b^+- \theta L^{(r)}(\tau_b^+)}; \t_b^+< \tau_a^-\right]=\lim_{\red{a\to-\infty}}\frac{\mathcal{H}^a_{q,r}(x,\theta)}{\mathcal{H}^a_{q,r}(b,\theta)}=\frac{Z_{\q,\p}(x,\th)}{Z_{\q,\p}(b,\th)},
%	\end{align*}
%	and the result follows.
%\end{proof}
%\subsection{Proof of Corollary  \ref{corollary_laplace_tau_b_a_infty_the0} }
%\red{[JL: I changed the order of arguments etc.]}
(iii) By taking $\theta \uparrow \infty$ in (ii), %Corollary \ref{corollary_laplace_tau_b_a_infty} ,	
\begin{multline*}
		\E_x \big( e^{- \q \tau_b^+(r)} ;\tau_b^+(r)<T ^-_0(1) \big)
		% &= \E_x[e^{- \q \tau_b^+} ;\tau_b^+<T ^-_0(1)]
		= \E_x \big( e^{-q\tau_b^+(r)} ; R_r (\tau_b^+ (r)) =  0\big) \\= \lim_{\th \uparrow \infty}\E_x \big( e^{-q\tau_b^+(r)-\theta R_r(\tau_b^+(r))}; \tau_b^+(r) < \infty\big) % =\frac{Z_q(x,\Phi_{\q+\p})}{Z_q(b,\Phi_{\q+\p})}.
		 = \lim_{\th \uparrow \infty}\frac{Z_q(x,\Phi_{\q+\p})+\p Z_q(x,\th)/ (q-\k(\th))}{Z_q(b,\Phi_{\q+\p})+\p Z_q(b,\th)/(q-\k(\th))}.
		\end{multline*}
		It is now left to show that $Z_q(x,\th)/(q-\k(\th))$ vanishes in the limit as $\theta \uparrow \infty$.

Indeed, by \cite[(7)]{AIZ} (see also the identity (3.19) in \cite{APP} for an older version)
\begin{align*}
&&\E_x \big( e^{-\q \tau_0^- +\th X(\tau_0^-)} ; \tau_0^- < \infty\big) = Z_\q(x,\th) -  \fr {\k(\th)-\q}{ \th-\Phi_\q} W_\q(x).
\end{align*}
By taking $\theta \uparrow \infty$ on both sides and using Theorem 2.6 (ii) of \cite{KKR}, we have
		\begin{align*}%\label{limr2}
		\lim_{\th \uparrow \infty}
		\left(Z_{\q}(x,\th)-\frac{\k(\th)-\q}
		{\th-\Phi_{\q}}W_{\q}(x)\right)
		%&=\lim_{\p\to\infty}\Bigg[\E_x\left(e^{-\q\tau_0^- + \th X(\tau_0^-)};X(\tau_0^-)=0, \tau_0^-<\infty\right)\notag\\
		%&+\E_x\left(e^{-\q\tau_0^-}e^{\th X(\tau_0^-)}; X(\tau_0^-)<0,
		%\tau_0^-<\infty\right)\Bigg]
		= \E_x \big( e^{-\q \tau_0^-}; X(\tau_0^-) = 0, \tau_0^- < \infty\big)
		=\frac{\sigma^2}{2}\left[ W_q'(x)-\Phi_qW_q(x)\right].
		\end{align*}
		%Finally we use (\ref{limr2}) to obtain
		Hence,
		\begin{align}\label{limr3}
		\frac{Z_{\q}(x,\th)}{\k(\th)-q}=\frac{1}{\k(\th)-q}\left(Z_{\q}(x,\th)-\frac{\k(\th)-\q}
		{\th-\Phi_{\q}}W_{\q}(x)\right)+\frac{1}{\th-\Phi_{\q}}W_{\q}(x) \xrightarrow{\theta \uparrow \infty} 0.
		\end{align}
%		Hence
%		\begin{align}\label{limr4}
%		\lim_{\th\to\infty}&\frac{(q-\k(\th))Z_q(x,\Phi_{\q+\p})+\p Z_q(x,\th)}{(q-\k(\th))Z_q(b,\Phi_{\q+\p})+\p Z_q(b,\th)}\notag\\&=\lim_{\th\to\infty}\frac{\k(\th)^{-1}(q-\k(\th))Z_q(x,\Phi_{\q+\p})+\k(\th)^{-1} Z_q(x,\th)}{\k(\th)^{-1}(q-\k(\th))Z_q(b,\Phi_{\q+\p})+\k(\th)^{-1} Z_q(b,\th)}=\frac{Z_q(x,\Phi_{\q+\p})}{Z_q(b,\Phi_{\q+\p})}.
%		\end{align}
%%	\begin{remark} %\red{[I guess we can combine with Lemma \ref{lemma_Z_q_theta_limit}?]}
%		To finish the proof we note that
%		And the result follows.
		%Note that the case $q = 0$ appears in \cite{AIrisk}.
		%\red{[So this matches some results by other people?]}\blue{I am not sure Florin wanted to add this remark, at least I know (according to \cite{AIZ}) that the case $q=0$ appears in \cite{AIrisk}.}
		%\red{JL: So, we don't want to reference \eqref{limr4} here, which is placed later in the paper. So how about making this a lemma and move the proof to the appendix?}
	%\end{remark}}
\subsubsection{Proof of Corollary \ref{corollary_L_b}}

%\red{[moved here and combined the lemma and the remark]}

%\begin{remark}  \label{convergence_w}
%	
%
%	%\blue{[Kazu: Would it be relevant to write a little more in this Remark?]}
%\end{remark}

(i)
For $q > 0$, by applying dominated convergence in Theorem \ref{expected_bailouts}, upon taking $b \uparrow \infty$, it is immediate by Lemma \ref{lemma_limits_b}. For the case with $q = 0$ and $\Phi_0 > 0$, it holds by monotone convergence upon taking $q \downarrow 0$ using the convergence obtained in the proof of Corollary \ref{corollary_laplace_tau_a_b_infty} (i).

(ii) By applying monotone convergence in Theorem \ref{expected_bailouts}, upon taking $a \downarrow -\infty$, by Lemma \ref{convergence_summary},
\begin{align*}
\E_x\left( \int_0^{\t_b^+ (r)}e^{-qt}\diff R_{\p}(t)\right)=\frac {Z_{\q,\p}( x)}  {Z_{\q,\p}( b)} k_{\q,\p}(b) -k_{\q,\p}(x).  %=  \frac {Z_{\q,\p}( x)}  {Z_{\q,\p}( b)} \tilde{h}_{q,r}(b) -\tilde{h}_{q,r}(x).
\end{align*}
In particular, when $q > 0$, this reduces to $\frac {Z_{\q,\p}( x)}  {Z_{\q,\p}( b)} \tilde{h}_{q,r}(b) -\tilde{h}_{q,r}(x)$.

%
%
%Using this lemma, we shall take $a \downarrow -\infty$ in .
%
%	Using Lemma \ref{convergence_summary} (ii) and  (\ref{nl})
%	\begin{align}
%	\lim_{a\to-\infty} \mathcal{H}_{\q,\p}^{a}(x,0)%&=\lim_{a\to-\infty}\frac{\p}{\q+\p}\left(Z_{\q}(x)-Z_{\q,\p}^a(x)+\frac{Z_{\q+\p}(-a)}{W_{\q+\p}(-a)}W_{\q,\p}^a(x)\right)+\frac{\q}{\q+\p}\frac{W_{\q,\p}^a(x)}{W_{\q+\p}(-a)}\\
%	&=\frac{\p}{\q+\p} Z_{\q}(x)+\frac{\q}{\q+\p} Z_{\q}(x,\Phi_{\q+\p})=Z_{\q,\p}(x). \label{big_H_convergence_0}
%	\end{align}
%
%Combining with Lemma \ref{convergence_summary} (iii) and after some simplification, we have the result.
%	Finally a simple computation shows that
%	\begin{align*}
%	 \red{\lim_{a \downarrow -\infty} f(x,a,b) =} L(x)-Z_{\q,\p}(x)Z_{\q,\p}(b)^{-1}L(b) =- \Big( h(x)-Z_{\q,\p}(x)Z_{\q,\p}(b)^{-1}h(b) \Big).
%	\end{align*}
%	Hence, using monotone convergence we have (i).

	(iii) 	%\red{[moved these here]
	Recall the convergence \eqref{limits}.
By  \eqref{scale_function_laplace} and \eqref{W_q_limit},
\begin{align*}
\frac{Z_{\q}(b,\Phi_{\q+\p})}
		{W_{\q}(b)}%= \frac {e^{\Phi_{q+r} b} [1 - r \int_0^b e^{-\Phi_{q+r} z} W_q (z) \diff z ]} {W_q(b)} \\ 
		= \frac {r \int_0^\infty e^{-\Phi_{q+r}z} W_q(z+b) \diff z} {W_q(b)} \xrightarrow{b \uparrow \infty} r \int_0^\infty e^{-(\Phi_{q+r}- \Phi_q) z} \diff z = \frac{\p}{\Phi_{\q+\p}-\Phi_{\q}}.
\end{align*}
Hence, upon taking $b \uparrow \infty$ in (ii), monotone convergence completes the proof.

 \subsection{Proofs of Corollaries in Section \ref{subsection_reflected_case}}

Again, we first summarize the limits needed for the proofs.
%\green{[JL: I guess this holds for $q \geq 0$.]}
\begin{lemma} \label{limit_a_derivative} Fix $q \geq 0$. %\green{[swapped (i) and (ii)]}
For all $x \in \R$, as $a \downarrow -\infty$,
(i) $ \mathcal{I}_{\q,\p}^{a \prime}(x) \rightarrow 0$, \quad
(ii) $\mathcal{H}^{a \prime}_{q,r}(x,0) \rightarrow Z'_{\q,\p}(x)$, \quad
(iii) $h_{\q,\p}^{a \prime}(x) \rightarrow k_{\q,\p}'(x) = \frac{r}{q+r}Z_{q}(x)- \frac{r}{(q+r)^2} \k'(0+) \Phi_{q+r} Z_q (x, \Phi_{q+r})$.
\end{lemma}
\begin{proof}
(i), (ii) By \eqref{h_0_reflected_2}, we have that $\widehat{h}(0,a,b,0)=- \mathcal{I}^{a \prime}_{q,r}(b) /\mathcal{H}^{a \prime}_{q,r}(b,0)$.
This together with \eqref{H_big_zero} gives
\begin{align}
\widehat{h}(0,a,b,0) {\frac q {q+r}   \Big(\frac{(W^{a }_{\q,\p})'(b)}{W_{q+r}(-a)}+r W_q(b) \Big)  } = {\mathcal{I}^{a \prime}_{q,r}(b)} \Big( \frac r {q+r} \widehat{h}(0,a,b,0) - 1 \Big).\label{h_hat_relation}
\end{align}
Now note that $\widehat{h}(0,a,b,0)\to0$ as $a\downarrow -\infty$. On the other hand by
	\eqref{W_a_def} and an integration by parts
\begin{align*}
(W^{a }_{\q,\p})'(b)  =W_{q+r}'((b-a)+)-rW_q(b)W_{q+r}(-a)-r\int_0^{-a}W_{q}(b-y)W_{q+r}'(y-a) \diff y,
\end{align*}
we have
\begin{align}
\lim_{a \downarrow -\infty}\frac{q}{q+r}\frac{(W^{a }_{\q,\p})'(b)}{W_{q+r}(-a)}+\frac{qr}{q+r}W_q(b)=Z_{q,r}'(b). \label{W_prime_ratio_limit}
\end{align}
Hence, upon taking limits as $a \downarrow -\infty$ in \eqref{h_hat_relation},
we obtain (i). In addition, (ii) holds as well by \eqref{H_big_zero} and \eqref{W_prime_ratio_limit}.

(iii) By (3.10), (3.18), and (3.16) in \cite{APP}, we have %\red{[deleted the indicator $\eta_0^- < \infty$ because it holds a.s.]}
\begin{align} \label{identities_AAP}
\begin{split}
\E_b\left(e^{-q\eta_0^-}\right)&=Z_q(b)-q\frac{W_q(b)^2}{W_q'(b+)},\\
\E_b\left(e^{-q\eta_0^-+\Phi_{q+r}Y^b(\eta_0^-)} \right)&=Z_q(b,\Phi_{q+r})-\frac{W_q(b)}{W_q'(b+)}Z_q'(b,\Phi_{q+r}),\\
\E_b\left(e^{-q\eta_0^-}Y^b(\eta_0^-) \right)&=\overline{Z}_q(b)-\k'(0+)\overline{W}_q(b)-\frac{W_q(b)}{W_q'(b+)}[Z_q(b)-\k'(0+)W_q(b)],
\end{split}
\end{align}
where $Z_q'(b,\Phi_{q+r})$ is the partial derivative with respect to $b$ given by $Z_q'(b,\Phi_{q+r}) = \Phi_{q+r} Z_q (b, \Phi_{q+r}) - r W_q(x)$.

%\green{[changed the argument here]
In addition, by \eqref{U_0_summary}, Lemma \ref{convergence_summary} (iii) applied to Corollary \ref{corollary_big_U}, for $x < 0$,
\begin{align*}
\E_{x}\left(e^{-q\mathbf{e}_{\p}}X(\mathbf{e}_{\p});\mathbf{e}_{\p}<\tau_0^+\right) = \lim_{a \downarrow -\infty} \lim_{b \uparrow \infty}U_4(x, a,b) &= - \lim_{a \downarrow -\infty} U_4^0(a,x) = k_{q,r} (x) = r\frac {(\q+\p)x +(1-e^{\Phi_{\q+\p} x})\k'(0+)} {(q+r)^2}.
\end{align*}
%\green{[combined the two equations. OK?]}
By this, \eqref{identities_AAP}, and  monotone convergence, it follows that
\begin{align*}
\lim_{a\downarrow -\infty}\widehat{U}_4(b,a,b)%&=\lim_{a\to-\infty}\E_b\Big(e^{-q\eta_0^-}\E_{Y^b(\eta_0^-)}\left(e^{-q\mathbf{e}_{\p}}X(\mathbf{e}_{\p});\mathbf{e}_{\p}<\tau_0^+\wedge\tau_a^-\right)\Big)\\
&=\E_b\Big[e^{-q\eta_0^-}\E_{Y^b(\eta_0^-)}\left(e^{-q\mathbf{e}_{\p}}X(\mathbf{e}_{\p});\mathbf{e}_{\p}<\tau_0^+\right) \Big] \\
%=k_{\q,\p}(b)-\frac{W_q(b)}{W_q'(b+)}k_{\q,\p}'(b).
%\end{align*}
%Hence,
%	\begin{align*} %\label{overshoot_expectation_reflected}
%	\begin{split}
%\E_b\Big(e^{-q\eta_0^-}\E_{Y^b(\eta_0^-)}\left(e^{-q\mathbf{e}_{\p}}X(\mathbf{e}_{\p});\mathbf{e}_{\p}<\tau_0^+\right)\Big)
	&=\frac{r}{(q+r)^2} \E_b\left[e^{-\q\eta_0^-}\left((\q+\p) Y^b(\eta_0^-)+(1-e^{\Phi_{\q+\p} Y^b(\eta_0^-)})\k'(0+)\right) \right] \\
	%&=\frac{r}{(r+q)^2}\left(Z_{1,q,r}(b)-\frac{W_q(b)}{W_q'(b)}Z_{1,q,r}'(b)\right) \\
	&=k_{\q,\p}(b)-\frac{W_q(b)}{W_q'(b+)}k_{\q,\p}'(b).
	\end{align*}
%	\blue{where
	%\red{this $k$ is probably the same as $h$?}\blue{ Hi Kazu: I guess we are not using $h$ anymore?} \red{[I see $h$ in the appendix a bit.  We can just stick to $k$ and maybe change $\tilde{h}$ to $h$?]}
	%\blue{Ok Agreed.}
%	\begin{align*}
%	k(x)&:=\overline{Z}_{q}(x)-\k'(0+)\overline{W}_q(x) -\frac{\k'(0+)}{\q+\p}(Z_{\q}(x,\Phi_{\q+\p})-Z_{\q}(x)), \\
%	k'(x)&=Z_{q}(x)-\frac{\k'(0+)}{\q+\p}\Phi_{q+r} Z_q (x, \Phi_{q+r}).
%	\end{align*}}
Therefore using
%\eqref{rel_h_hprime}
\eqref{H_H_derivative_relation} and Lemma \ref{convergence_summary} (iii), we obtain
\begin{align*}
\lim_{a \downarrow -\infty}h_{q,r}^{a \prime}(b)=\lim_{a \downarrow -\infty}\frac{W_q'(b+)}{W_q(b)}\left( h_{\q,\p}^{a}(b)- \widehat{U}_4(b,a,b)\right)=k_{\q,\p}'(b).
\end{align*}

\end{proof}
	
	 \subsubsection{Proof of Corollary \ref{corollary_R_b_r}}
By monotone convergence, the result is immediate upon taking $a \downarrow -\infty$ in Theorem \ref{theorem_R_b_r} by Lemmas \ref{convergence_summary} and \ref{limit_a_derivative}.	
	
%We shall apply monotone convergence in (i).  Recall as in \eqref{big_H_convergence_0}, $\mathcal{H}^a_{q,r}(x,0) \xrightarrow{a \downarrow -\infty} Z_{q,r} (x)$. %\red{need to comfirm}.
%
%Hence it is sufficient to show the convergence
%\begin{align}
%\lim_{a\to-\infty} \mathcal{H}^{a \prime}_{q,r}(b,0)=Z_{q,r}'(b). \label{convergence_H_prime}
%\end{align}

%It is not difficult to see that
%\begin{align*}
%\mathcal{H}^{a \prime}_{q,r}(b,0)=\frac{q}{q+r}\frac{\blue{(W^{a }_{\q,\p})'}(b)}{W_{q+r}(-a)}+\frac{qr}{q+r}W_q(b)-\frac{r}{r+q}\mathcal{I}^{a \prime}_{q,r}(b).
%\end{align*}

\subsubsection{Proof of Corollary \ref{corollary_L_b_r}}

%\blue{[Changed the proof to a more direct approach by taking limits.]}
%Using \eqref{lim_a_h} we have that $\lim_{a\to-\infty}h_{q,r}^a(x)=k(x)$.
%First recall the convergence \eqref{lim_a_h}.

By monotone convergence applied to  Theorem \ref{theorem_L_b_r} and by Lemmas \ref{convergence_summary} and \ref{limit_a_derivative}, for $x \leq b$,
%
%\ref{theorem_L_b_r}
%
%So using \eqref{H_limit_-a} and \eqref{convergence_H_prime} we  obtain that
%\begin{align*}
%\lim_{a\to-\infty}\frac r {q+r} \Big(\frac{\mathcal{H}_{\q,\p}^{a}(x, 0)}{\mathcal{H}_{\q,\p}^{a}(b, 0)}h_{\q,\p}^{a \prime}(b) -h_{\q,\p}^{a}(x)\Big)=\frac r {q+r} \Big(\frac{Z_{q,r}(x)}{Z_{\q,\p}'(b)}k'(b) -k(x)\Big).
%	\end{align*}
%\blue{Now we note that
%\red{[simplified this way]
\begin{align*}
\E_x\left( \int_{0}^{\infty}e^{-qt} \diff R^b_r(t)\right) = \frac{Z_{q,r}(x)}{Z_{q,r}'(b)} k_{\q,\p}'(b) - k_{\q,\p}(x)
%&= \frac{Z_{q,r}(x)}{q \Phi_{q+r} Z_q(b, \Phi_{q+r}) / (q+r)} \Big( Z_{q}(b)-\frac{\k'(0+)}{\q+\p}\Phi_{q+r} Z_q (b, \Phi_{q+r}) \Big) - k(x) \\
%&= \frac{Z_{q,r}(x)}{q \Phi_{q+r} Z_q(b, \Phi_{q+r}) / (q+r)} Z_{q}(b)-\frac{\k'(0+)} q Z_{q,r}(x) - k(x)\\
&=\frac{r}{q+r}\frac{Z_{q,r}(x)}{Z_{q,r}'(b)}Z_q(b) - \tilde{h}_{q,r}(x).
\end{align*}
%And the result follows.
%\subsection{Proof of Corollary \ref{corollary_laplace_tau_a_b_infty} and identity \eqref{ParBail_a_b_infty}\blue{[added this.]}}.
	
%	\subsection{Proof of Lemma \ref{lemma_u}}

\section*{Acknowledgements}
We are grateful to Xiaowen Zhou for useful suggestions, including
an alternative approximation approach, which may provide a different solution of the unbounded variation case. 
We thank as well the anonymous referees for careful reading and  helpful
suggestions. K. Yamazaki is in part supported by MEXT KAKENHI grant no. 26800092.
\bibliographystyle{abbrv}
\bibliography{Pare31}
\end{document}